\documentclass{article}
\usepackage{svg}
\usepackage[OT4]{fontenc}
\usepackage{amsmath}
\usepackage{centernot}
\usepackage{multirow}
\usepackage{pdflscape}
\usepackage{amsfonts}
\usepackage{subfigure}
\usepackage{booktabs}
\usepackage[linesnumbered, ruled,vlined]{algorithm2e}
\SetArgSty{textup}
\usepackage{rotating}
\usepackage{amsthm}
\usepackage{hyperref}
\usepackage{mathtools}
\usepackage{amssymb}

\usepackage{breakurl}
\usepackage[english]{babel}
\usepackage[utf8]{inputenc}
\usepackage{graphicx}
\usepackage{xcolor}
\usepackage{todonotes}

\usepackage{mathrsfs}
\usepackage[paper=a4paper,margin=1in]{geometry}

\usepackage[nottoc]{tocbibind}
\usepackage[numbers]{natbib}
\usepackage[nameinlink]{cleveref}
\Crefformat{footnote}{#2\footnotemark[#1]#3}
\usepackage{autonum}
\usepackage[normalem]{ulem}
\usepackage{caption}
\usepackage{pgfplots}

\DeclareMathOperator{\Tr}{Tr}
\DeclareMathOperator{\GenSet}{\mathcal{F}}

\DeclareMathOperator{\sign}{\text{sgn}}

\newtheorem{theorem}{Theorem}

\newtheorem{lemma}{Lemma}

\newtheorem{example}{Example}
\newtheorem{remark}{Remark}
\title{SDP SAT approaches}
\author{Lennart Sinjorgo}
\date{\vspace{-5ex}}

\theoremstyle{definition}

\makeatletter

\newcommand{\@citeX}[3]{{
  \@ifundefined{b@#2}{\mbox{\reset@font\bfseries ?}
    \G@refundefinedtrue
    \@latex@warning
      {Citation `#2' on page \thepage \space undefined}}
  {#1\csname b@#2\endcsname#3}}}
\newcommand{\citeA}[1]{\@citeX{[Ref.~}{#1}{]}}
\newcommand{\citeB}[1]{\@citeX{Reference~}{#1}{}}
\makeatother

\newcommand{\Logical}{\phi}
\newcommand{\SOSms}{SOS-MS}
\newcommand{\gitHubLink}{\url{https://github.com/LMSinjorgo/SOS-SDP_MAXSAT}}

\newcommand{\true}{\texttt{true}}
\newcommand{\false}{\texttt{false}}

\DeclareMathOperator*{\argmin}{argmin}

\newcommand{\heurMaxSolve}{CCLS}

\begin{document}

\title{On solving the MAX-SAT using sum of squares}
\author{Lennart Sinjorgo \thanks{Tilburg University, Department of Econometrics and OR, CentER,  The Netherlands, {\tt l.m.sinjorgo@tilburguniversity.edu}}
	\and {Renata Sotirov}  \thanks{Tilburg University, Department of Econometrics and OR, CentER,  The Netherlands, {\tt r.sotirov@tilburguniversity.edu}}}

\maketitle

\begin{abstract}
We consider semidefinite programming (SDP) approaches for solving  the 
maximum satisfiability problem (MAX-SAT) and the weighted partial MAX-SAT. It is widely known that SDP is well-suited to approximate the (MAX-)2-SAT.
Our work shows the potential of SDP also for other satisfiability problems, by being competitive with some of the best solvers in the yearly MAX-SAT competition. 
Our solver combines  sum of squares (SOS) based SDP  bounds and  an efficient parser within
a branch \&  bound scheme.
 
On the theoretical side, we  propose a family of semidefinite feasibility problems, and show that a member of this family provides the rank two guarantee. We also provide a parametric family of semidefinite relaxations for the MAX-SAT, and derive several properties of monomial bases used in the SOS approach. We connect two well-known SDP approaches for the (MAX)-SAT, in an elegant way. Moreover, we relate our SOS-SDP relaxations for the partial MAX-SAT  to the known SAT relaxations.

\end{abstract}

{\bf Keywords}   SAT, MAX-SAT, weighted partial MAX-SAT, 
 semidefinite programming, sum of squares, Peaceman-Rachford splitting method\\

{\bf AMS subject classifications.}  90C09, 90C22, 90C23.

\section{Introduction}

In this paper, we investigate semidefinite programming (SDP) approaches for the satisfiability problem, (SAT),   maximum satisfiability problem (MAX-SAT) and their variants. Given a logical proposition, built from a conjunction of clauses, the SAT asks whether there exists a truth assignment to the variables such that all clauses are satisfied. 
The optimization variant of SAT, known as the MAX-SAT, is to determine a truth assignment which satisfies the largest number of clauses. 

The SAT is a central problem in mathematical logic and computer science and finds various applications, including software or hardware verification \cite{marques2000boolean} and planning in artificial intelligence \cite{kautz1999unifying}. \citeauthor{cook1971complexity}~\cite{cook1971complexity} proved that SAT  is  $\mathcal{NP}$-complete in~\citeyear{cook1971complexity}.  The SAT was the first problem proven to be $\mathcal{NP}$-complete, which implies that any problem contained in the complexity class $\mathcal{NP}$ can be efficiently recast as a SAT instance. Thus, algorithms for the SAT can also solve a wide variety of other problems, such as timetabling \cite{asin2014curriculum, gattermann2016integrating} and product line engineering \cite{mendonca2009sat}. 

SDP approaches to the SAT  are  proposed first by \citeauthor{de2000relaxations}~\cite{de2000relaxations}, and later extended by \citeauthor{anjos2004proofs}~\cite{anjos2004proofs}--\cite{anjos2008extended}.
\citeauthor{goemans1995improved}~\cite{goemans1995improved} were  first to apply SDP to the MAX-SAT. 
They showed that for a specific class of MAX-SAT instances, known as MAX-2-SAT (in the MAX-$k$-SAT, each clause is a disjunction of at most $k$ variables), the MAX-SAT is equivalent to optimizing a multivariate quadratic polynomial, which is naturally well suited for semidefinite relaxations. 
In the same paper, \citeauthor{goemans1995improved}
proposed a 0.878-approximation algorithm for the MAX-2-SAT based on SDP. This result was later improved to 0.940 in \cite{lewin2002improved}. Further, \citeauthor{karloff19977}~\cite{karloff19977} obtained an optimal $7/8$ approximation algorithm for  the MAX-3-SAT, and \citeauthor{halperin2001approximation}~\cite{halperin2001approximation} a nearly optimal approximation algorithm for the MAX-4-SAT. 
In \cite{van2008sums},  \citeauthor{van2008sums}~exploit sum of squares (SOS) optimization  to compute bounds for the MAX-SAT.

Despite the great success in designing approximation algorithms using SDP, most modern MAX-SAT solvers do not exploit SDP. A possible reason for this is the fact that medium to large size SDP problems are computationally challenging to solve. Interior point methods, the conventional approach for solving SDPs, struggle from large memory requirements and prohibitive computation time per iteration already for medium size SDPs. Recently, first order methods such as the alternating direction method of multipliers (ADMM) \cite{Boyd,Gaby} and the Peaceman-Rachford splitting method (PRSM)~\cite{peaceman1955numerical} showed a great success in solving SDPs, see e.g.,~\cite{de2021sdp,graham2022restricted,oliveira2018admm}. 
Motivated by those results, we design a MAX-SAT solver that  incorporates SDP bounds and  the PRSM within a branch \& bound (B\&B) scheme.\\

In particular, we further exploit the SOS approach from~\cite{van2008sums} to derive SOS-based SDP relaxations that provide strong upper bounds to the optimal MAX-SAT solution. The derived  SDP relaxations are  strengthened SDP duals of the \citeauthor{goemans1995improved} MAX-SAT relaxation.
The strength of the upper bounds and the required time to compute the  relaxations depend on the chosen monomial basis. 
We experiment with different monomial bases and
propose a class of bases that provide good trade-offs between these effects. Moreover, we  derive several properties of monomial bases that are exploited in the design of our  solver.
We extend the SOS approach to the weighted partial MAX-SAT, a variant of the MAX-SAT in which clauses are divided in soft and hard clauses. Here, the goal is to maximize the weighted sum of  soft clauses, while satisfying all the hard clauses. We strengthen SDP bounds for the weighted partial MAX-SAT using the SAT resolution rule. To the best of our knowledge, we are the first to exploit SDP for solving the weighted partial MAX-SAT.

We show that the Peaceman-Rachford splitting method is well suited for exploiting the structure of the SOS-based SDP relaxations. Therefore, we implement the PRSM to (approximately) solve  large-scale SDP relaxations and obtain  upper bounds for the (weighted partial) MAX-SAT. 
The resulting algorithm is very efficient, e.g., it can compute upper bounds  with matrix variables of order 1800 in less than 2 minutes, and for matrices of order 2400 in less than 4 minutes. Our numerical results show that the upper  bounds are strong, in particular when larger monomial bases is used. 
We also exploit the output of the PRSM to efficiently compute lower bounds for the   MAX-SAT. 

We design an SOS-SDP based MAX-SAT solver  (named \SOSms{}) that exploits SOS-based SDP relaxations and the PRSM.  \SOSms{}  is one of the first SDP-based MAX-SAT solvers. The only alternative SDP-based solver is the MIXSAT algorithm~\cite{wang2019low} that is designed to solve  MAX-2-SAT instances.
 \SOSms{} is able to solve (weighted partial) MAX-$k$-SAT instances, for $k\leq 3$.
To solve a MAX-SAT instance,  \SOSms{} has to approximately 
solve multiple SDP subproblems. A crucial component of  \SOSms{} is therefore its ability to quickly construct the programme parameters of the required SDPs, i.e., the process of \textit{parsing}. We design an efficient parsing method, which is also applicable to other problems and publicly available.
Another efficient feature of our solver is warm starts.
Namely,  our solver uses the approximate PRSM solution  at a node, as warm starts for the corresponding children’s node.
We are able to solve a variety of MAX-SAT instances in a reasonable time, while solving some instances faster than the best solvers in the  Eleventh Evaluation of Max-SAT Solvers (MSE-2016). Moreover, we solve three previously unsolved MAX-3-SAT instances from  the MSE-2016. We are first to use SDP for solving weighted partial MAX-SAT.
Our results provide new perspectives on solving the MAX-SAT, and all its variants, by using SDPs.

Our paper provides  various theoretical results. 
We propose a family of semidefinite feasibility problems, and show that one member of this family provides the rank two guarantee. That is, whenever the semidefinite relaxation admits a feasible matrix of rank two or less, the underlying SAT instance is satisfiable. This result relates to a similar  rank two guarantee result  by \citeauthor{anjos2004semidefinite}~\cite{anjos2004semidefinite}. 
The rank value can be seen as a measure of the strength of the relaxation.
We also provide a parametric family of semidefinite relaxations for the (weighted partial) MAX-SAT. This parameter can be finely tuned to determine the strength of the relaxation, and any such relaxation can easily be incorporated within \SOSms{}. This allows the solver to be adapted per (class of) problem instances.

Further, we show how the SOS approach to the MAX-SAT of~ \citeauthor{van2008sums}~\cite{van2008sums} and here generalized moment relaxations of the SAT due to 
 \citeauthor{anjos2004proofs} \cite{anjos2004proofs}-\cite{anjos2006semidefinite} are related.  This is done by exploiting the duality theory of the moment  and SOS  approaches.  Our result generalizes a result by \citeauthor{van2008sums}~\cite{van2008sums}, who showed  the connection between the two approaches only for restricted cases. By exploiting duality theory, we also relate the SOS relaxations for the partial MAX-SAT  problem  to the SAT relaxations from~\cite{anjos2004semidefinite}.

Lastly, we investigate MAX-SAT resolution, a powerful technique used by many MAX-SAT solvers~\cite{abrame2014ahmaxsat}, in relation to the SDP approach to the MAX-SAT. Standard MAX-SAT solvers use resolution to determine upper bounds on the MAX-SAT solution, while \SOSms{} determines upper bounds through SDP.  
We show how resolution is related to the monomial basis. We also show how  the SAT resolution  can be exploited for the weighted partial MAX-SAT. 
\\

This paper is organized as follows. We provide notation and preliminaries in \Cref{section_Notation} and assumptions in \Cref{subsection_Assumptions}. \Cref{section_maxSatFormulations} provides an overview of the \citeauthor{goemans1995improved} approach \cite{goemans1995improved} to the MAX-SAT. \Cref{section_SATasSDPfeas} first outlines previous SDP approaches to the SAT, and then generalizes them.
\Cref{section_SOSandMaxSat} provides the details of the SOS theory, applied to the MAX-SAT. We also derive various properties of monomial bases in that section. \Cref{section_Resolution} concerns the combination of MAX-SAT resolution and SOS. In \Cref{section_dualRelation}, we show how two SDP approaches to (MAX-)SAT, i.e., \cite{anjos2004semidefinite} and \cite{van2008sums},  are connected. \Cref{section_PRSMsection} introduces the PRSM for SOS. We extend the SOS approach to the weighted partial MAX-SAT and connect the resulting programme to the relaxations in \cite{anjos2004semidefinite}, in \Cref{section_wpMS}. \Cref{section_algorithmDescr} provides an overview and pseudocode of our solver \SOSms{}. \Cref{section_NumericalResults} presents numerical results that include SOS-SDP bounds and performance of \SOSms{}. Concluding remarks are given in~\Cref{section_Conclusion}.

\subsection{Preliminaries and notation}
\label{section_Notation}
For any $n \in \mathbb{N}$, we write $[n] = \{1, \ldots, n \}$. We denote by $\Logical$ a propositional formula, in variables $x_1$ up to $x_n$, and assume that $\Logical$ is in \textit{conjunctive normal form} (CNF). That is, $\Logical$ is given by a conjunction of $m$ clauses,
\begin{align}
    \label{eqn_definitionCNF}
    \Logical = \bigwedge_{j = 1}^m C_j.
\end{align}
We will mostly use $n$ to refer to the number of variables and $m$ to refer to the number of clauses. Each clause $C_j$ is a disjunction of (possibly negated) variables. We define each clause $C_j$ as a subset of $[n]$, indicating the variables appearing in $C_j$. Moreover, we define $I^+_j \subseteq C_j$ as the set of unnegated variables appearing in $C_j$. Similarly, $I^-_j \subseteq C_j$ is defined as the set of negated variables appearing in $C_j$. Thus, the clause associated with $C_j$ reads
\begin{align}
    \label{eqn_IplusDef}
    \bigvee_{i \in I^+_j} x_i \vee \bigvee_{i \in I^-_j}  \neg x_i.
\end{align}
We refer to both $x_i$ and $\neg x_i$ as literals. For example, the literal $\neg x_i$ is true if $x_i$ is \texttt{false}. We denote the length of a clause by $\ell_j$, thus $\ell_j := | C_j |$. We say that $\Logical$ constitutes a (MAX-)$k$-SAT instance if $\max_{j \in [m]} \ell_j = k$.

The SAT is to decide, given $\Logical$, whether a satisfying truth assignment to the variables $x_i$, $i \in [n]$ exists. The MAX-SAT is to find an assignment which satisfies the largest number of clauses.

We associate to each clause a vector $a_j \in \{0, \pm 1\}^n$, having entries $a_{j,i}$ according to 
\begin{align}
    \label{def_ajVector}
    a_{j,i} = \begin{cases}
    -1,  & \text{ if } i \in {I^-_j}, \\
    0,  & \text{ if } i \notin I^+_j \cup I^-_j, \\
    1, & \text{ if } i \in I^+_j.
    \end{cases}
\end{align}
We write $\mathcal{S}^n$ for the set of symmetric $n \times n$ matrices and $\mathcal{S}^n_+$ for the cone of symmetric positive semidefinite matrices of size $n \times n$. If the context is clear, the superscript $n$ will be omitted. We denote by $\mathbf{1}_n \in \mathbb{R}^n$ and $\mathbf{0}_n \in \mathbb{R}^n$, vectors of all ones and zeroes, respectively (subscripts are omitted when the context is clear). The identity matrix $I_n \in \mathcal{S}^n$ has as columns the unit vectors $e_i$, $i \in [n]$. Matrix $\mathbf{0}_{m \times n}$ denotes the zero matrix of $m$ rows and $n$ columns.

For $X,Y \in \mathcal{S}$, we define the trace inner product as $\langle X,Y \rangle :=  \Tr(XY)$. We write $X \succeq Y$ if and only if $X - Y \in \mathcal{S}_+$. The Frobenius norm of a symmetric matrix $X$ is given by
\begin{align}
    \label{eqn_FrobeniusNorm}
    \| X \|_F = \sqrt{\langle X, X \rangle}.
\end{align}
For $X \in \mathcal{S}$ and $\mathcal{M} \subseteq \mathcal{S}$, we set
\begin{align}
    \label{eqn_ProjectionOperatorDef}
    \mathcal{P}_{\mathcal{M}}(X) := \argmin_{Z \in \mathcal{M} } \| Z-X \|_F,
\end{align}
as the projection of $X$ onto $\mathcal{M}$.

For $X \in \mathbb{R}^{n  \times n}$, diag$(X) \in \mathbb{R}^n$ denotes the vector equal to the diagonal of $X$. For $x \in \mathbb{R}^n$, we write $\| x \| := \sqrt{x^\top x}$ and
\begin{align}
    \label{eqn_xAlphaDef}
    x^\alpha := \prod_{i \in \alpha} x_i,
\end{align}
for some $\alpha \subseteq [n]$. We evaluate the empty product as 1. Note that \eqref{eqn_xAlphaDef} differs from the more common notation in SOS literature, where the $\alpha$ are considered as vectors in $\mathbb{N}^n$, see e.g., \cite{lasserre2007sum}.

\subsection{Assumptions}
\label{subsection_Assumptions}
In the rest of this work, we assume that all logical propositions $\Logical$, on $n$ variables and $m$ clauses, satisfy the following three properties:
\begin{enumerate}
    \item $I^+_j  \cap I^-_j = \emptyset$, $\forall j \in [m]$,
    \item $|C_j| \geq 2$, $\forall j \in [m]$,
    \item Each variable is contained it at least 2 clauses,
\end{enumerate}
along with $\Logical$ being in CNF. We explain now that properties 1 and 3 can be assumed without loss of generality. Property 1 states that a clause cannot contain both the unnegated and negated variants of a variable. If that were not the case, the clause is trivially satisfied and can be removed from $\Logical$ without loss of generality. For property 3, if we have that a variable occurs in exactly one clause, say $C_j$, we can set that variable to the truth value such that $C_j$ is satisfied and remove $C_j$ from $\Logical$. 

Property 2 can be assumed for SAT instances. If a SAT instance contains a clause $C_j$ with $|C_j| = 1$ (such a clause is known as a \textit{unit clause}), the literal in $C_j$ must be satisfied. The variable corresponding to this literal can thus be given an appropriate truth value and $\Logical$ can be reduced (such a reduction of $\Logical$ is referred to as \textit{unit resolution}). For MAX-SAT instances, it is possible that an optimal truth assignment might leave unit clauses unsatisfied. We note however, that the MAX-SAT benchmark instances we consider satisfy all the above assumptions.

\section{The MAX-SAT formulation and relaxation}
\label{section_maxSatFormulations}

We outline the approach of Goemans-Williamson for formulating the MAX-SAT as a polynomial optimization problem. We also present their SDP relaxation for the MAX-2-SAT.

 Let $x_1, \, x_2, \, \ldots, x_n$ be the variables of the MAX-SAT instance, given by a logical proposition $\Logical$. We thus assume that $\Logical$ is given in conjunctive normal form, see \eqref{eqn_definitionCNF}, and contains $m$ clauses. As customary in the SDP SAT literature, we associate $+1$ with \texttt{true} and $-1$ with \texttt{false}. An assignment of the $x_i$ values in $\{ \pm 1 \}$ is referred to as a truth assignment. As proposed by \citeauthor{goemans1995improved} \cite{goemans1995improved}, we define a truth function $v: \{ \pm 1 \}^n \to \{0, 1 \}$, such that, given a logical proposition $\Logical^\prime$, evaluated for some truth assignment, $v( \Logical^\prime) = 1$ if and only if $\Logical^\prime$ is satisfied, and 0 otherwise. This property uniquely determines $v$, i.e.,
\begin{align}
    \label{eqn_truthFunction}
    v(x_i) = \frac{1 + x_i}{2} \text{ and } v(\neg x_i) = \frac{1 - x_i}{2}.
\end{align}
And in general, for a clause $C_j \subseteq [n]$ of length $\ell_j$, we have
\begin{align}
\label{eqn_valueProp}
v(C_j) := 1- \prod_{i \in I^+_j} v\left(\neg{x}_{i}\right) \prod_{i \in I^-_j} v\left({x}_{j}\right) = 1 - 2^{-\ell_j} \Big(   \sum_{\gamma \subseteq C_j} (-1)^{|\gamma|} a_j^\gamma  x^\gamma \Big),
\end{align}
for $a_j$ as in \eqref{def_ajVector}, and both $a_j^\gamma$ and $x^\gamma$ as in \eqref{eqn_xAlphaDef}. The last equality in \eqref{eqn_valueProp} follows from the product expansion of $v(C_j)$, as shown in Proposition 1 in \cite{anjos2006semidefinite}. In \cite{goemans1995improved}, an extra variable $x_0 \in \{ \pm 1\}$ is defined, with the purpose of deciding the truth value, that is, $\phi^\prime$ is true if and only if $v(\phi^\prime) = x_0$. We set $x_0 = 1$ without loss of generality for sake of clarity. The \mbox{MAX-SAT}, given by $\Logical$, is to  maximize the following polynomial:
\begin{align}
\label{eqn_vLogicalDef}
v_\Logical := \sum_{j \in [m]}  v(C_j) = \sum_{\alpha \subseteq [n]} v^\alpha_\Logical x^\alpha,
\end{align}
subject to $x_i\in \{ \pm 1\}$ for all $i$, and for appropriate $v^\alpha_\Logical \in \mathbb{R}$, and $x^\alpha$ as in \eqref{eqn_xAlphaDef}. Observe that $v_\Logical$ is a $k$th degree polynomial if $\Logical$ represents a MAX-$k$-SAT instance. A MAX-2-SAT instance thus corresponds to a quadratic polynomial, and is therefore well suited for approximated by SDP. We return to $v_\Logical$ in \Cref{section_dualRelation}. 

Assuming now that $\Logical$ represents a MAX-2-SAT instance on $n$ variables, the corresponding MAX-2-SAT can be formulated as
\begin{equation}
\label{eqn_MaxSatFormulation}
\begin{aligned}
\max \quad &  \langle W, X \rangle \\
\textrm{s.t.} \quad & 
 \text{diag}(X) = \mathbf{1}, \\
& X\succeq 0, \,
X \in \{\pm 1 \}^{(n+1) \times (n+1)},
\end{aligned}
\end{equation}
where $\langle W, X \rangle$ describes the quadratic polynomial $v_\Logical$. Observe that the constraints of \eqref{eqn_MaxSatFormulation} enforce $X$ to satisfy $X = xx^\top$, for some $x \in \{ \pm 1 \}^{n+1}$. The size of this vector $x$ is one more than the number of variables $n$, to account for the additional truth value variable $x_0$.

A semidefinite relaxation of \eqref{eqn_MaxSatFormulation} is obtained by omitting the integrality constraint, or equivalently nonconvex rank one constraint. This constitutes the well-known Goemans-Williamson \cite{goemans1995improved} SDP relaxation of the MAX-2-SAT. That is,
\begin{equation}
\begin{aligned}
\label{eqn_GoemansWilliamson}
\max \quad &  \langle W, X \rangle \\
\textrm{s.t.} \quad & \text{diag}(X) = \mathbf{1}, \\
& X \in \mathcal{S}^{n+1}_+.
\end{aligned}
\end{equation}
\citeauthor{goemans1995improved} showed that the optimal matrix for \eqref{eqn_GoemansWilliamson} can be used to obtain a 0.878-approximation algorithm to the MAX-2-SAT. Assuming $\mathcal{P} \neq \mathcal{NP}$, for any $\varepsilon > 0$ there exists no $(\frac{21}{22}+\varepsilon)$ approximation algorithm to the MAX-2-SAT \cite{haastad2001some}.
\citeauthor{karloff19977}~\cite{karloff19977} introduce a canonical way of obtaining SDP relaxations for any MAX-SAT, that is exploited to obtain approximation algorithms to the MAX-3-SAT and the MAX-4-SAT in \cite{karloff19977} and \cite{halperin2001approximation}, respectively. To solve MAX-2-SAT instances, rather than approximate, \citeauthor{wang2019low} \cite{wang2019low} propose the MIXSAT algorithm, which combines \eqref{eqn_GoemansWilliamson} with a B\&B scheme.

\section{The SAT as semidefinite feasibility problem}
\label{section_SATasSDPfeas}
In this section we first present a brief overview of the work done by \citeauthor{de2000relaxations} \cite{de2000relaxations} and \citeauthor{anjos2004proofs} \cite{anjos2004proofs, anjos2004semidefinite, anjos2005improved, anjos2006semidefinite, anjos2008extended}. Their relaxations of the SAT involve semidefinite feasibility problems: infeasibility of the semidefinite programme implies unsatisfiability of the corresponding SAT instance. The difference between relaxations is the size of the SDP variable, and the method of encoding the structure of the SAT instance in the SDP relaxation. Then, we propose a family of semidefinite feasibility problems, that contains relaxations from \cite{anjos2004proofs}--\cite{anjos2008extended} and \cite{de2000relaxations} as special cases, and show that a particular member of the family  provides a rank-two guarantee, see \Cref{thm_NewFsetRank2Guarantee}.

We reconsider first programme \eqref{eqn_GoemansWilliamson}, which attempts to satisfy the maximum number of clauses  through its objective function. For the SAT specifically, one might move the clause satisfaction part from the objective to the feasible set of a semidefinite programme. This idea was first proposed by \citeauthor{de2000relaxations} \cite{de2000relaxations} in \citeyear{de2000relaxations}, and was later extended by \citeauthor{anjos2004proofs} \cite{anjos2004proofs}. To be precise: \citeauthor{de2000relaxations} propose the so called \ref{GapRelax} relaxation, or \ref{GapRelax} for short, which is a semidefinite feasibility problem, given by
\begin{equation}
\label{GapRelax}
\tag{GAP}
\begin{aligned}
\text{find} \quad &  Y \in \mathcal{S}^n_+, \, y \in \mathbb{R}^n \\
\textrm{s.t.} \quad & a_j^\top  Y a_j - 2 a_j^\top y \leq \ell_j(\ell_j-2), \, \forall j \in [m], \\
  & \text{diag}(Y) = \mathbf{1},    \\
  & Y \succeq yy^\top,
\end{aligned}
\end{equation}
for $a_j$ as in \eqref{def_ajVector}. It is noted in \cite{de2000relaxations}, that for $\ell_j \leq 2$, the corresponding inequalities in \ref{GapRelax} relaxation may be changed to equalities. The \ref{GapRelax} relaxation is suited for instances that contain a clause of length two. If $\ell_j \geq 3, \, \forall j \in [m]$, then $(Y,y) = (I, \mathbf{0})$ is always feasible for \ref{GapRelax}, whether the underlying SAT instance is satisfiable or not.

We now state the SDP relaxations of the SAT by \citeauthor{anjos2004semidefinite} \cite{anjos2004proofs}-\cite{anjos2008extended} that are not restricted to the lengths of the clauses in instances. Let $\Logical$ be a proposition on $n$ variables and $m$ clauses and $x \in \{ \pm 1 \}$ the truth assignment to the variables. Consider a family of subsets $\GenSet = \{ \alpha_1,  \ldots, \alpha_s \}$, let $\mathbf{x} = (x^{\alpha_1}, \, \ldots, x^{\alpha_{s}})^\top$, and define $Y := \mathbf{x} \mathbf{x}^\top$. It is clear that $\text{rank}(Y) = 1$, diag$(Y) = \mathbf{1}$, and $Y \succeq 0$. Later, to obtain a semidefinite relaxation of the SAT, we omit the rank one constraint.

We index the matrix $Y$ with the elements of $\GenSet$, and define for all subsets $\gamma$ contained in some clause of $\Logical$, the expression
\begin{align}
    \label{eqn_yTdef}
    Y(\gamma) := Y_{\alpha, \beta}, \text{ for some } \alpha, \, \beta \in \GenSet \text{ jointly contained in a single clause, such that } \alpha \triangle \beta = \gamma,
\end{align}
where $\triangle$ is the symmetric difference operator, which is induced by the fact that, for $x \in \{ \pm 1 \}^n$, we have $Y_{\alpha, \beta} = x^\alpha x^\beta = x^{\alpha \triangle \beta} = x^\gamma$, see \eqref{eqn_xAlphaDef}. In general, $Y(\emptyset)$ refers to a diagonal entry of $Y$, hence, $Y(\emptyset) = 1$. We may have $Y(\gamma) = Y_{\emptyset, \gamma}$, and we assume that, for all $\gamma$ contained in a clause of $\Logical$, we can always find $\alpha$ and $\beta$ as in \eqref{eqn_yTdef}.

The expression $Y(\gamma)$ can refer to multiple entries of $Y$. By construction of $Y$, these entries are equal. Stated formally, we have $Y \in \cap_{j \in [m]} \Delta_j$, where
\begin{align}   
    \label{eqn_symmDiffPodd}
    \Delta_j := \{ Y \in \mathcal{S} \, | \, Y_{\alpha_1,\beta_1} = Y_{\alpha_2,\beta_2} \,\, \forall \, (\alpha_1, \alpha_2, \beta_1, \beta_2) \in \GenSet \text{ such that } \alpha_1 \triangle \beta_1 = \alpha_2 \triangle \beta_2 \subseteq C_j \}.
\end{align}
Observe that the sets $\Delta_j$ do not capture all equalities present in $Y$, due to the restriction $\alpha_1 \triangle \beta_1 = \alpha_2 \triangle \beta_2 \subseteq C_j$. In this section, we choose to include only the equalities captured by $\Delta_j$. This keeps our work in line with previous relaxations by \citeauthor{anjos2004semidefinite} \cite{anjos2004semidefinite}, and these equalities suffice to prove the main theorem in this section, see \Cref{thm_NewFsetRank2Guarantee}. In \Cref{section_dualRelation}, we consider an SDP relaxation of the SAT which considers all equalities present in $Y$.

If $x$ is a satisfying assignment to $\Logical$, then $v(C_j) = 1$, see \eqref{eqn_valueProp}, for all $j \in [m]$. We can rewrite this constraint in terms of $Y(\gamma)$, see \eqref{eqn_yTdef}. We now omit the rank one constraint on $Y$, to obtain the following semidefinite feasibility programme, denoted \ref{eqn_RFprogramme}:
\begin{equation}
\label{eqn_RFprogramme}
\tag{$R_{\GenSet}(\Logical)$}
\begin{aligned}
\text{find} \quad &  Y \in \mathcal{S}^{|\GenSet|}_+ \\
\textrm{s.t.} &  \sum_{\gamma \subseteq C_j} (-1)^{|\gamma|} \, a_j^\gamma  \, \, Y(\gamma)  = 0, \, \text{ for all } j \in [m], \\
  & \text{diag}(Y) = \mathbf{1}, \, Y \in \bigcap_{j \in [m]} \Delta_j.
\end{aligned}
\end{equation}
The programme \ref{eqn_RFprogramme} contains both the \ref{GapRelax} relaxation, and the relaxations proposed by \citeauthor{anjos2004proofs} \cite{anjos2004proofs} as special cases. Specifically, one obtains the \ref{GapRelax} relaxation from \ref{eqn_RFprogramme} by taking $\GenSet = \{ \alpha \subseteq [n] \, | \, |\alpha| \leq 1 \}$. It was proved in \cite{de2000relaxations} that whenever \ref{GapRelax} is feasible for a 2-SAT instance, the 2-SAT instance is satisfiable.

The \ref{GapRelax} relaxation can be considered as a semidefinite programme in the first level of the well-known Lasserre hierarchy \cite{lasserre2001global}. \citeauthor{anjos2004proofs} \cite{anjos2004proofs} proposed semidefinite relaxations of the SAT in approximately levels two and three of the Lasserre hierarchy, by only adding a subset of products of variables to the moment relaxation. For example, in \cite{anjos2004semidefinite}, \citeauthor{anjos2004semidefinite} proposed the $R_2$ relaxation, which can be obtained from \ref{eqn_RFprogramme} by taking
\begin{align}
    \label{eqn_AnjosFset}
    \GenSet = \left \{ \alpha \, | \, \alpha \subseteq C_j \text{ for some } j, \, | \alpha | \text{ odd, or } \alpha = \emptyset \right \}.
\end{align}
It was proved in \cite{anjos2004semidefinite} that the $R_2$ relaxation attains a rank two guarantee on 3-SAT instances: whenever the SDP programme admits a feasible matrix of rank two or lower, the corresponding SAT instance is satisfiable. We will now prove that, for a different $\GenSet$ than \eqref{eqn_AnjosFset}, the resulting relaxation \ref{eqn_RFprogramme} provides the same rank two guarantee.

\begin{theorem}
\label{thm_NewFsetRank2Guarantee}
Let $\Logical$ be a 3-SAT instance and
\begin{align}
    \label{eqn_FinThm}
    \GenSet = \left \{ \alpha \subseteq [n] \, | \, \alpha \subseteq C_j \textup{ for some } j, |\alpha| \neq 1 \right \}.
\end{align}
Let $Y$ be the matrix variable of \ref{eqn_RFprogramme}, indexed by elements of $\GenSet$. If \ref{eqn_RFprogramme} admits a feasible rank two matrix, then $\Logical$ is satisfiable.
\end{theorem}
\begin{proof}
The proof is adapted from Theorem 3 in \cite{anjos2004semidefinite}, in which the theorem is proven for the case that $\GenSet$ is given as in \eqref{eqn_AnjosFset}.

Since $\Logical$ is a 3-SAT instance, there exists a clause of length three. Fix a $j$ for which $C_j = \{ i_1,i_2,i_3\}$ and set
\begin{align}
    \GenSet_j = \left \{ \alpha \subseteq [n] \, | \, \alpha \subseteq C_j, \, \alpha \in \GenSet \right \},
\end{align}
for $\GenSet$ as in \eqref{eqn_FinThm}. Let $Y$ be feasible solution to \ref{eqn_RFprogramme} of rank 2. 
Consider the submatrix of $Y$ indexed by the elements of $\GenSet_j$,
\begin{align}
    \label{eqn_thm1proof}
    \begin{bmatrix}
    1 & Y(i_1 i_2) &Y(i_1 i_3) &Y(i_2 i_3) & Y(i_1 i_2 i_3) \\
    Y(i_1 i_2) & 1 &Y(i_2 i_3) & Y( i_1i_3) & Y(i_3)\\
    Y(i_1i_3) & Y(i_2i_3) & 1 & Y(i_1i_2) & Y(i_2) \\
    Y(i_2i_3) & Y(i_1i_3) & Y(i_1i_2)  & 1 & Y(i_1) \\
    Y(i_1i_2i_3) & Y(i_3) & Y(i_2) & Y(i_1) & 1
    \end{bmatrix}.
\end{align}
where $Y(\cdot)$ is given as \eqref{eqn_yTdef}. For example, $Y(i_1 i_2) = Y_{\emptyset, i_1  i_2}$ and $Y(i_1) = Y_{i_2i_3,i_1i_2i_3}$. As the top left \mbox{4 $\times $ 4} submatrix of \eqref{eqn_thm1proof} has rank at most 2, it can be proven (Lemma 3.11 from \cite{anjos2002geometry}) that at least one of $Y(i_1i_2)$, $Y(i_1i_3)$, $Y(i_2i_3)$ equals $\delta \in \{ \pm 1 \}$.

Assume without loss of generality that $Y(i_1i_2) = \delta$. Consider now the submatrices of $Y$ indexed by rows and columns $\{ \emptyset,i_1i_2,i_1i_3 \}$, $\{ i_1i_3,i_2i_3,i_1i_2i_3 \}$ and  $\{ \emptyset, i_1i_2,i_1i_2i_3 \}$. As the diagonals of these positive semidefinite matrices equal $\mathbf{1}_3$, it can be shown (Lemma 3.9 from \cite{anjos2002geometry}) that
\begin{align}
    \label{eqn_deltaSwitches}
    Y(i_1i_3) =  \delta Y(i_2i_3), \,\,
    Y(i_2) =  \delta Y(i_1),  \,\,
    Y(i_3) = \delta Y(i_1i_2i_3).
\end{align}
This allows us to simplify the satisfiability constraint on $C_j$, given by
\begin{align}
    \label{eqn_SatConsPart}
    1 - a_1 Y(i_1) - a_2 Y(i_2) - a_3 Y(i_3) + a_1a_2Y(i_1i_2) + a_1 a_3 Y(i_1i_3) + a_2 a_3 Y(i_2 i_3) - a_1 a_2 a_3 Y(i_1 i_2 i_3) = 0,
\end{align}
see \eqref{eqn_valueProp}. Here, we have written $a_k$ for $a_{j,i_k}$, see \eqref{def_ajVector}. We rewrite \eqref{eqn_SatConsPart} by substituting $Y(i_1i_2) = \delta$ and \eqref{eqn_deltaSwitches} to obtain
\begin{align}
    \label{eqn_subeqnpart1}
     1 - a_1 Y(i_1) - a_2 \delta Y(i_1) - a_3 \delta Y(i_1 i_2 i_3) + a_1a_2 \delta + a_1 a_3 \delta Y(i_2i_3) + a_2 a_3 Y(i_2 i_3) - a_1 a_2 a_3 Y(i_1 i_2 i_3) = 0.
\end{align}
Using that $a_k^2 = 1$, \eqref{eqn_subeqnpart1} can be factorized as
\begin{align}
    \label{eqn_proofcase2}
    \bigg[1 + a_1 a_2 \delta \bigg] \bigg[ 1 - a_1 Y(i_1) + a_2 a_3 Y(i_2i_3) - a_1 a_2 a_3 Y(i_1i_2i_3) \bigg] = 0.
\end{align}
In the case $1+a_1 a_2 \delta \neq 0$, equation \eqref{eqn_proofcase2} reduces to
\begin{align}
    \label{eqn_proofcase3}
    1 - a_1 Y(i_1) + a_2 a_3 Y(i_2i_3) - a_1 a_2 a_3 Y(i_1i_2i_3) = 0,
\end{align}
which is a linear constraint in the entries of matrix $Y$. In case $1 + a_1 a_2 \delta = 0$, clause $C_j$ holds automatically for $Y(i_1i_2) = \delta$.

Consider the clauses $C_j$ of length three, which can all be factorized as $g_j(\delta) f_j(Y) = 0$, as shown in \eqref{eqn_proofcase2}. For all $j$ such that $g_j(\delta) \neq 0$, let $B_j$ be the set of the two subsets appearing in the intersection of $f_j(Y)$ and $\GenSet$. Note that for the particular equation \eqref{eqn_proofcase2}, we would have
\begin{align}
    B_j = \{ \{i_2,i_3\}, \{ i_1, i_2, i_3 \} \},
\end{align}
in case $1 + a_1 a_2 \delta  \neq 0$. Let $B$ be the union of all $B_j$. If we define
$    z_1 = \{ i_2, i_3 \} \text{ and } z_2 = \{i_1, i_2, i_3 \},$
then \eqref{eqn_proofcase3} is the constraint of the semidefinite relaxation of a clause of length two on the variables $z_1$ and $z_2$, as $Y(i_1) = Y_{z_1, z_2}$, see \eqref{eqn_valueProp}.

The submatrix of $Y$ indexed by all sets in $B$ can thus be viewed as a matrix indexed by singletons (by associating a new $z$ variable to each element of $B$). As $Y$ is feasible for \ref{eqn_RFprogramme}, it is automatically feasible for the \ref{GapRelax} relaxation corresponding to some 2-SAT instance on the $z$ variables, which implies satisfiability of the corresponding 2-SAT instance (Theorem 5.1 \cite{de2000relaxations}). This implies the $z$ variables have a satisfying truth assignment. Given such a truth assignment for the $z$ variables, it is not hard to extend the truth assignment to the original variables of $\Logical$, by using the appropriate values of $\delta$, see \eqref{eqn_deltaSwitches}. This proves the theorem.
\end{proof}

\section{Sum of squares and the MAX-SAT}
\label{section_SOSandMaxSat}

In  \Cref{sec:general} we first provide an  overview of the approach of  \citeauthor{van2008sums} \cite{van2008sums} for deriving relaxations for the MAX-SAT. Their approach exploits SOS optimization, which has received much attention in the literature, see e.g., \cite{lasserre2007sum,laurent2009sums, ParilloThomas,scheiderer2009positivity}.
Relaxations depend on a basis that is used to compute them. We introduce a parametric family of monomial bases with increasing complexity. In \Cref{sec:sosp} we derive several properties of monomial bases that are later used in our computations.  

\subsection{General overview}
\label{sec:general}
For a given logical proposition $\Logical$, on $n$ variables and $m$ clauses, the value
\begin{align}
    \label{eqn_FbFunctionDefinition}
    F_{\phi}(x) :=  \sum_{j = 1}^m \frac{1}{2^{\ell_j}}\prod_{i \in C_j}(1-a_{j,i} x_i ),
\end{align}
for $a_{j,i}$ as in \eqref{def_ajVector}, equals the number of unsatisfied clauses by truth assignment $x \in \{ \pm 1 \}^n$. Hence, we are interested in minimizing $F_\Logical$ over $\{ \pm 1 \}^n$, on which $F_\Logical$ is nonnegative. Let $\mathbb{R}[x]$ be the set of real polynomials in $x$. We define
\begin{align}
    \label{eqn_Putinar}
    \mathcal{V} := \Big\{ f \, \Big| \, f \equiv \sum_{j= 1}^k f_j^2 \mod I, \, f_j \in \mathbb{R}[x] \,\, \forall j \in [k], \, k \in \mathbb{N}  \Big\}
\end{align}
as the set of SOS polynomials modulo $I$, where $I$ is the vanishing ideal of $\{ \pm 1 \}^n$. That is
\begin{align}
\label{eqn_modIdef}    
I = \langle 1-x_1^2, 1-x_2^2, \ldots, 1- x_n^2 \rangle.
\end{align}

By \textit{Putinars Positivstellensatz} \cite{putinar1993positive}, $\mathcal{V}$ is the set of nonnegative polynomials on $\{ \pm 1 \}^n$. 
Generally, optimization over $\mathcal{V}$ is intractable due to its size, which is why we consider
\begin{align}
    \label{eqn_VxSet}
    \mathcal{V}_\mathbf{x} := \{ f \, | \, f \equiv \mathbf{x}^\top M \mathbf{x} \mod I, \, M \succeq 0 \},
\end{align}
where $\mathbf{x}$ is some monomial basis. Since $M \succeq 0$, it follows that all polynomials of $\mathcal{V}_\mathbf{x}$ are nonnegative on $\{ \pm 1 \}^n$. Therefore, $\mathcal{V}_\mathbf{x} \subseteq \mathcal{V}$, and we may approximate the minimum of $F_\Logical$ by 
\begin{align}
    \label{eqn_PolySOS}
    \min_{x \in \{ \pm 1 \}^n} F_{\phi} = \sup \{ \lambda \, | \, F_{\phi} - \lambda \in \mathcal{V} \} \geq \sup \{ \lambda \, | \, F_{\phi} - \lambda \in \mathcal{V}_\mathbf{x} \}.
\end{align}
The description of $\mathcal{V}_\mathbf{x}$ shows that computing this lower bound can be done via SDP.

It is important to note that in the quotient ring of $\mathbb{R}[x]$ modulo $I$, all terms $x_i^2 \equiv 1$, and thus it suffices to consider only monomials in $\mathbf{x}$ for which the highest power is at most 1. Thus, we can write
\begin{align}
    \label{eqn_polynomialDef}
    F_{\phi}(x) = \sum_{\alpha \subseteq [n]} p^{\alpha}_\Logical x^\alpha, 
\end{align}
where $p_\Logical^\alpha \in \mathbb{R}$ for all $\alpha \subseteq [n]$ and $x^\alpha$ as in \eqref{eqn_xAlphaDef}. 
For the constant term of $F_{\phi}(x)$, we have
\begin{align}
    \label{eqn_constantTermF}
    p_\Logical^{\emptyset} =  \sum_{j = 1}^m \frac{1}{2^{\ell_j}}.
\end{align}
We say that monomial basis $\mathbf{x}$ \textit{represents} a logical proposition $\Logical$ if matrix $\mathbf{X} \equiv \mathbf{x} \mathbf{x}^\top \mod I$ contains all monomials $x^\alpha$ for which $p^\alpha_\Logical \neq 0$. We index this matrix $\mathbf{X}$ and the matrix $M$ from \eqref{eqn_VxSet} with subsets $\alpha \subseteq [n]$ for which $x^\alpha \in\mathbf{x}$. Note that for such $\alpha, \beta \subseteq [n]$, we have
\begin{align}
    \label{eqn_xxTmod}
    \mathbf{X}_{\alpha, \beta} \equiv x^{\alpha \triangle \beta} \mod I.
\end{align}
For $\alpha \subseteq [n]$, we write $x^\alpha \in \mathbf{X}$ if $\mathbf{X}$ has an entry equal to $x^\alpha$ (modulo $I$). \citeauthor{van2008sums} \cite{van2008sums} propose multiple monomial bases $\mathbf{x}$, among them basis $SOS_p$, given by
\begin{align}
    \label{eqn_sosPbasis}
    \mathbf{x} = 1 \cup \{ x_i \, | \, i \in [n] \} \cup \{ x_i x_j \, | \, i \text{ and } j \text{ appear together in a clause}\}.
\end{align}
It is stated in \cite{van2008sums} that $SOS_p$ represents 2-SAT and 3-SAT instances. While this is true, this basis also represents 4-SAT instances (see \Cref{lemma_cardXGammaUpperBounds}). We additionally define for $\mathbf{Q} \in \mathbb{N}$, as extension to $SOS_p$, the basis
\begin{align}
    \label{eqn_sosPQBasis}
    SOS_p^{\mathbf{Q}} := SOS_p \cup \{ x_i x_j \, | \, i \text{ and } j \text{ are both in the top } \mathbf{Q} \text{ appearing variables} \}.
\end{align}
Basis $SOS_p^{\mathbf{Q}}$ takes basis $SOS_p$ and adds all the $\binom{\mathbf{Q}}{2}$ quadratic terms of the $\mathbf{Q}$ variables appearing in the highest number of clauses of $\Logical$. Any basis $\mathbf{x}$ is considered to have duplicate monomials removed, and so, for small values of $\mathbf{Q}$, bases $SOS_p^{\mathbf{Q}}$ and $SOS_p$ might coincide.

 We also define the basis $SOS_s^{\theta}$, for $\theta \in [0,1]$, which is suited for  (MAX-)2-SAT instances. This basis consists of all the monomials of degree one and zero, plus a percentage $\theta$ of all quadratic monomials appearing in $SOS_p$. The included quadratic monomials are those that appear in $SOS_p$ and attain the highest monomial weight $w$, which is defined as $w(x^\alpha) := \sum_{i \in \alpha} w(i)$, where $w(i) := | \{ C \in \Logical \, | \, i \in C\}|$, for $i \in [n]$. This results in the following chain of inclusions:
\begin{align}
    \label{eqn_monBasisInclusionChain}
    \{ x^\alpha \, | \, |\alpha| \leq 1 \} = SOS_s^0 \subseteq SOS^\theta_s \subseteq SOS_s^1 = SOS_p = SOS_p^0 \subseteq SOS_p^{\mathbf{Q}} \subseteq SOS_p^n = \{ x^\alpha \, | \, |\alpha| \leq 2 \}.
\end{align}

We now define for all $\gamma \subseteq [n]$ such that $x^\gamma \in \mathbf{X}$, a set of ordered pairs  as follows
\begin{align}
    \label{eqn_MgammaDef}
   \mathbf{x}^\gamma := \{ (\alpha, \beta) \subseteq [n]^2 \, | \, \alpha \triangle \beta = \gamma, \, x^\alpha \in \mathbf{x}, \, x^\beta \in \mathbf{x} \}.
\end{align}
Set $\mathbf{x}^\gamma$ contains the index pairs $(\alpha, \beta)$ such that $\mathbf{X}_{\alpha, \beta} \equiv x^\gamma$. Therefore, $F_{\phi} \equiv \mathbf{x}^\top M \mathbf{x}$ if and only if
\begin{align}
    \label{eqn_xMxExplicitCons}
    \sum_{(\alpha, \beta) \in \mathbf{x}^\gamma} M_{\alpha, \beta} = p^\gamma_{\Logical}, \quad \forall \gamma \text{ such that } x^\gamma \in \mathbf{X}.
\end{align}
Constraints of this form are sometimes referred to as \textit{coefficient matching conditions} in SOS literature \cite{zheng2017exploiting}. We define
\begin{align}
    \label{eqn_mathcalMset}
    \mathcal{M}_{\phi} := \bigg\{ M \in \mathcal{S}^{|\mathbf{x}|} \, | \,\sum_{(\alpha, \beta) \in \mathbf{x}^\gamma} M_{\alpha, \beta} = p_\Logical^{\gamma}, \quad \forall \gamma \neq \emptyset \text{ such that } x^\gamma \in \mathbf{X} \bigg\},
\end{align}
as the set of matrices satisfying the coefficient matching conditions, for all monomials except $x^\emptyset$.

Note that  $M$ is constrained to be symmetric which is reflected in the definition of $\mathbf{x}^\gamma$, since $(\alpha, \beta) \in \mathbf{x}^\gamma$ if and only if $(\beta, \alpha) \in \mathbf{x}^\gamma$. Moreover, $\mathbf{x}^\emptyset$ contains the index pairs of the diagonal entries of $M$, which correspond to zero-degree monomials in $\mathbf{X}$. Hence, if $F_\Logical - \lambda \equiv \mathbf{x}^\top M \mathbf{x}$, then $p^\emptyset_\Logical - \lambda = \langle I, M \rangle$, see \eqref{eqn_PolySOS} and \eqref{eqn_constantTermF}. To maximize the lower bound on $F_\Logical$, see \eqref{eqn_PolySOS}, we maximize $\lambda$, which is thus equivalent to minimizing $\langle I, M \rangle$. We can therefore compute this lower bound by solving the following SDP: 
\begin{equation}
\label{eqn_SumSquaresExplicit}
\tag{$\texttt{P}_\Logical$}
\begin{aligned}
\text{min} \quad & \langle I, M \rangle \\
\textrm{s.t.} \quad & M \in \mathcal{M}_\Logical \cap \mathcal{S}_+.
\end{aligned}
\end{equation}

We note that, for the purpose of solving \ref{eqn_SumSquaresExplicit} through interior point methods, programme \ref{eqn_SumSquaresExplicit} is strictly feasible: for any feasible matrix $M$, matrix $M+I$ is strictly feasible. The existence of any such feasible matrix $M$ follows from the nonnegativity of $F_\Logical$ on $\{ \pm 1 \}^n$. We postpone the derivation of the dual of \ref{eqn_SumSquaresExplicit} to \Cref{section_dualRelation}, where we also show its strict feasibility in \Cref{thm_dualEquivalence}.

\subsection{Properties of $SOS_p^\mathbf{Q}$}
\label{sec:sosp}
We provide provide several properties of monomial bases that are exploited within the PRSM, see~\Cref{section_PRSMsection}.

Denote by $| \mathbf{x}^\gamma|$ the cardinality of the set $\mathbf{x}^\gamma$, see \eqref{eqn_MgammaDef}. Due to the symmetry of $\mathbf{X}$, see \eqref{eqn_xxTmod}, $|\mathbf{x}^\gamma|$ is an even number and greater than or equal to 2. In particular, when $|\mathbf{x}^\gamma| = 2$, say $\mathbf{x}^\gamma = \{ (\alpha, \beta), (\beta, \alpha)\}$, we have
\begin{align}
    \label{eqn_unitCons}
    M_{\alpha, \beta} + M_{\beta, \alpha} = p^\gamma_\Logical \text{ and }  M_{\alpha, \beta} = M_{\beta, \alpha} \Longrightarrow M_{\alpha, \beta} = M_{\beta, \alpha} = p^\gamma_\Logical / 2.
\end{align}
Thus, whenever $|\mathbf{x}^\gamma| = 2$, the constraint involving $\mathbf{x}^\gamma$ in $\mathcal{M}_\Logical$, see \eqref{eqn_mathcalMset}, simply fixes two entries of $M$. \citeauthor{van2008sums} \cite{van2008sums} refer to these constraints arising from $|\mathbf{x}^\gamma| = 2$ as \textit{unit constraints}. In section 7 of \cite{van2008sums}, the authors empirically show that a high percentage of the constraints of $\mathcal{M}_\Logical$ are unit constraints. The authors of \cite{van2008sums} propose as future work the development of an SDP solver that is able to exploit the large number of unit constraints. We propose an algorithm for approximately solving \ref{eqn_SumSquaresExplicit} in \Cref{section_PRSMsection}, which is able to do so.

The following lemma describes the subsets $\gamma$ that induce unit constraints.

\begin{lemma}
\label{lemma_unitConstraintsCoeff}
Let $\Logical$ be a (MAX-)SAT instance on $n$ variables and $m$ clauses, and $\mathbf{x}$ a monomial basis which contains at least all of the monomials induced by the $SOS_p$ basis, see \eqref{eqn_sosPbasis}. Then, for all $\gamma \subseteq [n]$, we have
\begin{align}
    | \mathbf{x}^\gamma | = 2 \quad \Longrightarrow  \quad p_\Logical^\gamma = 0,
\end{align}
where $p_\Logical^\gamma$ is a coefficient of $F_{\phi}(x)$, see \eqref{eqn_polynomialDef}.
\end{lemma}
\begin{proof}
From the definition of $F_{\phi}(x)$, see \eqref{eqn_FbFunctionDefinition} and \eqref{eqn_polynomialDef}, it follows that, for all $\gamma \subseteq [n]$,
\begin{align}
    \label{eqn_subsetZeroCoeffImplication}
    \gamma \not \subseteq C_j \quad \forall j \in [m] \quad \Longrightarrow \quad p^\gamma_\Logical = 0.
\end{align}
We proceed by showing that, if $| \mathbf{x}^\gamma | = 2$ for some $\gamma \subseteq [n]$, then $\gamma$ cannot be contained in a clause. For this purpose, let $\alpha$ be contained in some clause of $\Logical$, say $C_j$. We distinguish separate cases for $| \alpha |$. Let us consider the case $| \alpha | \in \{ 0,1,2\}$. Then, without loss of generality, we assume $\alpha \subseteq \{ i_1, i_2 \} \subseteq C_j$. Now (some of) the monomials $x^\alpha \in \mathbf{x}$ induced by $C_j$ are those for which $\alpha \in S := \{ \emptyset, i_1, i_2, \{ i_1, i_2 \} \}$.

We display all the subsets obtained by taking pairwise symmetric differences of elements of $S$ in the symmetric matrix
\begin{align}
\label{eqn_SjMatrix}
S_\triangle := \begin{bmatrix}
       \emptyset &  i_1  &  i_2   &  i_1 i_2    \\
    &\emptyset &  i_1 i_2  &  i_2\\
    &  &\emptyset &  i_1 \\
    &  &  &\emptyset \\
   \end{bmatrix}.
\end{align}

For clarity, we omitted the lower triangular part of $S_\triangle$ (which is fixed by symmetry). Observe that for any of these $\alpha$ satisfying $|\alpha| \in \{0,1,2\}$, we have $| \mathbf{x}^\alpha | \geq 4$. In the case $| \alpha | = 3$, clause $C_j$ must be of length at least three, and we assume, without loss of generality, that $\alpha = \{ i_1, i_2, i_3 \} \subseteq C_j$. By again constructing $S_\triangle$ (details omitted), one can show that $|\mathbf{x}^\alpha| \geq 4$. The proof is similar for $| \alpha | = 4$. Lastly, if $| \alpha | = 5$, then by definition of $SOS_p$, see \eqref{eqn_sosPbasis}, $| \mathbf{x}^\alpha |= 0$.

Thus, for all $\alpha \subseteq [n]$ contained in some clause of $\Logical$, $| \mathbf{x}^\alpha| \neq 2$. Therefore,
\begin{align}
    |\mathbf{x}^\gamma| = 2 \Longrightarrow \gamma \not \subseteq C_j \quad \forall j \in [m],
\end{align}
which, combined with \eqref{eqn_subsetZeroCoeffImplication}, completes the proof.
\end{proof}

\Cref{lemma_unitConstraintsCoeff} implies that, in an implementation which uses the $SOS_p$ basis, it is not required to store the coefficients corresponding to unit constraints (since these all equal 0), but only the indices restricted by the unit constraints. The converse of \Cref{lemma_unitConstraintsCoeff} is generally not true. That is, there can exist many subsets $\gamma \subseteq [n]$ for which $p_\Logical^\gamma = 0$, but $|\mathbf{x}^\gamma| > 2$.

\begin{lemma}
\label{lemma_cardXGammaUpperBounds}
Let $\Logical$ be a SAT instance on $n$ variables and $\mathbf{x}$ its monomial basis according to $SOS_p^\mathbf{Q}$, for some $\mathbf{Q} \in \mathbb{N}$. Let $\gamma \subseteq [n]$. Then
\begin{enumerate}
    \item $|\gamma| \in \{1,2\} \Longrightarrow |\mathbf{x}^\gamma| \leq 2n$. \label{item_12}
    \item $|\gamma| \in \{3,4\} \Longrightarrow |\mathbf{x}^\gamma| \leq 6$. \label{item_34}
    \item $|\gamma| > 4 \Longrightarrow |\mathbf{x}^\gamma| = 0$. \label{item_gamma5}
\end{enumerate}
\end{lemma}
\begin{proof}
The proof follows from enumerating the combination of sets such that their pairwise difference is equal to $\gamma$. For part \ref{item_12} of the lemma, assume first, without loss of generality, that $\gamma = \{ 1 \}$. Then consider the tuples $( \emptyset, \{ 1 \} )$ and $( \{1,k\}, \{  k \} )$, for $k \in [n] \setminus \{1 \}$. There are $2n$ of these tuples (counting the symmetry of order twice) and their symmetric differences all equal $\gamma$. Next, we assume without loss of generality that $\gamma = \{1,2\}$. Then the tuples $( \{1\}, \{2\} ), \, (\{1,2\}, \emptyset )$ and $( \{1,k\}, \{2,k\} )$, for $k \in [n] \setminus \{1,2 \}$, have their pairwise symmetric difference equal to $\gamma$. There are $2n$ of these tuples, which proves part \ref{item_12} of the lemma.

Assuming that $\gamma = \{1,2,3\}$, we find the 6 tuples as $( \{ 1\}, \{ 2,3\} ), \, (\{ 2\}, \{ 1,3\}), \, (\{ 3\}, \{ 1, 2\})$. If instead $|\gamma| = 4$, each tuple corresponds to one of $\binom{4}{2} = 6$ partitions which proves part \ref{item_34}.

Lastly, it follows from the definition of $SOS^{\mathbf{Q}}_p$, that any monomial in matrix $\mathbf{x} \mathbf{x}^\top$ is of degree at most four, which proves part \ref{item_gamma5} of the lemma.
\end{proof}

Part \ref{item_gamma5} of \Cref{lemma_cardXGammaUpperBounds} shows that the $SOS^{\mathbf{Q}}_p$ bases are only suited for the (MAX-)$k$-SAT when $k \leq 4$.

\section{Resolution and monomial bases}
\label{section_Resolution}
In this section, we consider resolution in combination with the SOS approach to the MAX-SAT. Resolution is a technique from mathematical logic, and widely employed by MAX-SAT solvers \cite{prasad2005survey}. Resolution takes as inputs two clauses of a proposition $\Logical$, and returns a set of new clauses, named the \textit{resolvent} clauses. The resolvent clauses transform $\Logical$ into $\Logical^\prime$, by either replacing the original clauses, or by adding the resolvent clauses to $\Logical$ (depending on which resolution rule is used). We show in this section that the MAX-SAT resolution rule might not be beneficial for the SOS approach applied to the MAX-SAT,  and can even decrease its effectiveness. However, in \Cref{subsection_SATresolutionRule} we show how to benefit from the SAT resolution rule in the partial MAX-SAT.

We show this at the hand of an example. For $k \geq 3$, we define the following proposition on $k$ variables
\begin{align}
    \label{eqn_logicalKresolutionExample}
    \Logical_k = \begin{cases} 
    \neg x_1 \wedge (x_1 \vee \neg x_2) \wedge (x_2 \vee x_3) \wedge \neg x_3 \hphantom{\wedge \Big[ \bigwedge_{j=4}^k (\neg x_{j-1} \vee x_{j}) \Big]} \quad \text{ if } k = 3, \\
    \neg x_1 \wedge (x_1 \vee \neg x_2) \wedge (x_2 \vee x_3) \wedge \Big[ \bigwedge_{j=3}^{k-1} (\neg x_{j} \vee x_{j+1}) \Big] \wedge \neg x_k \, \, \, \text{ else.}
    \end{cases}
\end{align}
It is clear that $\Logical_k$ is unsatisfiable. If one satisfies the initial two unit clauses and performs unit resolution, more unit clauses appear. Repeating this process will lead to an all \texttt{false} truth assignment, leaving clause $x_2 \vee x_3$ unsatisfied. Therefore, any truth assignment leaves at least one clause unsatisfied, and hence, $\min_{x \in \{ \pm 1\}^k}F_{\Logical_k} = F_{\Logical_k}(- \mathbf{1}_k) = 1$, for $F_{\Logical_{k}}$ as in \eqref{eqn_FbFunctionDefinition}.

In the following lemma, we show that the $SOS_p$ basis, see \eqref{eqn_sosPbasis}, suffices for proving optimality of this assignment.
\begin{lemma}
    \label{lemma_sosPboundStrength}
    For all $k \geq 3$, we have
    \begin{align}
        \label{eqn_sosPBoundStrength1}
        \max\{ \lambda \, | \, F_{\Logical_k} - \lambda \in \mathcal{V}_{\mathbf{x}} \} = 1,
    \end{align}
    where $\mathbf{x} = SOS_p(\Logical_k)$, and $\mathcal{V}_{\mathbf{x}}$ as in \eqref{eqn_VxSet}.
\end{lemma}
\begin{proof}
Since $\min_{x \in \{ \pm 1\}^k}F_{\Logical_k} = 1$, $\nexists \lambda > 1$ such that $F_{\Logical_k} - 
\lambda \in \mathcal{V}$, see \eqref{eqn_Putinar}. As $\mathcal{V}_\mathbf{x} \subseteq \mathcal{V}$, this implies that $\nexists \lambda > 1$ such that $F_{\Logical_k} - 
\lambda \in \mathcal{V}_\mathbf{x}$. Thus, to prove \eqref{eqn_sosPBoundStrength1} it suffices to show that $F_{\Logical_k} - 
1 \in \mathcal{V}_\mathbf{x}$, for all $k \geq 3$.

We prove this by induction. For the base case $k=3$, we have
\begin{align}
     F_{\Logical_3} &= (6 + x_1 + x_3 - x_1 x_2 + x_2 x_3)/4 \\
     &\equiv 1+\Big[ (2+x_1+x_3-x_1 x_2 + x_2 x_3)^2 + (-x_1 + 2x_2 + x_3 +x_1 x_2 + x_2 x_3)^2 \Big] / 32 \mod I.
\end{align}
Clearly, the monomials appearing in the above decomposition are contained in $SOS_p(\Logical_k)$. Therefore, $F_{\Logical_3} - 1 \in \mathcal{V}_\mathbf{x}$. Now for $k+1$, we have
\begin{align}
    F_{\Logical_{k+1}} &= F_{\Logical_{k}} + (1-x_{k}+x_{k+1}-x_k x_{k+1})/4 \\
    &\equiv 1+(F_{\Logical_{k}}-1)  + (1-x_{k}+x_{k+1}-x_k x_{k+1})^2/16 \mod I.
\end{align}
By the induction hypothesis, $F_{\Logical_{k}}-1 \in \mathcal{V}_\mathbf{x}$, and so it follows that $F_{\Logical_{k+1}}-1 \in \mathcal{V}_\mathbf{x}$ as well.
\end{proof}

Let us now present the MAX-SAT resolution rule, see e.g., \cite{abrame2014ahmaxsat}. For clauses $C_1$ and $C_2$ of some proposition $\Logical$, on literals $x$, $z_i$, $i \in [s]$ and $y_i$, $i \in [t]$ construct the clauses below the horizontal line:
\begin{equation}
    \label{eqn_maxSATresolution}
  \begin{gathered}
    C_1 = [x \vee z_1 \vee \ldots \vee z_s], \, \, C_2 = [\neg x \vee y_1 \vee \ldots \vee y_t]  \\
    \hline
    z_1 \vee \ldots \vee z_s \vee y_1 \vee \ldots \vee y_t, \\
    [C_1 \vee \neg y_1 \vee y_2 \vee \ldots \vee y_t], \, [C_1 \vee \neg y_2 \vee y_3 \vee \ldots \vee y_t], \ldots, \, [C_1 \vee \neg y_t], \\
    [C_2 \vee \neg z_1 \vee z_2 \vee \ldots \vee z_s], \, [C_2 \vee \neg z_2 \vee z_3 \vee \ldots \vee z_s], \ldots, \, [C_2 \vee \neg z_s].
  \end{gathered}
\end{equation}
The MAX-SAT resolution rule states that one may replace clauses $C_1$ and $C_2$ in $\Logical$ with the $1+s+t$ resolvent clauses below the horizontal line. We refer to the resulting new proposition, obtained after resolution, as $\Logical^\prime$. In \cite{bonet2007resolution}, Theorem 4, it is proven that any truth assignment leaves the same number of clauses unsatisfied for $\Logical$ and $\Logical^\prime$. This is referred to as \textit{soundness} of the MAX-SAT resolution rule. By soundness and the definition of $F_\Logical$, see \eqref{eqn_FbFunctionDefinition}, it follows that $F_\Logical = F_{\Logical^\prime}$.

For standard MAX-SAT solvers, one of the goals of resolution is to create new unit clauses, which are used to compute upper bounds on the MAX-SAT solution \cite{abrame2014ahmaxsat}. For our SDP approach, if given the same basis, programme \ref{eqn_SumSquaresExplicit} equals programme $\texttt{P}_{\Logical^\prime}$, as $F_{\Logical} = F_{\Logical^\prime}$. This seems to suggest that MAX-SAT resolution does not change our approach, however, we find that in general $SOS_p(\Logical) \neq SOS_p(\Logical^\prime)$, see~\eqref{eqn_sosPbasis}. We investigate the effect of this difference.

Returning to the example of $\Logical_k$ in \eqref{eqn_logicalKresolutionExample}, let us define $C_q = \neg x_q \vee x_{q+1}$. Observe that for $3 \leq q \leq k-1$, $C_q \in \Logical_k$ (assuming $k > 3$). Let us fix some $q$, $3 \leq q \leq k-3$, and consider the clauses $C_q, \, C_{q+1}, \, C_{q+2} \in \Logical_k$.
We perform resolution as:
\begin{equation}
\label{eqn_resolutionTwoClauses}
  \begin{gathered}
    C_q = [\neg x_q \vee x_{q+1} ], \, \, C_{q+1} = [\neg x_{q+1} \vee x_{q+2}]  \\ \hline   
     [\neg x_q \vee x_{q+2}], \, \, [\neg x_q \vee x_{q+1} \vee \neg x_{q+2}], \, \, [x_q \vee \neg x_{q+1} \vee x_{q+2}].
  \end{gathered}
\end{equation}
We perform resolution again, on the third new clause obtained in \eqref{eqn_resolutionTwoClauses} and $ C_{q+2}$, to obtain:
\begin{equation}
\label{eqn_resolutionTwoClausesPart2}
  \begin{gathered}
   [x_q \vee \neg x_{q+1} \vee x_{q+2}], \, \, C_{q+2} = [\neg x_{q+2} \vee x_{q+3}]  \\ \hline   
     [x_q \vee \neg x_{q+1} \vee x_{q+3}], \, \, [x_q \vee \neg x_{q+1} \vee x_{q+2} \vee \neg x_{q+3}], \\
     [\neg x_q \vee \neg x_{q+1} \vee \neg x_{q+2} \vee x_{q+3}], \, \,[x_{q+1} \vee \neg x_{q+2} \vee x_{q+3}], \, \,  
  \end{gathered}
\end{equation}
The resolution rule states that one may replace the original clauses $C_1$, $C_2$ and $C_3$ with the 6 new resolvent clauses obtained from \eqref{eqn_resolutionTwoClauses} and \eqref{eqn_resolutionTwoClausesPart2} (the third resolvent from \eqref{eqn_resolutionTwoClauses} is not counted, since it is replaced in the resolution in \eqref{eqn_resolutionTwoClausesPart2}). 

Observe that the $SOS_p$ basis generates 6 quadratic monomials for the new resolvent clauses, while originally, only 3 quadratic monomials are generated for $C_q$, $C_{q+1}$ and $C_{q+2}$. We now define, $\Logical^\prime_k$ for $k \geq 6$, 
as the logical proposition, obtained by taking $\Logical_k$, and performing resolution as in \eqref{eqn_resolutionTwoClauses} and \eqref{eqn_resolutionTwoClausesPart2}, for each triple of clauses $\{ C_q, \, C_{q+1}, \, C_{q+2} \}$, for each $q \in \{3, \, 6, \, 9, \ldots, k-3 \}$ (let us assume here that $k$ is a multiple of 3). Note that proposition $\Logical_{k}^\prime$ constitutes a MAX-4-SAT instance, and therefore basis $SOS_p$ is applicable. Let us compare the sizes of the resulting $SOS_p$ bases, denoted as $|SOS_p|$. We have
 \begin{align}
    |SOS_p(\Logical_{k})| = 2k < 3k-3 = |SOS_p(\Logical^\prime_{k})|.
\end{align}
Thus, compared to $SOS_p(\Logical_{k})$, basis $SOS_p(\Logical^\prime_{k})$ adds approximately $k$ monomials. None of these monomials strengthen the bound, since $SOS_p(\Logical_k)$ is already sufficient for proving optimality, by \Cref{lemma_sosPboundStrength}. It is clear that having a larger basis without offering a stronger bound is inefficient, since solving \ref{eqn_SumSquaresExplicit} requires more time for larger matrices.

The example of $\Logical^\prime_k$ and $\Logical_k$ shows that not all monomials are (equally) useful in determining bounds. It also shows that resolution can decrease the effectiveness of the SOS approach to the MAX-SAT, by providing `bad'\ monomial bases, or it can occur that the $SOS_p$ basis misses `good'\ monomials. Our proposed basis $SOS^\mathbf{Q}_p$, see \eqref{eqn_sosPQBasis}, attempts to solve this issue.

\section{Relating sum of squares and method of moments}
\label{section_dualRelation}
In this section, we show how the SOS-SDP relaxation of \citeauthor{van2008sums} \cite{maaren2005sums,van2008sums} and moment relaxations of \citeauthor{anjos2004proofs} \cite{anjos2004proofs}-\cite{anjos2006semidefinite} are related. The relaxations of \citeauthor{anjos2004proofs}, as described in \Cref{section_SATasSDPfeas}, were first introduced in \citeyear{anjos2004semidefinite} \cite{anjos2004semidefinite} and can be considered as extensions of the \ref{GapRelax} relaxation via the well-known Lasserre hierarchy \cite{lasserre2001global}. In \citeyear{maaren2005sums}, \citeauthor{van2008sums} \cite{maaren2005sums,van2008sums} proposed the SOS approach to the \mbox{(MAX-)SAT}. Subsequently, \citeauthor{van2008sums} \cite{van2008sums} showed that the SOS relaxation, using monomial basis $SOS_{pt}$ that is larger than $SOS_p$, see \eqref{eqn_sosPbasis},  outperforms the $R_3$ relaxation of \citeauthor{anjos2005improved} \cite{anjos2005improved}, in deciding on the satisfiability of 3-SAT instances. The  $R_3$ relaxation is known to dominate the $R_2$ relaxation, see \eqref{eqn_AnjosFset}. In \citeyear{anjos2008extended}, 
\citeauthor{anjos2008extended} \cite{anjos2008extended} strengthened his $R_3$ relaxation further and left it as future work to determine which SDP relaxation was the strongest.

We complete that work here, by showing a simple relation between the two approaches. In particular, Anjos' relaxations can be considered as method of moments in the Lasserre hierarchy. It is well-known that method of moments is dual to SOS optimization, see \cite{lasserre2001global}, and we work out the details here. Let us first derive the dual of the SOS programme~\ref{eqn_SumSquaresExplicit}, see  \Cref{thm_dualEquivalence}, and then relate it to the here proposed strengthened  version of Anjos' relaxations.

To this end, we require the following intermediate result on $v_\Logical$, see \eqref{eqn_vLogicalDef}.
\begin{lemma}
    \label{lemma_GoemansAndSOS}
        Let $\Logical = \bigwedge_{j = 1}^m C_j$ be a logical proposition, $v_\Logical=\sum_{\alpha \subseteq [n]} v^\alpha_\Logical x^\alpha$, see \eqref{eqn_vLogicalDef}, and $F_\Logical = \sum_{\alpha \subseteq [n]} p^{\alpha}_\Logical x^\alpha$, see \eqref{eqn_polynomialDef}. Then,
    \begin{align}
        v_\Logical = m - F_\Logical.
    \end{align}
    and $v_\Logical^\alpha = - p^\alpha_\Logical$ for all nonempty $\alpha \subseteq [n]$.
\end{lemma} 
\begin{proof}
    Let clause $C_j$ have length $\ell_j$. We have $v(C_j) =1 - 2^{-\ell_j} \prod_{i \in C_j} (1 - a_{j,i} x_i)$, see \eqref{eqn_valueProp}. Summing over all clauses yields the desired result.
\end{proof}
Let $\mathbf{x}$ be a given monomial basis, $S \subseteq \mathcal{S}^{|\mathbf{x}|}$. Matrix $S$ is indexed by all $\alpha \subseteq [n]$ for which $x^\alpha \in \mathbf{x}$. To simplify the comparison between the SOS approach and the relaxations of \citeauthor{anjos2004proofs} \cite{anjos2004proofs}, we define the set
\begin{align}
    \label{eqn_mathcalXset}
    \mathcal{X}_\Logical := \{ S \in \mathcal{S} \, | \, \text{diag}(S) = \mathbf{1}, \,  S_{\alpha, \beta} = S_{\alpha^\prime, \beta^\prime } \, \, \forall(\alpha,\beta,\alpha^\prime, \beta^\prime) \subseteq [n] \text{ such that } \alpha \triangle \beta = \alpha^\prime \triangle \beta^\prime \},
\end{align}
for a proposition $\Logical$ on $n$ variables and $m$ clauses. Note that $\mathcal{X}_\Logical \subseteq \cap_{j \in [m]} \Delta_j$, see \eqref{eqn_symmDiffPodd}, since $\Delta_j$ only restricts entries $S_{\alpha, \beta}$ whenever $\alpha$ and $\beta$ are jointly contained in a single clause. We use $\mathcal{X}_\Logical$ in the following theorem.

\begin{theorem}
\label{thm_dualEquivalence}
Let $\Logical$ be a logical proposition and $\mathbf{x}$ a monomial basis. The SOS programme \ref{eqn_SumSquaresExplicit} defined by $\Logical$ and $\mathbf{x}$, is equivalent to
\begin{equation}
\label{eqn_SumSquaresDual}
\begin{aligned}
\mathrm{max} \quad & \langle C, S \rangle \\
\mathrm{s.t.} \quad & S \in \mathcal{X}_\Logical \cap \mathcal{S}_+,
\end{aligned}
\end{equation}
where $\mathcal{X}_\Logical$ is given by \eqref{eqn_mathcalXset} and $C \in \mathcal{S}^{|\mathbf{x}|}$, indexed by the subsets $\alpha \subseteq [n]$ for which $x^\alpha \in \mathbf{x}$, is any matrix which satisfies
\begin{align}
    \sum_{(\alpha, \beta) \in \mathbf{x}^\gamma} C_{\alpha, \beta} =  v^\gamma_\Logical, \,\, \forall \gamma \neq \emptyset, \mathbf{x}^\gamma \neq \emptyset
\end{align}
for $v^\gamma_\Logical$ as in \eqref{eqn_vLogicalDef}. Moreover, \eqref{eqn_SumSquaresDual} is strictly feasible.
\end{theorem}

\begin{proof}
    We rewrite programme \ref{eqn_SumSquaresExplicit} by splitting the matrix variable $M$ as follows
    \begin{equation}
\label{eqn_prsmDual1}
\begin{aligned}
v := \, \, &\text{min}  && \langle I, M \rangle \\
&\textrm{s.t.} &&  M = Z, \, M \in \mathcal{M}_{\Logical}, \, Z \in \mathcal{S}_+, 
\end{aligned}
\end{equation}
where $\mathcal{M}_\Logical$ as in \eqref{eqn_mathcalMset}. We dualize the constraint $M = Z$, and set
\begin{align}
    g(S) := \min_{M \in \mathcal{M}_\Logical, \, Z \succeq 0} \langle I, M \rangle + \langle S, M - Z \rangle,
\end{align}
for some $S \in \mathcal{S}$. Clearly, $g(S) \leq v$ for all $S$, and we thus look to maximize $g(S)$, i.e.,
\begin{align}
    \max_{S} g(S) &= \max_S \Big[  \min_{M \in \mathcal{M}_\Logical}  \langle I+S, M \rangle + \min_{Z \succeq 0} \, \langle S, -Z \rangle \Big] 
    = \max_{S \preceq \, 0} \phantom{\Big[} \min_{M \in \mathcal{M}_\Logical}  \langle I+S, M \rangle \nonumber \\
    \label{eqn_dualProof1}
    &= \max_{S \succeq \, 0} \phantom{\Big[} \min_{M \in \mathcal{M}_\Logical}  \langle I-S, M \rangle.
\end{align}
We now determine the set $\mathcal{X}_\Logical$ such that, whenever $S \in \mathcal{X}_\Logical$, the minimization over $M \in \mathcal{M}_\Logical$ in \eqref{eqn_dualProof1} is bounded. Observe that $\mathcal{M}_\Logical$ places no restrictions on the diagonal. To guarantee a bounded minimum, set $\mathcal{X}_\Logical$ should restrict $\text{diag}(I-S) = \mathbf{0}$. Each off-diagonal element of a matrix in $\mathcal{M}_\Logical$ is restricted by a single constraint of the form \eqref{eqn_xMxExplicitCons}. Therefore, solving \eqref{eqn_dualProof1} for $M$ can be done by considering separately the elements of $M$ restricted by a single constraint. That is,
\begin{equation}
\begin{aligned}
&\min_{M\in \mathcal{S}}  && \sum_{\gamma \in \mathbf{X}} \, \, \sum_{(\alpha, \beta) \in \mathbf{x}^\gamma} -S_{\alpha, \beta} M_{\alpha, \beta} \\
&\textrm{s.t.}  && \sum_{(\alpha, \beta) \in \mathbf{x}^\gamma} M_{\alpha, \beta} = p^\gamma_\Logical,
\end{aligned}
\end{equation}
where $\mathbf{x}^\gamma$ and $\mathbf{X}$ are defined in \eqref{eqn_MgammaDef} and \eqref{eqn_xxTmod}, respectively. This minimization problem is bounded if and only if
\begin{align}
    \label{eqn_xGammaEquivalence}
    S_{\alpha, \beta} = S_{\alpha^\prime, \beta^\prime }, \,\, \forall (\alpha, \beta), \, (\alpha^\prime, \beta^\prime) \in \mathbf{x}^\gamma \, \, \forall x^\gamma \in \mathbf{X},
\end{align}
or equivalently, $S_{\alpha, \beta} = S_{\alpha^\prime, \beta^\prime }$ for all possible index pairs $(\alpha, \beta)$ and $(\alpha^\prime, \beta^\prime)$ that satisfy $\alpha \triangle \beta = \alpha^\prime \triangle \beta^\prime$.
It follows that $\mathcal{X}_\Logical$ is given by \eqref{eqn_mathcalXset}. Now, for fixed $S \in \mathcal{X}_\Logical \cap \mathcal{S}_+$, any matrix $M \in \mathcal{M}_\Logical$ obtains the same value in \eqref{eqn_dualProof1}. Note also that w.l.o.g., we may fix $M = \mathcal{P}_{\mathcal{M}_\Logical}(\mathbf{0})$, i.e., the projection of the zero matrix onto $\mathcal{M}_\Logical$, (see \Cref{lemma_linearProjectionLemma}) which has zero diagonal. This yields the equivalent programme of the form \eqref{eqn_SumSquaresDual}, for $C = -  M = -\mathcal{P}_{\mathcal{M}_\Logical}(\mathbf{0})$. Written explicitly,
\begin{align}
    C_{\alpha, \beta} = -\frac{p^\gamma_\Logical}{| \mathbf{x}^\gamma | }, \, \,\, \forall \alpha, \beta \subseteq [n] \text{ such that } \alpha \triangle \beta = \gamma \text{ (i.e., $(\alpha, \, \beta) \in \mathbf{x}^\gamma$)}.
\end{align}
This combined with \Cref{lemma_GoemansAndSOS}, proves the claim on matrix $C$. Lastly, observe that the identity matrix of appropriate size is strictly feasible for \eqref{eqn_SumSquaresDual}.
\end{proof}

We define, for $S \in \mathcal{X}_\Logical$ and each clause $C_j$, the function $v^{\text{SDP}}(S, C_j)$, which is obtained by taking \eqref{eqn_valueProp}, and replacing each  $x^\gamma$ by $S_{\alpha, \beta}$, for some $(\alpha, \beta) \in \mathbf{x}^\gamma$. By \eqref{eqn_xGammaEquivalence}, we are allowed to pick any such $(\alpha, \beta)$.
By \Cref{lemma_GoemansAndSOS}, for any nonempty $\gamma \subseteq [n]$, $S \in \mathcal{X}_\Logical$ and $C$ as in \Cref{thm_dualEquivalence}, we have
\begin{align}
    \sum_{(\alpha, \beta) \in \mathbf{x}^\gamma} C_{\alpha, \beta} S_{\alpha, \beta} = \sum_{(\alpha, \beta) \in \mathbf{x}^\gamma} \frac{-p^\gamma_\Logical}{|\mathbf{x}^\gamma|} S_{\alpha, \beta} = -p^\gamma_\Logical S_{\alpha, \beta} = v^\gamma_\Logical S_{\alpha, \beta}.
\end{align}
Hence, maximizing $\langle C, S \rangle$ is equivalent to maximizing the semidefinite relaxation of $v_\Logical$, see \eqref{eqn_vLogicalDef}, which equals $\sum_{j \in [m]} v^\text{SDP}(S,C_j)$.

Moreover, in the relaxations of  \citeauthor{anjos2004semidefinite} \cite{anjos2004proofs, anjos2004semidefinite, anjos2005improved, anjos2006semidefinite, anjos2008extended}, outlined in \Cref{section_SATasSDPfeas}, the matrix variable is restricted to satisfy $v^\text{SDP}(S,C_j) = 1$. Now we can easily observe the difference between the SOS-SDP relaxations and those proposed by \citeauthor{anjos2004proofs}. We present the equivalent dual formulation of the SOS approach below on the left-hand side and the latter (in slightly adapted form) on the right. \\

\noindent \begin{minipage}{.5\linewidth}
\begin{equation}
\begin{aligned}
\label{eqn_sumOfSquaresComparison}
v^*=\, &\text{max}  && \sum_{j \in [m]} v^\text{SDP}(S, C_j)  \\
&\textrm{s.t.} && S \in \mathcal{X}_{\Logical} \cap \mathcal{S}_+. \, 
\end{aligned}
\end{equation}
\end{minipage}
\begin{minipage}{.5\linewidth}
\begin{equation}
\begin{aligned}
\label{eqn_AnjosComparison}
&\text{max}  && 0 \\
&\textrm{s.t.}   && S \in \mathcal{X}_{\Logical} \cap \mathcal{S}_+, \\
& && v^\text{SDP}(S, C_j) = 1, \, \forall C_j.
\end{aligned}
\end{equation}
\end{minipage} \\

Note again the difference between \eqref{eqn_AnjosComparison} and the relaxations described in \Cref{section_SATasSDPfeas}, resulting from using set $\mathcal{X}_\Logical$ instead of the intersection of the $\Delta_j$, see \eqref{eqn_symmDiffPodd}. Thus, we compare the SOS approach with a strengthened variant of the relaxation proposed by \citeauthor{anjos2004proofs}. In \Cref{section_partialMSduality}, we determine the dual of \eqref{eqn_AnjosComparison}.

Programme \eqref{eqn_sumOfSquaresComparison} proves unsatisfiability of $\Logical$ if $v^* < m$ (with some margin of error, due to numerical precision), while \eqref{eqn_AnjosComparison} does so whenever the programme is infeasible. The above programmes are not equivalent in this sense: we have empirically found instances $\Logical$ for which $v^* \geq m$, while \eqref{eqn_AnjosComparison} is infeasible. Neither programme can directly prove satisfiability. However, solutions to both programmes can be used to guide the search towards satisfying assignments (should they exist), see \Cref{subsection_lowerBoundsAndRounding}.

If \eqref{eqn_AnjosComparison} admits a feasible matrix $S^*$, then matrix $S^*$ is clearly also feasible for \eqref{eqn_sumOfSquaresComparison} and attains an objective value of $m$. Consequently, in this case, we have $v^* \geq m$. Thus, if \eqref{eqn_AnjosComparison} does not prove unsatisfiability of $\Logical$, then neither does \eqref{eqn_sumOfSquaresComparison}. In \Cref{section_NumericalResults} we show that \eqref{eqn_sumOfSquaresComparison} can be computed efficiently by applying the PRSM to its dual\footnote{Programme \eqref{eqn_sumOfSquaresComparison} can also be directly solved with the PRSM, as projecting onto $\mathcal{X}_\Logical$ is computationally cheap.}. It is currently unclear whether a good algorithm for solving \eqref{eqn_AnjosComparison} exists, and if so, how efficient it would be. Previous numerical experiments on \eqref{eqn_AnjosComparison} have used general purpose SDP solvers. An immediate improvement might be to use an SDP feasibility problem solver, see \cite{drusvyatskiy2017note, henrion2011projection}.

 Lastly, the objective value of \eqref{eqn_sumOfSquaresComparison} is more useful for the \mbox{MAX-SAT}: if the underlying instance is infeasible, $v^*$ provides an upper bound to the number of satisfiable clauses, which is useful in a B\&B scheme. Programme \eqref{eqn_AnjosComparison} might also show unsatisfiability of the same instance, but its infeasibility offers no additional value, as to \textit{how} unsatisfiable the instance is.

\section{The Peaceman-Rachford splitting method for the MAX-SAT}
\label{section_PRSMsection}
In this section, we introduce the \textit{Peaceman-Rachford splitting method} 
\cite{peaceman1955numerical} for solving SOS-SDP problems and apply it to the \mbox{MAX-SAT} SOS programme \ref{eqn_SumSquaresExplicit}. Conventionally, interior point methods are used to solve SDP problems. However, for medium and large size instances, interior point methods suffer from a large computation time and memory demand, which has recently motivated researchers to consider first order methods, such as the PRSM. For recent applications of PRSM to SDP, see e.g.,~\cite{de2021sdp, graham2022restricted}.

\Cref{subsection_SosSatUpperBounds} and \Cref{subsection_lowerBoundsAndRounding}
provide details on obtaining valid upper and lower bounds,  from the output of the PRSM algorithm.

\subsection{The PRSM for SOS relaxations of the MAX-SAT}
We start from the reformulation of \ref{eqn_SumSquaresExplicit} given in \eqref{eqn_prsmDual1}. The augmented Lagrangian function of~\eqref{eqn_prsmDual1} w.r.t.~the constraint $M = Z$
for a penalty parameter $\beta>0$ is:
\begin{align}
    \label{eqn_augmentedLagrangian}
    L_\beta(Z,M,S) = \langle I, M \rangle + \langle S, M - Z \rangle + \frac{\beta}{2} \| M - Z \|^2_F.
\end{align}
Here,  $S \in \mathcal{S}^n$ is the Lagrange multiplier and $\| \cdot \|_F$ denotes the Frobenius matrix norm, see \eqref{eqn_FrobeniusNorm}.

The PRSM now entails iteratively optimizing over the variables $Z$ and $M$ separately, and updating $S$ twice per cycle. We write superscript $k$ to denote the the value of the variable at iteration $k$. 
\begin{equation}
\label{eqn_PRSMiterates}
\left\lbrace\,
\begin{array}{@{}r@{\quad}l@{}l@{}}
&Z^{k+1} &= \argmin_{Z \succeq 0} L_\beta(Z,M^k,S^k) \hspace{4em} = \mathcal{P}_{\mathcal{S}_+} \left (M^k + \frac{1}{\beta}S^k \right ),\\
    &S^{k+1/2} \, &= S^k + \gamma_1 \beta (M^k - Z^{k+1}),\\
    &M^{k+1} &= \argmin_{M \in \mathcal{M}_\Logical} L_\beta(Z^{k+1},M,S^{k+1/2}) = \mathcal{P}_{\mathcal{M}_\Logical} \left ( Z^{k+1} - \frac{1}{\beta} \big[ I + S^{k+1/2} \big] \right ), \\
    &S^{k+1} &= S^{k+1/2} + \gamma_2 \beta (M^{k+1} - Z^{k+1}).
\end{array}
\right.
\end{equation}
Here, $\mathcal{M}_\Logical$ is as in \eqref{eqn_mathcalMset}, and $\mathcal{P}$ is the projection operator as in \eqref{eqn_ProjectionOperatorDef}. We have used that
\begin{equation}
\begin{aligned}
\label{eqn_argminAdaptedPRSM}
    \argmin_{Z \succeq 0} L_{\beta}(Z,M,S) &= \argmin_{Z \succeq 0} ~\langle I, M \rangle  - \frac{1}{2 \beta} \Big\| S \Big\|^2_F + \frac{\beta}{2} \Big\| Z - (M + \frac{1}{\beta} S) \Big\|^2_F \\
    &= \argmin_{Z \succeq 0}  ~\frac{\beta}{2} \Big\| Z - (M + \frac{1}{\beta} S) \Big\|^2_F,
\end{aligned}
\end{equation}
see e.g.,~\cite{oliveira2018admm}. In an implementation of \eqref{eqn_PRSMiterates}, one should not store matrix $S^k$ directly, but rather, the matrix $\frac{1}{\beta}S^k$, see \Cref{subsection_compactPRSM}.
When $X \in \mathcal{S}$ has eigenvalues $\lambda_i$, and corresponding eigenvectors $v_i$, it is well known 
that the projection onto the positive semidefinite cone is given by
\begin{align}
    \label{eqn_psdConeProjection}
    \mathcal{P}_{\mathcal{S}_+}(X) = \sum_{\{ i \, | \, \lambda_i > 0 \} } \lambda_i v_i v_i^\top = X - \sum_{\{ i \, | \, \lambda_i < 0 \} } \lambda_i v_i v_i^\top. 
\end{align}
Depending on the number of positive eigenvalues of $X$, one of the above expressions will be cheaper to compute.
The next lemma shows how to compute a projection onto $\mathcal{M}_\Logical$.
\begin{lemma}
\label{lemma_linearProjectionLemma}
Let matrices $M, \, \widehat{M} \in \mathcal{S}$, indexed by subsets of $[n]$, such that $\widehat{M} =  \mathcal{P}_{\mathcal{M}_\Logical}( M )$, where the projection operator $\mathcal{P}_{\mathcal{M}_\Logical}( \cdot )$ is given by \eqref{eqn_ProjectionOperatorDef}. Then $\text{diag}(\widehat{M}) = \text{diag}(M)$ and
\begin{align}
    \label{eqn_lemmaProjection1}
    \widehat{M}_{\delta,\mu} = M_{\delta,\mu} - \frac{1}{| \mathbf{x}^\gamma|} \bigg( \sum_{(\alpha, \beta) \in \mathbf{x}^\gamma} M_{\alpha, \beta} - p_\Logical^\gamma \bigg),
\end{align}
for $(\delta,\mu) \in \mathbf{x}^\gamma$, $\gamma \neq \emptyset$. In particular, when $|\mathbf{x}^\gamma| = 2$, \eqref{eqn_lemmaProjection1} reduces to
\begin{align}
    \label{eqn_lemmaSimplification}
    \widehat{M}_{\delta,\mu} = \widehat{M}_{\mu,\delta} = p^\gamma_\Logical / 2.
\end{align}
\end{lemma}

\begin{proof}
Let $M \in \mathcal{S}$. To find $\widehat{M} =  \mathcal{P}_{\mathcal{M}_\Logical}( M )$, note that in $\mathcal{M}_\Logical$, each off-diagonal entry is restricted by exactly one constraint. This follows from \eqref{eqn_mathcalMset}. Since $\mathcal{M}_\Logical$ does not restrict the diagonal, it is easily seen that $\text{diag}(\widehat{M}) = \text{diag}(M)$. Now for the off-diagonal entries, we fix a nonempty $\gamma \subseteq [n]$ and define $\mathbf{m}$ as the vector that contains upper triangular entries of $M$, $M_{\alpha, \beta}$, such that $(\alpha, \beta) \in \mathbf{x}^\gamma$. Similarly, we define $\widehat{\mathbf{m}}$ as the vector containing the same entries of matrix $\widehat{M}$, rather than $M$. Note that $\mathbf{1}^\top \widehat{\mathbf{m}} = p^\gamma_\Logical / 2$. Minimizing the Frobenius norm of $\widehat{M} - M$ is now equivalent to minimizing the norm of $\widehat{\mathbf{m}} - \mathbf{m}$. Thus, we solve
\begin{align}
    \widehat{\mathbf{m}} = \argmin_{\mathbf{1}^\top v = p^\gamma_\Logical / 2} \| v - \mathbf{m} \|^2,
\end{align}
which can be done analytically and leads to \eqref{eqn_lemmaProjection1}. The simplification of \eqref{eqn_lemmaProjection1} to \eqref{eqn_lemmaSimplification} follows from the equality $\sum_{(\alpha, \beta) \in \mathbf{x}^\gamma} M_{\alpha, \beta} = 2 M_{\delta, \mu}$, whenever $|\mathbf{x}^\gamma| = 2$ and $(\delta, \mu) \in \mathbf{x}^\gamma$.
\end{proof}
Due to the presence of many unit constraints, see \eqref{eqn_unitCons}, these projections are computationally cheap to compute, and hence, the PRSM is well suited to exploit this.
Lastly, it is proven \cite{he2016convergence} that \eqref{eqn_PRSMiterates} converges for $(\gamma_1, \, \gamma_2) \in \mathcal{D}$, where
\begin{align}
    \label{eqn_mathcalD_definition}
    \mathcal{D} = \biggl\{  (\gamma_1, \gamma_2) \, \Big| \, \gamma_1 + \gamma_2 > 0, \, | \gamma_1 | < 1, \, 0 < \gamma_2 < \frac{1+\sqrt{5}}{2},  \, |\gamma_1| < 1 + \gamma_2 - \gamma_2^2 \biggr\}.
\end{align}
The values that we choose for $(\gamma_1, \gamma_2)$, and other parameters, are given in \Cref{section_NumericalResults}.

\subsection{Upper bounds and early stopping}
\label{subsection_SosSatUpperBounds}
After each PRSM iterate $k$ we obtain a triple $(Z^k,M^k,S^k)$ and the resulting $\langle I,M^k\rangle.$
Although this value converges to the optimal solution of the SDP, the convergence is (typically) not monotonic and therefore this value does not necessarily provide a valid upper bound for the problem.
In this section we describe how to obtain a valid upper bound from the output of the PRSM.

Observe that the feasible set of \ref{eqn_SumSquaresExplicit} depends on the chosen monomial $\mathbf{x}$ through $\mathcal{V}_{\mathbf{x}}$, see \eqref{eqn_VxSet}. Hence, by \eqref{eqn_PolySOS}, we have
\begin{align}
    p^\emptyset_\Logical - \min_{M \in \mathcal{M}_\Logical \cap \mathcal{S}_+} \langle I, M \rangle 
    =  \sup_{\lambda \in \mathbb{R}} \{ \lambda \, | \, F_\Logical - \lambda \in \mathcal{V}_{\mathbf{x}} \} \leq \min_{x \in \{ \pm 1 \}^n } F_\Logical,
\end{align}
for $p^\emptyset_\Logical$ as in \eqref{eqn_constantTermF}. From the above it follows that the maximum number of satisfiable clauses of $\Logical$ is bounded from above by
\begin{align}
    \label{eqn_upperBoundOptimal}
    m - p^\emptyset_\Logical + \min_{M \in \mathcal{M}_\Logical \cap \mathcal{S}_+} \langle I, M \rangle,
\end{align}
for $m$ equal to the number of clauses in $\Logical$. Since the number of satisfied clauses is an integer, the bound \eqref{eqn_upperBoundOptimal} can be improved by rounding down the result.

Ideally, the PRSM algorithm \eqref{eqn_PRSMiterates} computes the upper bound \eqref{eqn_upperBoundOptimal} by finding the optimal $M$ in the set $\mathcal{M}_\Logical \cap S_+$ (up to some given numerical precision). However, in practice one terminates 
the PRSM algorithm before this optimal $M$ has been found.
Let matrix $M^k$ then be defined as in \eqref{eqn_PRSMiterates} and let $\lambda_\text{min}(M^k)$ be its smallest eigenvalue. Note that 
\begin{align}
    \label{eqn_ADMMmatrixApprox}
    \widetilde{M}^k = M^k - \lambda_\text{min}(M^k) I \in \mathcal{M}_\Logical \cap \mathcal{S}_+,    
\end{align}
and so, $\widetilde{M}^k$ is feasible for \ref{eqn_SumSquaresExplicit}. Thus, a valid upper bound at iteration $k$ is obtained as follows:
\begin{align} 
\label{trueUB}
    \left \lfloor  m - p^\emptyset_\Logical + \langle I, \widetilde{M}^k \rangle \right \rfloor.    
\end{align}

\subsection{Lower bounds and rounding}
\label{subsection_lowerBoundsAndRounding}

In order to obtain a truth assignment of the variables
from the output of the PRSM one needs a rounding procedure.  We describe here the rounding procedure proposed by \citeauthor{van2008sums} \cite{van2008sums} and a modification of the procedure that is implemented in our solver.

Let matrix $M^*$ be the optimal solution to the SOS programme \ref{eqn_SumSquaresExplicit}, induced by a logical proposition $\Logical$ on $n$ variables. Let $\mathbf{x}$ be its monomial basis of size $s$, and $\lambda^*$ such that $F_\Logical(x) - \lambda^* = \mathbf{x}^\top M^* \mathbf{x}$. It is clear that, by optimality of $M^*$, $\lambda_{\text{min}}(M^*) = 0$. Let $N$ be the multiplicity of the zero eigenvalue, and $\mathbf{v}_i$, $i \in [N]$ the corresponding eigenvectors. If $y \in \{ \pm 1 \}^n$ is an optimal \mbox{MAX-SAT} truth assignment of $\Logical$, then $y$ minimizes $F_\Logical$. Let $y'$ be the monomial basis vector $\mathbf{x}$, evaluated with the entries of $y$. Then $y'^\top M^* y' = F_\Logical(y) - \lambda^*$. If the SOS relaxation \ref{eqn_SumSquaresExplicit} computes the optimal bound, we have $\lambda^* = F_\Logical(y)$, which implies that $y' M^* y' = 0$. As the eigenvectors $\mathbf{v}_i$ satisfy the same relation, they can be considered as approximations of maximally satisfying assignments.

Let $V \in \mathbb{R}^{s \times N}$ be the matrix having the vectors $\mathbf{v}_i$, $i \in [N]$ as columns. Each row of $V$ corresponds to a monomial of $\mathbf{x}$. For $p$ the number of monomials in $\mathbf{x}$ of degree two or more, matrix $B \in \mathbb{R}^{p\times N}$ is the submatrix of $V$ obtained by taking the rows of $V$ corresponding to these $p$ monomials. We define $U \in \mathbb{R}^{N \times N}$ as the matrix with columns the eigenvectors of $B^\top B$.

The rounding procedure proposed by \citeauthor{van2008sums} is to compute $x_\lambda \in \{ \pm 1\}^n$ as 
\begin{align}
    \label{eqn_simpleRounding}
    x_\lambda = \sign (P_1) \begin{bmatrix}
    \mathbf{0}_{n\times 1} & I_n & \mathbf{0}_{n \times (s-n-1)}
    \end{bmatrix} \sign (P ), \text{ for } P = V (U \tilde{\lambda}) \text{ and } \tilde{\lambda}_i = \xi_i \lambda_i \, \forall i \in [N].
\end{align}
Here, $\sign(\cdot)$ is the sign operator and $\lambda \in \mathbb{R}^N$ a vector generated uniformly at random on the unit sphere (which allows us to perform multiple roundings by generating multiple $\lambda$). Observe that $P_1$, the first entry of vector $P$, corresponds to the monomial $x^\emptyset$. The vector $\xi \in \mathbb{R}^N$ is a parameter to be chosen. We refer to \cite{van2008sums} for the details. 

The optimal matrix $M^*$ is a low  rank matrix, which ensures that $N$, the multiplicity of eigenvalue 0, satisfies $N > 1$. In practice however, we do not find $M^*$, but its approximation $\widetilde{M}^k$ at some iteration, see \eqref{eqn_ADMMmatrixApprox}. Due to early stopping, matrix $\widetilde{M}^k$ often has eigenvalue 0 with multiplicity 1. Then $\lambda$ is a scalar, which does not affect \eqref{eqn_simpleRounding}. Thus, when $N = 1$, we can only perform one rounding. To solve this issue, we propose constructing $V$ with the columns of $q$ eigenvectors corresponding to the $q$ smallest eigenvalues of $\widetilde{M}^k$ (only in case $N < q$). Thus, whenever $N < q$, we add $q- N$ eigenvectors corresponding to nonzero eigenvalues to the matrix $V$, in order to perform multiple roundings. In \cite{van2008sums}, it is observed that the rounding procedure works better when $N$ is small. We use this information by setting $q = 4$, so that we take at least four vectors for the rounding procedure.

\section{The weighted partial MAX-SAT}
\label{section_wpMS}
In this section, we extend the SOS approach from the MAX-SAT to  the weighted partial MAX-SAT. We also show that the dual formulation of the SOS programme for certain partial MAX-SATs, equals the relaxations by \citeauthor{anjos2004semidefinite}~\cite{anjos2005improved}.

In the \textit{weighted} \mbox{MAX-SAT}, each clause is given a weight, and the objective is to maximize the sum of the weights of the satisfied clauses. In  the \textit{partial} MAX-SAT, clauses are divided in soft and hard clauses. The aim is to maximize the number of satisfied soft clauses, while satisfying all the hard clauses. The combination of the weighted and partial MAX-SAT is clear, and referred to as the \textit{weighted partial} MAX-SAT \cite{li2021maxsat}.

Consider again a logical proposition $\Logical$. Let $w_j \in \mathbb{R}$ be the weight associated to clause $C_j$. The generalization of \eqref{eqn_FbFunctionDefinition} for the (unweighted) MAX-SAT, to  the weighted MAX-SAT follows  by setting
\begin{align}
    \label{eqn_Fb_weighted}
    F_{\phi}^\mathcal{W}(x) =  \sum_{j = 1}^m \frac{w_j}{2^{\ell_j}}\prod_{i \in C_j}(1-a_{j,i} x_i ),
\end{align}
and then minimizing $F_{\phi}^\mathcal{W}$ for $x \in \{ \pm 1 \}^n$. This minimization can be approximated by SOS optimization, using directly the semidefinite programme \ref{eqn_SumSquaresExplicit}.

For the weighted partial MAX-SAT, consider a logical proposition $\Logical$, on $n$ variables, $m$ soft clauses $C_j$ and $q$ hard clauses $C_p^\mathbf{H}$. To each hard clause $C^\mathbf{H}_p$, $p \in [q]$, we associate the polynomial \mbox{$f_p = \prod_{i \in [n]} ( 1- a_{p,i}x_i)$}, 
similar to \eqref{eqn_Fb_weighted}. Note that $f_p$ vanishes for all truth assignments that satisfy clause $C^\mathbf{H}_p$.
Similar to \eqref{eqn_Putinar} and \eqref{eqn_VxSet}, we define the sets
\begin{align}
    \label{eqn_mathcalHset}
     \mathcal{H}_{\mathbf{x}}:=\Big\{ \sum_{p \in [q]} c_p f_p \mod I, \, c_p \in \mathbb{R} \,\, \forall p \in [q] \Big\}
     \subseteq
     \mathcal{H} :=\Big\{   \sum_{p \in [q]} g_p f_p \mod I, \, g_p \in \mathbb{R}[x] \,\, \forall p \in [q] \Big\}. 
\end{align}
Let $\mathbf{SAT} \subseteq \{ \pm 1 \}^n$ be the set of all truth assignments satisfying the hard clauses (which we assume to be nonempty). From \cite{putinar1993positive} it follows that 
\begin{align}
    \label{eqn_wpMS_putinar}
    \min_{x \in \mathbf{SAT}} F^{\mathcal{W}}_\Logical(x) = \sup\{ \lambda \, | \, F^{\mathcal{W}}_\Logical - \lambda \in \mathcal{V} + \mathcal{H} \} \geq \sup\{ \lambda \, | \, F^{\mathcal{W}}_\Logical - \lambda \in \mathcal{V}_\mathbf{x} + \mathcal{H}_\mathbf{x} \},
\end{align}
where `+'\ denotes the Minkowski sum of sets. We proceed by writing the lower bound in \eqref{eqn_wpMS_putinar} as an explicit SDP, for which we introduce the following sets
\begin{align}
    \mathbf{H}^\gamma := \{ p \in [q] \, | \, \gamma \in C^{\mathbf{H}}_p \},
\end{align}
for $\gamma \subseteq [n]$. Set $\mathbf{H}^\gamma$ contains all $p$ for which $f_p$, when expanded, contains the term $\pm x^\gamma$. The sign here is determined by the parity of $| \gamma \cap I^+_p |$, see \eqref{eqn_IplusDef}. Additionally, we define as analogue to $\mathcal{M}_\Logical$, see \eqref{eqn_mathcalMset}, the set $\mathcal{M}^{\mathbf{H}}_\Logical$. This set contains all matrices $M$ and vectors $c$ such that $F^\mathcal{W}_\Logical - \lambda \equiv \mathbf{x}^\top M \mathbf{x} + \sum_{p \in [q]} c_p f_p \mod I$. It is therefore defined as
\begin{align}
    \label{eqn_mathcalMHlogical}
    \mathcal{M}^{\mathbf{H}}_\Logical := \Big\{ (M,c) \in \mathcal{S} \times \mathbb{R}^q \, | \, \sum_{(\alpha, \beta) \in \mathbf{x}^\gamma} M_{\alpha, \beta} + \sum_{p \in \mathbf{H}^\gamma} (-1)^{| \gamma \cap I^+_p |} c_p= p^\gamma_\Logical, \, \forall \gamma \neq \emptyset \text{ such that } x^\gamma \in \mathbf{X} \Big\}.
\end{align}
This allows us to adapt \ref{eqn_SumSquaresExplicit} to the weighted partial MAX-SAT as follows:
\begin{equation}
\label{eqn_wpMS_SDPformulation}
\begin{aligned}
\text{min} \quad & \langle I, M \rangle + \sum_{p \in [q]} c_p \\
\textrm{s.t.} \quad & (M,c) \in \mathcal{M}^\mathbf{H}_\Logical, \, M \in \mathcal{S}_+.
\end{aligned}
\end{equation}
We 
approximately solve \eqref{eqn_wpMS_SDPformulation} by the PRSM, see \Cref{section_pMaxSatSolutionMethod}.
 Let us elaborate on how to adapt the monomial bases to the (weighted) partial MAX-SAT. We make no distinction between soft and hard clauses for the $SOS_p$ basis, see \eqref{eqn_sosPbasis}. For basis $SOS_s^\theta$, see \eqref{eqn_monBasisInclusionChain}, we determine the variable weights as $w(i) := \sum_{ \{j \, | \, i \in C_j \} } w_j + \sum_{ \{ p  \, | \, i \in C^{\textbf{H}}_p \}} \overline{w}$, for $\overline{w}$ the mean of all soft clause weights $w_j$. For basis $SOS_p^{\mathbf{Q}}$, we add all $\binom{\mathbf{Q}}{2}$ quadratic terms of the $\mathbf{Q}$ variables that attain the highest value of $w(i)$. For unweighted partial MAX-SAT instances, we consider all $w_j$ to equal 1.

\subsection{The PRSM for SOS of the weighted partial MAX-SAT}
\label{section_pMaxSatSolutionMethod}
We show here how to solve \eqref{eqn_wpMS_SDPformulation} by the PRSM.  We first rewrite~\eqref{eqn_wpMS_SDPformulation}  by introducing the matrix variable $Z$, see also~\eqref{eqn_prsmDual1},
\begin{equation}
\label{eqn_wpMS_SDPformulationRew}
\begin{aligned}
\text{min} \quad & \langle I, M \rangle + \sum_{p \in [q]} c_p \\
\textrm{s.t.} \quad & M=Z, ~(M,c) \in \mathcal{M}^\mathbf{H}_\Logical, \, Z \in \mathcal{S}_+.
\end{aligned}
\end{equation}
Then,  the augmented Lagrangian function of~\eqref{eqn_wpMS_SDPformulationRew} w.r.t.~$Z=M$ and for a penalty parameter $\beta>0$ is:
\begin{align}
    L_\beta(Z,M,S,c) = \langle I, M \rangle + \langle S, M - Z \rangle + \frac{\beta}{2} \| M - Z \|^2_F + \mathbf{1}^\top c.
\end{align}
The PRSM is iteratively and separately optimizing over $(M,c) \in  \mathcal{M}^{\mathbf{H}}_\Logical$ and  $Z \in \mathcal{S}_+$, and updating $S$ twice per cycle, similarly to~\eqref{eqn_PRSMiterates}. However, in this case,   the $M$-subproblem from~\eqref{eqn_PRSMiterates} is replaced by the  $(M,c)$-subproblem.

We now show that minimization over $(M,c) \in  \mathcal{M}^{\mathbf{H}}_\Logical$ can be performed efficiently. By derivations similar to \eqref{eqn_argminAdaptedPRSM}, we have:
\begin{align}
    \label{eqn_argMinMC}
     \argmin_{(M,c) \in  \mathcal{M}^{\mathbf{H}}_\Logical} L_{\beta}(Z,M,S,c) &= \argmin_{(M,c) \in  \mathcal{M}^{\mathbf{H}}_\Logical} \| M - X \|^2_F + \frac{2}{\beta} \mathbf{1}^\top c,
\end{align}
where $X := Z - (S+I)/ \beta$. This is a convex quadratic programme  (QP) that we solve in two steps.
Firstly, consider the matrix-entries $M_{\alpha, \beta}$, with $(\alpha, \beta) \in \mathbf{x}^\gamma$, see \eqref{eqn_MgammaDef}, and $\mathbf{H}^\gamma = \emptyset$. Since $\mathbf{H}^\gamma = \emptyset$, these entries are unaffected by the $c_p$ variables. This implies that these $M_{\alpha, \beta}$ variables are not coupled with the other entries of $M$, and one can minimize separately over such $M_{\alpha, \beta}$. This separate minimization problem can be solved by applying \Cref{lemma_linearProjectionLemma}.

Secondly, the remaining QP
\begin{align}
    \label{eqn_extendedQP}
    \min_{} \sum_{ \{ \gamma \, | \, \mathbf{H}^\gamma \neq \emptyset \} } \sum_{(\alpha, \beta) \in \mathbf{x}^\gamma }(M_{\alpha, \beta} - X_{\alpha, \beta} )^2 + \frac{2}{\beta} \mathbf{1}^\top c,
\end{align}
can be simplified by the following observation. If $M^*$ is an optimal solution to \eqref{eqn_argMinMC}, then
\begin{align}
    (\alpha, \beta), (\alpha^\prime, \beta^\prime) \in \mathbf{x}^\gamma \Longrightarrow M^*_{\alpha, \beta} - X_{\alpha, \beta} = M^*_{\alpha^\prime, \beta^\prime} - X_{\alpha^\prime, \beta^\prime}.
\end{align}
Hence, \eqref{eqn_extendedQP} can be simplified by substituting each term $\sum_{(\alpha, \beta) \in \mathbf{x}^\gamma }(M_{\alpha, \beta} - X_{\alpha, \beta} )^2$ with a single squared variable. We solve the resulting QP either  by solving the KKT conditions using the LU decomposition,  or via MOSEK \cite{mosek}. The solving method depends on the underlying QP.

\subsection{Strengthening the bounds}
\label{subsection_SATresolutionRule}
We demonstrate a simple technique for improving the upper bounds given by programme \eqref{eqn_wpMS_SDPformulation}. This technique is based on the SAT resolution rule, which is given as follows. For two hard clauses of some proposition $\Logical$, on literals $x$, $z_i$, $i \in [s]$ and $y_i$, $i \in [t]$, construct the  clause below the horizontal line:
\begin{equation}
    \label{eqn_SATresolution}
  \begin{gathered}
    [x \vee z_1 \vee \ldots \vee z_s], \, \, [\neg x \vee y_1 \vee \ldots \vee y_t]  \\
    \hline
    z_1 \vee \ldots \vee z_s \vee y_1 \vee \ldots \vee y_t.
  \end{gathered}
\end{equation}
In contrast to the MAX-SAT resolution rule \eqref{eqn_maxSATresolution}, the SAT resolution rule states that one may add the clause below the horizontal line to $\Logical$, without changing its (un)satisfiability (we say that the new clause is \textit{implied} by the original two clauses). We may apply this SAT resolution rule to the hard clauses of a partial MAX-SAT instance to generate more hard clauses. As each new clause induces a new variable $c_p$, the bound of programme \eqref{eqn_wpMS_SDPformulation} can only improve. One may also regard SAT resolution as extending the set $\mathcal{H}_\mathbf{x}$, see \eqref{eqn_mathcalHset}, by including terms of the form {$c_px^\alpha f_p$, for some $\alpha \subseteq [n]$}  where $c_p\in \mathbb{R}$.

Additionally, SAT resolution can generate hard unit clauses. This is advantageous, since hard unit clauses reduce the number of variables in the MAX-SAT, see \Cref{subsection_Assumptions}.

\subsection{Duality in the partial MAX-SAT}
\label{section_partialMSduality}
Now we consider a partial MAX-SAT with only hard clauses. Solving such instances is thus equivalent to determining the satisfiability of the given hard clauses. We show that by taking the dual of the resulting SOS programme, one obtains (a stronger version of) the relaxations of \citeauthor{anjos2004semidefinite} \cite{anjos2004semidefinite}, given by \eqref{eqn_AnjosComparison}.

We define, for $A \in \mathcal{S}^n$, $\text{vec}(A) \in \mathbb{R}^{n^2}$ the vector whose entries are the columns of $A$ stacked together. We start from programme \eqref{eqn_wpMS_SDPformulation} and perform variable splitting on $M$, similar to \eqref{eqn_prsmDual1}. We take the dual $g(S)$ of this formulation, similar to \eqref{eqn_dualProof1}, and consider the problem
\begin{align}
    \label{eqn_wpMSdualMax}
    \max_{S} g(S) &= \max_{S \in \mathcal{X}_\Logical \cap \mathcal{S}_+} \phantom{\Big[} \min_{(M,c) \in \mathcal{M}^\mathbf{H}_\Logical}  \langle S, -M \rangle + \mathbf{1}^\top c,
\end{align}
for $\mathcal{X}_\Logical$ as in \eqref{eqn_mathcalXset}. The steps that show that $S \in \mathcal{X}_\Logical \cap \mathcal{S}_+$ is necessary for the above expression to be finite are provided in the proof of \Cref{thm_dualEquivalence}. We rewrite the inner minimization problem in \eqref{eqn_wpMSdualMax} as
\begin{align}
    \label{eqn_dualVecRewritten}
    \min_{(M,c) \in \mathcal{M}^\mathbf{H}_\Logical}  \begin{bmatrix}
        \text{vec}(S) \\ \mathbf{1}
    \end{bmatrix}^\top
    \begin{bmatrix}
        \text{vec}(-M) \\ c
    \end{bmatrix},
\end{align}
and proceed to show under which conditions this value is bounded. Observe that the coefficients $p^\gamma_\Logical = 0$, see \eqref{eqn_mathcalMHlogical}, since there are no soft clauses. Moreover, the set $\mathcal{M}^\mathbf{H}_\Logical$ places only linear constraints on the entries of $M$ and $c$. Therefore, there exists a matrix $D$ that satisfies
\begin{align}
    (M,c) \in \mathcal{M}^\mathbf{H}_\Logical \iff D \begin{bmatrix}
        \text{vec}(-M) \\ c
    \end{bmatrix} = \mathbf{0}.
\end{align}
Hence, \eqref{eqn_dualVecRewritten} is bounded if and only if $\begin{bmatrix}
\text{vec}(S)^\top & \mathbf{1}^\top
\end{bmatrix}$ is contained in the row space of $D$. This is precisely the requirement that $v^\text{SDP}(S, C^{\mathbf{H}}_p) = 1, \, \forall p \in [q]$, as in \eqref{eqn_AnjosComparison}. We provide one example of this claim. 

\begin{example}
Consider the monomial basis $\mathbf{x} = (x^\emptyset,x_1,x_2)$. Let $C^{\mathbf{H}}_1 = x_1 \vee x_2$, so that $f_1 = 1-x_1-x_2+x_1 x_2$. Now 
\begin{align}
    &(M,c) \in \mathcal{M}^\mathbf{H}_\Logical \Longrightarrow 
    D u
      = \mathbf{0}, \text{ for }
      D = \begin{bmatrix}
        -1 & -1&0 &0& 0 &0& -1 \\
        0&0 & -1 &-1& 0 &0& -1 \\
        0&0&0&0&-1&-1&1
    \end{bmatrix}, \\
    &\text{ and } u = \begin{bmatrix}
        -M_{1,\emptyset} & -M_{\emptyset,1} &
        -M_{2,\emptyset} & -M_{\emptyset,2} &
        -M_{1,2} & -M_{2,1} & c_1
    \end{bmatrix}^\top.
\end{align}
For $S \in \mathcal{X}_\Logical$, and by definition of $\mathcal{X}_\Logical$ \eqref{eqn_mathcalXset}, we have $S_{1,\emptyset} = S_{\emptyset,1}$, and similar equalities hold for all other related entries of matrix $S$. Thus, we may remove duplicate columns in $D$. We have
\begin{align}
    \begin{bmatrix}
        S_{1,\emptyset} & S_{2,\emptyset} & S_{1,2} & 1
    \end{bmatrix} \in \mathrm{row} \begin{bmatrix}
-1 & 0&0  & -1 \\
 0& -1 & 0 & -1 \\
 0&0  & -1 & 1 
\end{bmatrix} \Longrightarrow S_{1,\emptyset} + S_{2,\emptyset} - S_{1,2} = 1 \Longrightarrow v^\mathrm{SDP}(S, C^{\mathbf{H}}_1) = 1.
\end{align}
\end{example}
Thus, \eqref{eqn_dualVecRewritten} is bounded if $\begin{bmatrix}
\text{vec}(S)^\top & \mathbf{1}^\top
\end{bmatrix} \in \text{row}(D)$, in which case, the value equals zero. Hence, programme \eqref{eqn_wpMSdualMax} is equivalent to \eqref{eqn_AnjosComparison}.

\section{SOS-MS: Algorithm description}
\label{section_algorithmDescr}

In this section, we elaborate on the algorithm behind our complete SOS-SDP based  MAX-SAT solver, named \SOSms{}. In particular, we outline the main parts of \SOSms{} and provide a pseudocode, see~\Cref{alg_SosSdpMaxSat}.

Consider a given MAX-$k$-SAT instance, $k \leq 3$, and corresponding logical proposition $\Logical$. \SOSms{} uses the PRSM, see \eqref{eqn_PRSMiterates}, to obtain an approximate solution $\widetilde{M}$, see \eqref{eqn_ADMMmatrixApprox}, to \ref{eqn_SumSquaresExplicit}. Then, using \eqref{trueUB}, \SOSms{} determines an upper bound \texttt{UB}  and a lower bound  on the optimal value, by applying the rounding procedure from \Cref{subsection_lowerBoundsAndRounding}.
The solver calls once, at the beginning,  the \heurMaxSolve{}\footnote{The \heurMaxSolve{} algorithm is publicly available at \url{http://lcs.ios.ac.cn/~caisw/MaxSAT.html}.} algorithm \cite{CCLS} for the MAX-SAT, to compute a lower bound. \heurMaxSolve{} is a local search algorithm whose performance was one of the best among tested heuristic algorithms in the MSE-2016.
We set \texttt{LB} as the maximum value attained by these two methods.

In case $\texttt{LB} = \texttt{UB}$, we have proven optimality and the algorithm terminates. In case $\texttt{LB} < \texttt{UB}$, we branch on some variable $x_i$, $i \in [n]$, by assigning it either \true{} or \false{}. This resembles  to performing unit resolution, see ~\Cref{subsection_Assumptions}. We write $\Logical^\prime = \texttt{unitRes}(\Logical, i )$ to indicate that $\Logical^\prime$ is the logical proposition obtained from $\Logical$ by setting $x_i = 1$ (equivalently, $x_i = \true{}$). We use the same notation to indicate
the logical proposition obtained from $\Logical$  by 
setting $x_i = -1$ (equivalently, $x_i = \texttt{false}$), i.e., ${\Logical^\prime =\texttt{unitRes}(\Logical, -i )}$. To emphasize the difference between $\Logical$ and $\Logical^\prime$, in this section, we write $n_\Logical$ and $m_\Logical$ for the number of variables and clauses of $\Logical$.

If we branch on $x_i$, we remove from the monomial basis $\mathbf{x}$ all monomials $x^\alpha$ that satisfy $i \in \alpha$. We remove from matrices $Z^k, \, M^k$ and $S^k$,  that were obtained in the last relevant call of the PRSM,
all rows and columns corresponding to such subsets $\alpha$.
The resulting matrices are then used as the new $Z^0, \, M^0$ and $S^0$ in the next PRSM call, i.e., those are used as a warm start.

We determine the order of variables for branching as follows. First, we consider for $b \in \{ -n, \ldots,n\} \setminus \{ 0 \}$, the values
\begin{align}
    \label{eqn_orderValues}
     u_b = m_{\Logical^\prime} - p^\emptyset_{\Logical^\prime} + \langle I, \widetilde{M} \rangle - \sum_{   |b| \in \alpha  }\widetilde{M}_{\alpha, \alpha}, \text{ for } \Logical^\prime = \texttt{unitRes}(\Logical, b ),
\end{align}
similar to \eqref{eqn_upperBoundOptimal}. Here, $\widetilde{M}$ is the approximate solution to \ref{eqn_SumSquaresExplicit}. We remark that \eqref{eqn_orderValues} can be quickly computed without explicitly performing the unit resolution. Observe that $\langle I, \widetilde{M} \rangle - \sum_{   |b| \in \alpha  } \widetilde{M}_{\alpha, \alpha}$ equals the trace of the new matrix $M^0$, which is used as warm start for $\texttt{P}_{\Logical^\prime}$.

Second, we perform the rounding procedure to $\widetilde{M}$ as described in \Cref{subsection_lowerBoundsAndRounding}. That is, we randomly generate a set $\Lambda$ of  vectors drawn uniformly at random on the unit sphere, and compute the corresponding rounded truth assignments $x_{\lambda} \in \{ \pm 1 \}^{n_\Logical}$, $\lambda \in \Lambda$, by \eqref{eqn_simpleRounding}. 
We update \texttt{LB} if a better truth assignment is found.
Let $\Lambda^* \subseteq \Lambda$ contain all the vectors $\lambda$ that satisfy $F_\Logical(x_\lambda) = m_\Logical - \texttt{LB}$, and, assuming $\Lambda^* \neq \emptyset$, set $\mathbf{v} := \frac{1}{|\Lambda^*|}\sum_{\lambda \in \Lambda^*} x_\lambda$. Note that $-1 \leq \mathbf{v}_i \leq 1$ $\forall i \in [n]$  with equality if and only if $x_i$ is assigned the same truth value for all $x_\lambda$, $\lambda \in \Lambda^*$. We define
\begin{align}
    B := \left \{ b \, \, | \,\, 0 < |b| \leq n_\Logical, \,\, b \in \mathbb{Z}, \,\, \mathbf{v}_{|b|} = - \text{sgn}(b) \right \},
\end{align}
and explain its purpose by an example. If $-3 \in B$, then $x_3$ is assigned \true{} by all $x_\lambda$, $\lambda \in \Lambda^*$. Heuristically, branching by setting $x_3 = -1$ would then hopefully lead to low  upper bounds in the resulting search tree. This is advantageous because low upper bounds lead to faster pruning. In case $\Lambda^* = \emptyset$, we set $B = \{ -n, \ldots,n\} \setminus \{ 0 \}$.

Lastly, for all $b \in B$, we sort them in increasing order of $u_b$, see \eqref{eqn_orderValues}, and store this order in vector $\sigma$. Thus, the entries of $\sigma$ satisfy
\begin{align}
    \label{eqn_sigmaVec}
    u_{\sigma_j} \leq u_{\sigma_{j+1}} \text{ and } \sigma_j \in B.
\end{align}
Vector $\sigma$ determines the variable selection in the branching process of \SOSms{}, which we describe in more detail in the sequel. 

Consider a node in the \SOSms{}  search tree, in which we consider the proposition $\Logical$. We initialize $b_* := 0$. For increasing values of $j \in [b_{\text{max}}]$, where $b_{\text{max}} > 1$ is some integer (see \eqref{eqn_bMax}), we compute the SDP upper bound corresponding to $\texttt{unitRes}(\Logical, \sigma_j)$, denoted \texttt{UB}.
In case $\lfloor \texttt{UB} \rfloor \leq \texttt{LB}$, we repeat this process with the next value of $\sigma_j$, and update $b_* := b_*+1$. In case $\lfloor \texttt{UB} \rfloor > \texttt{LB}$, we terminate the process. 

In case $b_* > 0$, we find that for all $j \leq b_*$, the propositions $\texttt{unitRes}(\Logical, \sigma_j)$ cannot improve on \texttt{LB}.
Thus, we may limit the search for better truth assignments to $\texttt{unitRes}(\Logical, -\sigma_1, \ldots, - \sigma_{b_*} )$. In case $b_* = 0$, we add both $\texttt{unitRes}(\Logical, \sigma_1)$ and $\texttt{unitRes}(\Logical, -\sigma_1)$ to the search tree, as we cannot exclude either one from attaining a value strictly greater than $\texttt{LB}$.

We take the previously mentioned $b_\text{max}$ as
\begin{equation}
    \label{eqn_bMax}
    \begin{aligned}
    b_\text{max} = \min \left \{ \max \left\{ 3; \lfloor 6 \, \texttt{GAP} + 1/2 \rfloor \right \}; 6 \right \}, \text{ for } \texttt{GAP} = m_{\Logical^\prime} - p^\emptyset_{\Logical^\prime} + \langle I, \widetilde{M} \rangle - \texttt{LB}-1, \\[-20pt]
    \,
    \end{aligned}
\end{equation}
where $\widetilde{M}$ is an approximate solution to \ref{eqn_SumSquaresExplicit}.

A  pseudocode of \SOSms{} is given by \Cref{alg_SosSdpMaxSat}. In particular, the branching process is described in \Cref{line_loop1,line_PrsmCall2,line_pruneCondition,line_loop2,line_loop3,line_loop4}. Note the two PRSM calls in \Cref{line_PrsmCall1,line_PrsmCall2}. In \Cref{line_PrsmCall1}, the main purpose of the PRSM is to find  an upper bound which equals the best known lower bound. When this does not occur, we use the approximate solutions as warm start for the PRSM call in \Cref{line_PrsmCall2}. In \Cref{line_PrsmCall2}, the purpose of the PRSM is to prune the node corresponding to $\Logical^\prime$. We use the LOBPCG algorithm \cite{knyazev2001toward} to efficiently approximate $\lambda_\text{min}(M^k)$, which allows us to compute approximate upper bounds during the PRSM iterates, see \eqref{eqn_upperBoundOptimal} and \eqref{eqn_ADMMmatrixApprox}. In case the approximate upper bound indicates that the node can be pruned, we recompute $\lambda_\text{min}(M^k)$ with the more accurate MATLAB \texttt{eig} function.

The algorithm can then stop iterating as soon as the condition in \Cref{line_pruneCondition} is satisfied, or when it is clear that this condition cannot be satisfied in reasonable time.

\begin{remark}
\label{remark_evDecomp}
Most of the running time of \SOSms{} is spent on computing projections of matrices onto $\mathcal{S}_+$, see \eqref{eqn_psdConeProjection}. We found that computing the full spectra of the matrices to be projected (in single-precision, rather than standard double-precision) trough MATLAB's \texttt{eig} command was the fastest method of computing $\mathcal{P}_{\mathcal{S}_+}(\cdot)$, even though only the positive eigenpairs, or negative eigenpairs are required. For a variant of the PRSM, the authors of \cite{rontsis2022efficient} propose using the LOBPCG algorithm~\cite{knyazev2001toward} to compute only positive/negative eigenpairs (when this number is deemed small enough) of matrices to be projected. We were unable to obtain a speedup over \texttt{eig} through this method in the PRSM framework.
\end{remark}

\begin{algorithm}[ht!]
\caption{SOS-SDP MAX-SAT Solver}\label{alg_SosSdpMaxSat}
\textbf{Input:} A MAX-$k$-SAT instance $\Logical$ ($k 
\leq 3$), on $n_\Logical$ variables and $m_\Logical$ clauses.
\label{alg_lineTest} \\
\textbf{Output:} Optimal truth assignment $x \in \{ \pm 1 \}^n$. \\
Set $\texttt{UB} :=m_\Logical$. \\
Use the \heurMaxSolve{} algorithm on $\Logical$ to obtain a value for current best lower bound \texttt{LB} and corresponding truth assignment $x \in \{ \pm 1 \}^n$. \\
Initialize the stack $Q := ( \Logical, \texttt{UB} )$. \\
\While{$Q \neq \emptyset$ \label{line_branchCondition}}{
Take $(\Logical, \texttt{UB} )$ as the first element of $Q$. \label{line_initialBranch} \\
Use the PRSM, see \eqref{eqn_PRSMiterates}, to obtain an approximate solution $\widetilde{M}$, see \eqref{eqn_ADMMmatrixApprox}, to \ref{eqn_SumSquaresExplicit}. \label{line_PrsmCall1} \\
 Apply the rounding procedure from \Cref{subsection_lowerBoundsAndRounding} to $\widetilde{M}$ and obtain a lower bound $\texttt{LB}$. Update $\texttt{LB}$ and $x$ if a better truth assignment has been found. \\
 Set $\texttt{UB} := \min\{ \texttt{UB}, \lfloor m_\Logical - p^\emptyset_\Logical + \langle I, \widetilde{M} \rangle \rfloor \}$, see \eqref{eqn_upperBoundOptimal}. \\
 Remove $(\phi, \texttt{UB})$ from the stack $Q$. \\
\If{$\texttt{LB} \geq \texttt{UB}$}{
    {\color{blue} \tcc{The current node $(\Logical, \texttt{UB})$ can be pruned, so return to \Cref{line_branchCondition}, and check if Q is empty.} }
    \textbf{continue}
    }

Determine $\sigma$, see \eqref{eqn_sigmaVec}, and $b_{\text{max}}$, see \eqref{eqn_bMax}. Set $b_* := 0$.\\
\For{$j = 1 \texttt{ to } b_{\text{max}}$ \label{line_loop1}}{
        Set $\Logical^\prime := \texttt{unitRes}(\Logical, \sigma_j )$, use the PRSM to obtain an approximate solution  $\widetilde{M}$ to $\texttt{P}_{\Logical^\prime}$. \label{line_PrsmCall2} \\
        {\color{blue} \tcc{Use as warm start the matrices obtained from the programme \ref{eqn_SumSquaresExplicit} in \Cref{line_PrsmCall1}.} }
    \uIf{$\texttt{LB} \geq \lfloor m_{\Logical^\prime} - p^\emptyset_{\Logical^\prime} + \langle I, \widetilde{M} \rangle \rfloor$ \label{line_pruneCondition}}
        {
            $b_* := b_*+1$. \label{line_loop2}
        }
    \Else{ \label{line_loop3}
        \textbf{break} \label{line_loop4}
    }}
    \uIf{$b_* > 0$}{
    Set $\Logical := \texttt{unitRes}(\Logical, - \sigma_1, - \sigma_2, \ldots, -\sigma_{b_*})$, and $Q := Q \cup (\Logical, \texttt{UB} )$. \label{line_splitProp1} \\
    {\color{blue} \tcc{Q is nonempty, so return to \Cref{line_branchCondition}.} }
    }
    \Else{
    Set $\Logical_1 := \texttt{unitRes}(\Logical, \sigma_1)$, $\Logical_2 := \texttt{unitRes}(\Logical, -\sigma_1)$ and $Q := Q \cup \{ (\Logical_1, \texttt{UB} ), (\Logical_2, \texttt{UB} ) \}$. \label{line_splitProp2} \\
    {\color{blue} \tcc{Q is nonempty, so return to \Cref{line_branchCondition}.} }
    }
    }
    \Return Optimal truth assignment $x$.
\end{algorithm}

\subsection{Parsing sum of squares programmes} 
\label{section_Parsing}
 We cover here the problem of initializing an SOS semidefinite programme. Methods for achieving this are built in most SOS packages, such as SOSTOOLS \cite{sostools} and GlobtiPoly \cite{henrion2009gloptipoly}. In our application, we are interested in SOS modulo a vanishing ideal, which is not natively implemented in most SOS software, but rather, by restriction of the support set of the variables to a semialgebraic set (such as $x^2 = 1$), which incurs additional variables in the semidefinite programme.

For most SDP applications, it is implicitly assumed that the programme parameters are either already given, or require a negligible time to compute, in comparison to the time required for solving the resulting semidefinite programme.
For most SOS programmes however, this is decidedly not the case. The authors of SOSTOOLS \cite{sostools}, a third-party MATLAB package for formulating and solving SOS programmes, confirm this observation. In chapter one of the user's manual to SOSTOOLS \cite{sostools}, it is stated that defining the semidefinite programme, rather than solving it, is often the limiting factor for tractable problem size. In accordance with \cite{sostools}, we thus consider the problem of \textit{parsing} an SOS programme as defining it by its programme parameters. This definition of parsing includes the problem of choosing the monomial basis $\mathbf{x}$ such that the given polynomial can be accurately described. Many theoretical results can guide the choice of $\mathbf{x}$, such as the Newton polytope \cite{sturmfels1998polynomial}, see also \cite{reznick1978extremal}, or facial reduction \cite{permenter2014basis}. For our purposes however, choosing $\mathbf{x}$ can be done with a simple fixed procedure, see e.g., \eqref{eqn_sosPbasis}. We therefore consider parsing as only the purely numerical problem of finding, rather than choosing, $\mathbf{x}$.

In this section, we provide a short overview of our parsing method. We exploit the fact that our variables are  $\{ \pm 1 \}^n$, which allows us to achieve fast parsing times, compared to general purpose SOS software. For example, due to the properties of computation modulo $I$, see \eqref{eqn_modIdef}, monomials can be stored in two ways: either we store some $\alpha \subseteq [n]$, corresponding to $x^\alpha$ as in \eqref{eqn_xAlphaDef}, (\textit{subset format}) or we store monomials as a (possibly sparse) vector $\mathbf{v} \in \{0, 1 \}^n$, corresponding to $x^\mathbf{v} = x_1^{v_1} \ldots x_n^{v_n}$ (\textit{vector format}). Such vectors can be saved as sparse boolean vectors. It is trivial to switch between these two formats, which is what we use in our parsing algorithm: we implement each step using the best suited format. Also note that for monomial basis $\mathbf{x}$ given by \eqref{eqn_sosPbasis}, monomials in $\mathbf{X} \equiv \mathbf{x} \mathbf{x}^\top $, see \eqref{eqn_xxTmod}, will have degree at most 4, which ensures that both formats require little storage.

Let $\Logical$ be the considered proposition, on $n$ variables and $m$ clauses. Initially, to compute $\mathbf{x}$ according to \eqref{eqn_sosPbasis}, we consider the unique clauses $C_j$ of $\Logical$, $j \in [m]$. Recall that we define a clause $C_j$ as a subset of $[n]$, see \Cref{section_Notation}.
To compute the monomials in $\mathbf{X}$, we need to compute the cross products of all monomials in $\mathbf{x}$. This is best done in vector format, by using the entrywise \textit{exclusive or} operation on boolean vectors, denoted $\oplus$. That is, $x^\mathbf{v}x^\mathbf{u} \equiv x^\mathbf{v \oplus u } \mod I$, which is computationally cheaper than the symmetric difference operator as in \eqref{eqn_xxTmod}.

Next, to determine the sets $\mathbf{x}^\gamma$, as in \eqref{eqn_MgammaDef}, we need to find the sets of equal monomials in $\mathbf{X}$. We first split up the monomials in groups based on their degree, which is trivially computed in vector format as $\mathbf{v}^\top \mathbf{1}$. Then for the monomials of degree one, we switch to subset format and find the groups of equal monomials based on these one element subsets. For degrees two and three, this procedure is similar, except we further divide these groups in $n$ smaller subgroups, based on what the first nonzero entry in their vector $\mathbf{v}$ is. If we consider the $SOS_p$ basis, see \eqref{eqn_sosPbasis}, then for non-trivially sized instances, most of the monomials in $\mathbf{x} \mathbf{x}^\top$ will be of degree 4. Therefore, we divide the monomials of degree 4 in $n^2$ subgroups, based on the first and second nonzero entry in their vector $\mathbf{v}$. Then, in each of these subgroups, we switch to subset format, which yields a matrix of four columns, and each row associated to a single monomial. The first two columns of this matrix are fixed by which subgroup we consider; hence we only evaluate the third and fourth columns of this matrix. In these columns, we then search for unique rows, which represent unique monomials.

We compute the coefficients $p^\gamma_\Logical$, see \eqref{eqn_polynomialDef}, iteratively per clause. If $k$ is the size of the largest clause in $\Logical$, we create $k+1$  matrices $A_s$, for $0 \leq s \leq k$, where matrix $A_s$ is an $s$-dimensional matrix. Then, for $\gamma = \{ \gamma_1, \ldots, \gamma_s \}$, we store $p^\gamma_\Logical$ at position $(A_s)_{\gamma_1, \ldots, \gamma_s}$. Note that matrix $A_0$ is a number storing $p^\emptyset_\Logical$, see \eqref{eqn_constantTermF}. For \mbox{(MAX-)3-SAT} instances, matrices $A_1$ and $A_2$ will be full, matrix $A_3$ will (generally) be sparse, and matrix $A_4$, although we do not create it,  would be empty. We have also tested the recently developed \texttt{dpvar} structure of SOSTOOLS \cite{jagt2022efficient} to compute the $p^\gamma_\Logical$ symbolically, but found that it compared unfavourably to our procedure, in terms of computation time.

The last step is to match the coefficients $p^\gamma_\Logical$ to the monomial $x^\gamma$ of $\mathbf{X}$. Clearly, we need only to  consider those coefficients $p^\gamma_\Logical$ for which $p^\gamma_\Logical \neq 0$. Thus, for some nonzero $p^\gamma_\Logical$ stored in $A_s$, we know that $| \gamma | = s$. As the monomials in $\mathbf{X}$ have already been divided into groups based on their degree, we search only the monomials of degree $s$, to find $x^\gamma$. In case we use the $SOS_p$ basis, then by \Cref{lemma_unitConstraintsCoeff}, we need only check those monomials $x^\alpha$ for which $| \mathbf{x}^\alpha | > 2$, see \eqref{eqn_MgammaDef}. Now to find $x^\gamma$ in this subgroup of monomials (of which one equals $x^\gamma$), we use the vector format. By exploiting properties of the ideal $I$, see \eqref{eqn_modIdef}, and $SOS_p$, we are able to obtain high parsing speeds. For example, instance \texttt{s3v70c1500-1.cnf} (70 variables and 1500 clauses of length 3) from the MSE-2016\footnote{All instances are available at \url{http://www.maxsat.udl.cat/16/benchmarks/index.html}.},  induces an $SOS_p$ basis of size 2107. Matrix $\mathbf{X}$, see \eqref{eqn_xxTmod}, then contains 4,439,449 monomials. Our algorithm parses this basis in (approximately) 1 second.

This completes the summary of our parsing algorithm. For more details, interested readers are referred to the code available on \texttt{GitHub}\footnote{Code available at \gitHubLink{}.}.

\section{Numerical results}
\label{section_NumericalResults}

In this section, we test \SOSms{}, described in \Cref{section_algorithmDescr},  on  MAX-3-SAT,  weighted partial MAX-2-SAT, and weighted MAX-3-SAT instances from the MSE-2016. We use instances from the same source to compute bounds for the partial MAX-3-SAT.
We choose the year 2016, because later years of the MSE offer no MAX-2-SAT or MAX-3-SAT instances. The instances in this section are taken from the MSE-2016\footnote{For more information on the MSE-2016, see \url{http://maxsat.ia.udl.cat/introduction/}.}random track.

Experiments are carried out on a 16 GB RAM laptop, with an Intel Core i7-1165G7 (2.8 GHz) and four cores, running Windows 10 Enterprise. We set the PRSM parameters, see \eqref{eqn_PRSMiterates}, as
\begin{align}
    \label{eqn_gammaParameter}
    (\gamma_1, \, \gamma_2, \, \beta) = \Big(\frac{1}{2}, \, \frac{9}{10} \frac{1+\sqrt{5}}{2}, \,  
    \frac{2s}{5} \Big),
\end{align}
where $s$ is the order of the matrix variables. Our MATLAB implementation of the PRSM is available at \gitHubLink{}.

\subsection{MAX-3-SAT instances}
\label{subsect_numericsMax3SAT}

In this section we first show a relation between upper bounds obtained by solving ~\ref{eqn_SumSquaresExplicit} and $SOS_p^{\mathbf{Q}}$ bases~\eqref{eqn_sosPQBasis}. Then, we demonstrate the performance of \SOSms{}.

\Cref{table_basisBoundsComparison} presents a comparison of the  bounds  obtained   by solving the SOS programme \ref{eqn_SumSquaresExplicit}  using the $SOS_p$ basis~\eqref{eqn_sosPbasis} and the $SOS_p^{\mathbf{Q}}$ bases~\eqref{eqn_sosPQBasis} where $\mathbf{Q} \in \{40,50,60,70, 110\}$, for several instances in category MAX-3-SAT on 70, 90 and 110 variables. The first column reports the name of the corresponding instance, on $\texttt{v}$ variables and \texttt{c} clauses.  Column \texttt{LB} provides the best found lower bound, obtained by running the CCLS algorithm~\cite{CCLS} for five seconds. Column \texttt{UB} reports the computed upper bounds. Column \textbf{Iter.}~gives the number of PRSM iterations, divided by $10^2$. For each instance and monomial basis, we run 200 iterations. Then, if the observed value of \texttt{UB} satisfies $\texttt{UB} \leq \texttt{LB}+1.5$, we perform 200 additional iterations. We stop early if $\lfloor \texttt{UB} \rfloor = \texttt{LB}$. The results show the strength of the $SOS_p$ basis for instances with with 70 variables. In particular, that basis is sufficiently large for closing gaps of several instances. The results also show that the $SOS_p^{\mathbf{Q}}$ basis can be used to further improve bounds for instances with 70, 90 and 110 variables.

 In \Cref{table_bestUBperInstance}, we present more details on the best upper bounds attained on the same instances as in  \Cref{table_basisBoundsComparison}. The columns of \Cref{table_bestUBperInstance} follow the same definitions as the previous table. Additionally, column $\mathbf{Q}$ relates to the $SOS_p^\mathbf{Q}$ basis used
  and column $| \mathbf{x} |$ reports the number of monomials in that basis (equivalently, the order of the matrix variable of programme \ref{eqn_SumSquaresExplicit}). Here, we refer to $SOS_p$ as $SOS_p^0$. Column $\textbf{T.~(s)}$ reports the computation time in seconds. The results show that we closed the optimality gap for all instances with 70 and 90 variables in less than 9 minutes. \Cref{table_bestUBperInstance} also shows that the computational time increases w.r.t.~the size of the basis. Furthermore, \Cref{table_bestUBperInstance} shows that uniform random MAX-SAT instances on the same number of variables and clauses can differ in difficulty to solve. For example, instances \texttt{s3v70c800-1.cnf} and \texttt{s3v70c800-3.cnf} have the same number of variables and clauses. However, proving optimality of the lower bound of the former requires almost three times the computation time as for the latter.

 In \Cref{picture_cactusPlotS3V70}, we show the performance of \SOSms{}, using the $SOS_p^{50}$ basis, on all the MAX-3-SAT instances on 70 variables from the MSE-2016. For detailed running times see  \Cref{table_70varM3SAT} in \Cref{appendix_runtimesPerInstance}. We compare the running times of \SOSms{} with the corresponding running times of the best results of the MSE-2016. That is, for each instance, we compare \SOSms{} with the participant of the MSE-2016 that was able to solve that specific instance in the least time.
All solvers at the MSE-2016 were tested on an Intel Xeon E5-2620 processor with 2.0 GHz and 3.5~GB RAM\footnote{The full specifications of this machine are available at \url{http://maxsat.ia.udl.cat/machinespecifications/}.}. It is hard to compare the running times of algorithms on different machines, since it depends on many factors such as the processor, RAM, the operating system, et cetera. We choose to consider the difference in clock speeds, i.e., 2.8 GHz vs.~2.0 GHz, as well as wall  time. Therefore, in \Cref{picture_cactusPlotS3V70} and \Cref{table_70varM3SAT}, we have multiplied all the original running times of \SOSms{} with 1.4. 
Additionally, as the solvers in the MSE-2016 were given a maximum time of 30 minutes per instance, we provided \SOSms{} with a maximum of $30/1.4 ( \approx 21.4)$ minutes.
 
 Under these constraints, \SOSms{} is able to solve all 45 instances in category MAX-3-SAT on 70 variables. The participants from the MSE-2016 could solve at most 42 instances. Specifically, they were unable to solve \texttt{s3v70c1500-1.cnf}, \texttt{s3v70c1500-4.cnf} and \texttt{s3v70c1500-5.cnf}. For instances with a lower number of clauses (around 800) however, the best times per instance of the MSE-2016 are much lower than for \SOSms{}. \SOSms{} takes on average 309.10 seconds per solved instance, compared to  367.7 seconds per solved instance for the best MSE-2016 solvers. 
  We have also tested \SOSms{} using the $SOS_p$ basis, on the same instances. With this different basis, \SOSms{} was also able to solve all 45 instances, taking on average 316.7 seconds per instance.

 We also investigate the performance of \SOSms{} on 80 variable MAX-3-SAT.  We consider instances from the MSE-2016 database, although the MSE-2016 did not test 80 variable MAX-3-SAT, and so, it is not known what the best solver per instance. However, we compare \SOSms{} to the CCLS2akms MAX-SAT algorithm. This algorithm first runs CCLS \cite{CCLS} to find a good starting lower bound for the MAX-SAT solution. It then passes this lower bound to the akmaxsat algorithm \cite{kuegel2012improved}, which solves the instance to optimality. Out of all the publicly available solvers, CCLS2akms\footnote{The CCLS2akms algorithm is available at \url{http://www.maxsat.udl.cat/16/solvers/index.html}.} placed highest in the random MAX-SAT category of MSE-2016\footnote{The MSE-2016 results are available at \url{http://maxsat.ia.udl.cat/docs/ms.pdf}.}. 

Results for the 80 variable MAX-3-SAT instances from MSE-2016  are given in \Cref{picture_cactusPlotS3V80}.
The running times per tested MAX-3-SAT instance are provided in \Cref{appendix_runtimesPerInstance}, \Cref{table_80varM3SAT_instances}. Both \SOSms{} and CCLS2akms are provided a maximum of 30 minutes per instance and are tested on the same hardware (our 16 GB RAM laptop), and thus not scaled. For \SOSms{}, we used the $SOS_p$ basis, as $SOS_p^{\mathbf{Q}}$ for $\mathbf{Q} = 40$ and $\mathbf{Q} = 55$ provided worse results. CCLS2akms is able to solve 40 of the 48 instances within the time limit, while \SOSms{} solves 43. \SOSms{} is however slower: it requires on average 772.53 seconds per solved instance, compared to 370.48 per instance. Again, \SOSms{} performs well for instances with a large number of clauses, but requires more time for those with a low number of clauses.

\subsection{Weighted partial MAX-2-SAT instances}
\label{subsection_wpM2S}
We show the performance of our solver on (weighted) partial MAX-2-SAT instances from MSE-2016. For this purpose, we adjust  \SOSms{} by using the  theory outlined in~\Cref{section_wpMS}.

In particular, we perform a B\&B search, using \eqref{eqn_wpMS_SDPformulationRew} to compute upper bounds. 
However, we first preprocess an instance by using the SAT resolution rule \eqref{eqn_SATresolution} on the hard clauses until we find all implied hard clauses of length two or less. Note that this preprocessing may result in  improved upper bounds, as described in  \Cref{subsection_SATresolutionRule}. If hard unit clauses are found this way, we perform unit resolution and continue with the reduced problem. As initial lower bound for our solver, we take the best known lower bound reported in  MSE-2016.

 Note that in case of a (weighted) partial MAX-SAT instance, truth values assigned during branching might create hard unit clauses, which leads to more forced truth assignments.  Additionally, if a node contains few unassigned variables, determining good upper bounds can be done with a small monomial basis, such as $SOS_s^\theta$ \eqref{eqn_monBasisInclusionChain} for small values of $\theta$. Let us describe the choice of $\theta$ through the B\&B tree. We initialize $\theta_{\text{start}} = 0$. At a node in which we compute bounds, we compute an upper bound $\texttt{UB}$ to the (weighted) partial MAX-SAT solution by first using the basis $SOS_s^{\theta_{\text{start}}}$. If $\lfloor \texttt{UB} \rfloor \leq \texttt{LB}$, we prune the current node. If not, we consider the value $\texttt{GAP} = \texttt{UB} - \texttt{LB} > 0$. When $\texttt{GAP} < 10$, we set $\theta = 0.5$ and recompute a stronger upper bound. For $\texttt{GAP} \geq 10$, we recompute an upper bound with $\theta = 1$ instead. In both cases, we use the PRSM variables corresponding to $\theta_{\text{start}}$ as warm starts for the next PRSM. If the upper bound obtained using $\theta \in \{ 0.5,1\}$ is  not equal \texttt{LB}, we set $\theta_{\text{start}} = 0.1$ for the remainder of the algorithm. In \Cref{subsection_searchTreeWpM2S} we demonstrate our branching rule and basis selection on illustrative  examples, see \Cref{fig_searchTree1,fig_searchTree2,fig_searchTree3}.

 In case the tighter upper bound (corresponding to either $\theta = 0.5$ or $\theta = 1$)  does not equal $\texttt{LB}$, we branch at this node. We determine our branching variable in the following way.  For $n^\prime$ ($n^\prime\leq n)$ the number of unassigned variables at the current node, we consider the five truth assignments that create the largest number of hard unit clauses.

 For each of these five truth assignments $\sigma_i$, $i \in [5]$, $\sigma_i \in \{ -n^\prime, \ldots, n^\prime \} {\setminus \{0 \}}$, we compute $\Logical_i = \texttt{unitRes}(\Logical, \sigma_i)$, see \Cref{section_algorithmDescr}. Lastly, we select the truth assignment $\sigma_i$ for which the polynomial $\Logical_i$ has the highest constant term, see \eqref{eqn_constantTermF}. The branching variable is then given by $| \sigma_i |$.  This branching rule aims to create nodes with many assigned variables, due to the hard unit clauses. This allows for setting $\theta_{\text{start}}$ to small values, while still providing strong bounds, see also \Cref{subsection_searchTreeWpM2S}. 
  
The computation of SOS-SDP based upper bounds is expensive, compared to bounding methods used in other MAX-SAT solvers. Therefore, we only compute bounds at selected nodes (see also \cite{wang2019low}). Our selection process is described in detail in \Cref{appendix_branchingRulesWPM2S}.

 We test the described procedure on the 60 unweighted partial MAX-2-SAT, and the 90 weighted partial MAX-2-SAT instances from the MSE-2016, setting a maximum time of 30 minutes per instance. Each instance contains 150 variables and 150 hard clauses. The number of total clauses (both soft and hard) ranges from 1000 to 5000,  and all of them have length two. In the weighted variant, soft clause weights range from 1 up to and including 10. \Cref{table_pM2S_runTimes,table_wpM2S_runTimes} report the running times per instance (rounded to the nearest second), for unweighted and weighted partial MAX-2-SAT, respectively. Here we provide original runtimes, thus not multiplied by some factor. A \mbox{`-'\ } value indicates a time-out of 30 minutes. Variable $m$ denotes the number of total clauses. The $\textbf{Instance}$ row corresponds to the instance file name, as taken from the MSE-2016. For example, $m = 2500$ and $\textbf{Instance} = 1$ refer to \texttt{file\_rpms\_wcnf\_L2\_V150\_C2500\_H150\_1.wcnf} in \Cref{table_pM2S_runTimes}, and \texttt{file\_rwpms\_wcnf\_L2\_V150\_C2500\_H150\_1.wcnf} in \Cref{table_wpM2S_runTimes}.

The table shows that we are able to solve many instances within the 30-minute time limit. This shows the strength of SDP applied also to the (weighted) partial MAX-SAT. Since we are first to solve the (weighted) partial MAX-SAT by using SDP approaches, our work opens new perspectives on solving variants of the MAX-SAT.

\subsection{Partial MAX-3-SAT}
\label{section_pM3S}
We show the quality  of the SDP bounds for partial MAX-3-SAT, based on 10 instances from the MSE-2016. Each instance contains 500 soft clauses and 100 hard clauses, all of length three.

We compute an upper bound for each instance $\Logical$ (on $n$ variables, having hard clauses $C^{\textbf{H}}_p$) in the following way. We first perform SAT resolution \eqref{eqn_SATresolution} on the hard clauses to find all implied hard clauses of length 4 or less. Then, for $\mathbf{Q} \in \mathbb{N}$, we consider the $\mathbf{Q}$ variables that appear in the largest number of (soft and hard) clauses.
Let $V \subseteq [n]$ be the subset indicating those variables. We construct additional hard clauses of the form $C^{\textbf{H}}_p \vee x_i$, for all $i \in V$ and $C^{\textbf{H}}_p \subseteq V$, with $|C^{\textbf{H}}_p| \leq 3$. Note that these additional hard clauses do not change the set of satisfying assignments. On so generated instance we compute an upper bound using the $SOS_p^\mathbf{Q}$ basis, see \eqref{eqn_sosPQBasis}. Note that, as the newly generated clauses $C^{\textbf{H}}_p \vee x_i$ are contained in $V$, they create no additional monomials in the $SOS_p^\mathbf{Q}$. This ensures that the size of the matrix variable remains manageable. Moreover, these additional hard clauses strengthen the bound, see \Cref{subsection_SATresolutionRule}.

We perform this procedure for each instance and $\mathbf{Q} \in \{ 70,75,80,85,90 \}$, for a time limit of 30 minutes. These values of $\mathbf{Q}$ are chosen with the goal of computing the tightest bound at the 30-minute mark. We report results in \Cref{table_pM3S}. Column $\textbf{Inst.}$ reports the instance file name, according to the naming scheme \texttt{file\_rpms\_wcnf\_L3\_V100\_C600\_H100\_}[\textbf{Inst.}]\texttt{.wcnf}. Each instance has 600 clauses, of which 100 are hard, all of length three, on 100 variables. Column $\texttt{LB}$ reports the optimal lower bound, as verified by solvers in the MSE-2016. Column $\mathbf{Q}$ refers to the $SOS_p^\mathbf{Q}$ basis used, of which the number of monomials is reported in column $| \mathbf{x} |$. Column $|C^{\textbf{H}}|$ reports the number of hard clauses used. The next 6 columns (\texttt{GAP} at $\mathbf{X}$ minutes) report the value of the \texttt{GAP} (i.e., \texttt{UB} - \texttt{LB}) at different time points. To compute the values of \texttt{UB}, we compute the smallest eigenvalue of $M$ using the LOBPCG algorithm \cite{knyazev2001toward}, see \eqref{eqn_ADMMmatrixApprox}. Lastly, as a measure of convergence, the final column reports the (absolute) value of the smallest eigenvalue of the matrix variable at the final iteration, multiplied by the size of this matrix. This column thus reports the difference in trace between $\widetilde{M}$ and $M$, see \eqref{eqn_ADMMmatrixApprox}.

Considering the bound at 30 minutes, the $SOS_p^{85}$ basis performs best. A larger basis is unable to converge in 30 minutes, while smaller bases result in  weaker bounds. Since differences in bounds for different $\mathbf{Q}$ are small,   smaller bases  might be more useful in combination with a B\&B scheme.

\subsection{Weighted MAX-3-SAT}
\label{subsection_wM3S}
Lastly, we test \SOSms{} for some weighted MAX-3-SAT instances. 

We consider 10 weighted MAX-3-SAT instances from the MSE-2016. Each instance contains 70 variables, and either 1400 or 1500 weighted soft clauses. These weights range between 1 and 10. There are no hard clauses. 

For the weighted MAX-3-SAT, we compare the running times of \SOSms{} with CCLS2akms,  on the same hardware. For \SOSms{}, we attempted multiple monomial bases, and found that the $SOS_p$ basis required the least time to solve the weighted MAX-3-SAT instances. 

The running times per instance are reported in \Cref{table_wM3S}. Column $m$  reports the number of clauses, and \textbf{Inst.}~reports the instance, according to the scheme \texttt{s3v70c}[$m$]\texttt{-}[\textbf{Inst.}]\texttt{.wcnf}.
\SOSms{} is able to solve three instances in less time than CCLS2akms and can solve 9 of the 10 instances in less than 30 minutes. This demonstrates that \SOSms{} is well suited for solving weighted MAX-3-SAT instances with a large number of clauses.

\section{Conclusions and future work}
\label{section_Conclusion}

In this paper we consider SOS optimization for solving the MAX-SAT and weighted partial MAX-SAT. We design an SOS-SDP based exact MAX-SAT solver, called  SOS-MS.  Our  solver is competitive with the  best-known solvers on solving various (weighted partial) MAX-SAT instances. We are also first to compute SDP bounds for the  weighted partial MAX-SAT.

In \Cref{section_SATasSDPfeas} we propose a family of semidefinite feasibility problems \ref{eqn_RFprogramme} and show that one member of this family provides the rank two guarantee, see \Cref{thm_NewFsetRank2Guarantee}. That is, 
the existence of a feasible rank two matrix implies satisfiability of the corresponding SAT instance.
In \Cref{section_SOSandMaxSat}, we outline  the SOS approach to the MAX-SAT, due to \citeauthor{van2008sums} \cite{van2008sums} and propose new bases. We introduce the $SOS_s^{\theta}$ and $SOS^\mathbf{Q}_p$ bases, see \eqref{eqn_monBasisInclusionChain}, and provide  several  theoretical results related to these bases, see~\Cref{lemma_unitConstraintsCoeff,lemma_cardXGammaUpperBounds}. Clearly, the strength of the SOS-SDP based  relaxations and the required time to compute them depend on the chosen monomial basis.
The  SOS-SDP relaxation  for the MAX-SAT is denoted by~\ref{eqn_SumSquaresExplicit}.
 We consider MAX-SAT resolution in~\Cref{section_Resolution} and show  that resolution might not be beneficial for the SOS approach applied to the MAX-SAT.

In \Cref{section_dualRelation}, we elegantly show a connection between the SOS approach to the MAX-SAT and the family of semidefinite feasibility problems~\ref{eqn_RFprogramme}. This is done by deriving the dual problem to~\ref{eqn_SumSquaresExplicit}, see~\Cref{thm_dualEquivalence}. 
In \Cref{section_PRSMsection}, we propose PRSM for solving \ref{eqn_SumSquaresExplicit}. We show that PRSM is well suited for exploiting the structure of \ref{eqn_SumSquaresExplicit}, in particular, the unit constraints, see \eqref{eqn_unitCons}. We thus provide an affirmative answer to the key question posed by \citeauthor{van2008sums}~\cite{van2008sums}: \textit{``Whether SDP software can be developed dealing with unit constraints efficiently?"}. 

We extend the SOS approach for the MAX-SAT  to the weighted partial MAX-SAT in \Cref{section_wpMS}. Here, the variables are restricted to satisfy a set of hard clauses. We show that such hard clauses can be incorporated in the SOS programme \ref{eqn_SumSquaresExplicit} by adding scalar variables. We show in \Cref{section_pMaxSatSolutionMethod} that the resulting programme \eqref{eqn_wpMS_SDPformulation} is also well suited for the PRSM.

In \Cref{section_algorithmDescr}, we provide implementation details of our SOS-SDP based MAX-SAT solver, whose  pseudocode is given  in \mbox{\Cref{alg_SosSdpMaxSat}}. \SOSms{} is a  B\&B algorithm and has two crucial components. 
The first one is the use of warm starts to programme \ref{eqn_SumSquaresExplicit}, in order to quickly obtain strong  bounds. The second one is its ability to quickly parse 
 \ref{eqn_SumSquaresExplicit}, as outlined in \Cref{section_Parsing}.
Our algorithm parses a basis that contains 4,439,449 monomials in (approximately) one second (!).

In \Cref{section_NumericalResults} we provide extensive numerical results that verify efficiency of our exact solver \SOSms{} and
quality of SOS upper bounds. We show that \SOSms{} can solve a variety of MAX-SAT instances in reasonable time, while solving some instances faster than the best solvers in the MSE-2016. We show that the $SOS^\mathbf{Q}_p$ bases \eqref{eqn_sosPQBasis} are able to prove optimality of some MAX-SAT instances, and that the parameter $\mathbf{Q}$ provides the option to adjust the trade-off between quality of the bounds and computation time. We also test our B\&B algorithm for (weighted) partial MAX-SAT instances in \Cref{subsection_wpM2S,subsection_wM3S}. Our solver is able to solve many (weighted) partial MAX-SAT instances in a reasonable time.

This paper has demonstrated the strong performance of \SOSms{} on (weighted partial) MAX-SAT instances from the MSE random track. In the future, we hope to also solve instances with \SOSms{} from the so-called \textit{industrial} and \textit{crafted} tracks. These tracks currently impose two challenges on \SOSms{}. Firstly, these instances induce prohibitively large $SOS_p$ bases, which hinders the computation of strong bounds. To solve this, we require a more sophisticated method for choosing a smaller, manageable, basis, like $SOS_s^\theta$. Secondly, these instances can possess clauses of length $k$, where $k \geq 4$. This is problematic in the current settings, since $F_\Logical$, see \eqref{eqn_FbFunctionDefinition}, is a $k$th degree polynomial, which requires a large basis to be represented. One possible way to overcome these challenges is through exploiting the structure present in these instances. For example, function $F_\Logical$ might have few nonzero coefficients, which allows for finding $SOS$ decompositions with small monomial bases, see also \cite{ahmadi2017improving}.

\begin{landscape}
\begin{table}
    \centering
\begin{tabular}{llllllllllllll}
    \toprule
     & & \multicolumn{2}{c}{$SOS_p$}     & \multicolumn{2}{c}{$SOS^{40}_p$}  &   \multicolumn{2}{c}{$SOS^{50}_p$}    &   \multicolumn{2}{c}{$SOS^{60}_p$}    &   \multicolumn{2}{c}{$SOS^{70}_p$} &   \multicolumn{2}{c}{$SOS^{110}_p$}  \\
     \cmidrule(lr){3-4} \cmidrule(lr){5-6} \cmidrule(l){7-8} \cmidrule(l){9-10} \cmidrule(l){11-12} \cmidrule(l){13-14}
\textbf{Instance}     & \texttt{LB} & \texttt{UB} & $\! \! \! \! \!  \textbf{Iter.}$ & \texttt{UB} & $ \! \! \! \! \! \! \textbf{Iter.}$ & \texttt{UB} & $\! \! \! \! \!  \textbf{Iter.}$ & \texttt{UB} & $\! \! \! \! \! \textbf{Iter.}$ & \texttt{UB} & $\!  \! \! \! \! \textbf{Iter.}$  
& \texttt{UB} & $\!  \! \! \! \! \textbf{Iter.}$
\\ \hline
\texttt{s3v70c800-1}    & 769         & 771.29      & 2   & 770.79      & 2   & 770.67      & 2 & 769.99      & 3.9 &             &    && \\
\texttt{s3v70c800-3}    & 770         & 770.996     & 3   &             &     &             &     &             &     &             &    && \\
\texttt{s3v70c800-4}    & 772         & 772.99      & 2.8 &             &     &             &     &             &     &             &     &&\\
\texttt{s3v70c900-4}    & 861         & 863.11      & 2   & 862.63      & 2   & 861.99      & 3.3 &             &     &             &     &&\\
\texttt{s3v70c1000-1}   & 953         & 954.89      & 2   & 954.70       & 2   & 954.65      & 2   & 953.999     & 3.5 &             &    && \\
\texttt{s3v70c1000-2}   & 957         & 957.99      & 2.7 &             &     &             &     &             &     &             &     &&\\
\texttt{s3v70c1000-5}   & 958         & 958.996     & 2.2 &             &     &             &     &             &     &             &     &&\\
\texttt{s3v70c1100-4}   & 1048        & 1049.03     & 4   & 1048.997    & 3.2 &             &     &             &     &             &     &&\\
\texttt{s3v70c1500-2}   & 1411        & 1411.998    & 2.8 &             &     &             &     &             &     &             &     &&\\
\texttt{s3v90c900-5}    & 875         & 877.44      & 2   & 875.99      & 2.9 &             &     &             &     &             &     &&\\
\texttt{s3v90c900-7}    & 873         & 877.16      & 2   & 875.56      & 2   & 875.05      & 2   & 874.62      & 2   & 873.995     & 2.7 &&\\
\texttt{s3v110c1000-7}  & 969         & 984.01      & 2   & 980.93      & 2   & 979.77      & 2   & 978.63      & 2   & 977.61      & 2   & 974.49
 & 2\\
\texttt{s3v110c1100-10} & 1064        & 1076.81     & 2   & 1074.15     & 2   & 1073.03     & 2   & 1071.98     & 2   & 1071.01     & 2   & 1068.32 & 2 \\
    \cmidrule[\heavyrulewidth]{1-14}
\end{tabular}
\caption{Comparison of the MAX-3-SAT bounds attained by different monomial bases.}
    \label{table_basisBoundsComparison}
\end{table}
\end{landscape}

{
    \centering
\begin{tabular}{lllllll}
\toprule
\textbf{Instance}        & $\texttt{LB}$ & $\texttt{UB}$ & \textbf{T. (s)} & $\mathbf{Q}$ & $| \mathbf{x}|$  & $\textbf{Iter.}$ \\ \hline
$\texttt{s3v70c800-1}$     & 769           & 769.99        & 243.0         & 60           & 2181 & 3.9              \\
$\texttt{s3v70c800-3}$     & 770           & 770.996       & 85.0          & 0            & 1603 & 3.0              \\
$\texttt{s3v70c800-4}$ & 772           & 772.99        & 80.7          & 0            & 1588 & 2.8              \\
$\texttt{s3v70c900-4}$     & 861           & 861.99        & 168.4         & 50           & 2022 & 3.3              \\
$\texttt{s3v70c1000-1}$    & 953           & 953.999       & 236.6         & 60           & 2244 & 3.5              \\
$\texttt{s3v70c1000-2}$    & 957           & 957.99        & 102.4         & 0            & 1810 & 2.7              \\
$\texttt{s3v70c1000-5}$    & 958           & 958.996       & 83.6          & 0            & 1798 & 2.2              \\
$\texttt{s3v70c1100-4}$    & 1048          & 1048.997      & 164.4         & 40           & 2014 & 3.2              \\
$\texttt{s3v70c1500-2}$    & 1411          & 1411.998      & 165.7         & 0            & 2130 & 2.8              \\
$\texttt{s3v90c900-5}$     & 875           & 875.99        & 218.4         & 40           & 2366 & 2.9              \\
$\texttt{s3v90c900-7}$     & 873           & 873.995       & 519.9         & 70           & 3185 & 2.7              \\
$\texttt{s3v110c1000-7}$   & 969           & 974.49        & 3089.3        & 110           & 6106 & 2.0              \\
$\texttt{s3v110c1100-10}$  & 1064          & 1068.32       & 3232.5
        & 110           & 6106 & 2.0             \\
\cmidrule[\heavyrulewidth]{1-7}
\end{tabular}
\captionof{table}{Best upper bounds for the MAX-3-SAT per instance}
    \label{table_bestUBperInstance}
    }
    
\vspace{10pt}

\begin{center}
\captionsetup{justification=centering}
    \captionof{figure}{\SOSms{} on 70 variable MAX-3-SAT (basis $SOS_p^{50}$)}
    \label{picture_cactusPlotS3V70}
\begin{tikzpicture}[trim axis left, trim axis right]
\begin{axis}[
    ylabel style={rotate=-90},
    xlabel={Instances solved},
    ylabel={Time (s)},
    xmin=0, xmax=45,
    ymin=0, ymax=1500,
    xtick={0,15,30,45},
    ytick={0,500,1000,1500},
    legend pos=north west,
    ymajorgrids=true,
    grid style=dashed,
]

\addplot[
    color=blue,
    mark=square,
    ]
    coordinates{
    (1.00,4.61)(2.00,5.00)(3.00,7.06)(4.00,13.83)(5.00,14.14)(6.00,15.14)(7.00,16.76)(8.00,26.16)(9.00,27.20)(10.00,47.84)(11.00,53.09)(12.00,54.46)(13.00,56.66)(14.00,59.97)(15.00,67.52)(16.00,78.94)(17.00,81.33)(18.00,87.16)(19.00,122.22)(20.00,123.08)(21.00,136.06)(22.00,169.47)(23.00,184.75)(24.00,227.34)(25.00,284.52)(26.00,316.19)(27.00,322.68)(28.00,331.47)(29.00,544.98)(30.00,608.15)(31.00,650.70)(32.00,705.64)(33.00,708.07)(34.00,752.93)(35.00,893.82)(36.00,946.71)(37.00,990.89)(38.00,1011.74)(39.00,1062.63)(40.00,1139.98)(41.00,1182.99)(42.00,1309.01)
    };

\addplot[
    color=red,
    mark=o,
    ]
    coordinates {
    (1.00,100.70)(2.00,118.23)(3.00,122.33)(4.00,122.36)(5.00,125.14)(6.00,129.03)(7.00,129.06)(8.00,132.43)(9.00,143.14)(10.00,145.19)(11.00,145.49)(12.00,153.01)(13.00,161.94)(14.00,176.46)(15.00,179.74)(16.00,184.05)(17.00,191.68)(18.00,193.14)(19.00,219.21)(20.00,230.81)(21.00,235.61)(22.00,244.50)(23.00,249.01)(24.00,270.72)(25.00,278.93)(26.00,280.31)(27.00,281.03)(28.00,298.63)(29.00,299.69)(30.00,312.57)(31.00,321.44)(32.00,367.58)(33.00,378.84)(34.00,381.82)(35.00,433.93)(36.00,466.85)(37.00,473.66)(38.00,520.31)(39.00,531.20)(40.00,541.27)(41.00,642.74)(42.00,645.56)(43.00,739.14)(44.00,756.47)(45.00,854.36)
    };
    
    \addlegendentry{Best MSE-2016}
    \addlegendentry{\SOSms{}}
\end{axis}
\end{tikzpicture}
\end{center}

\vspace{10pt}

\begin{center}
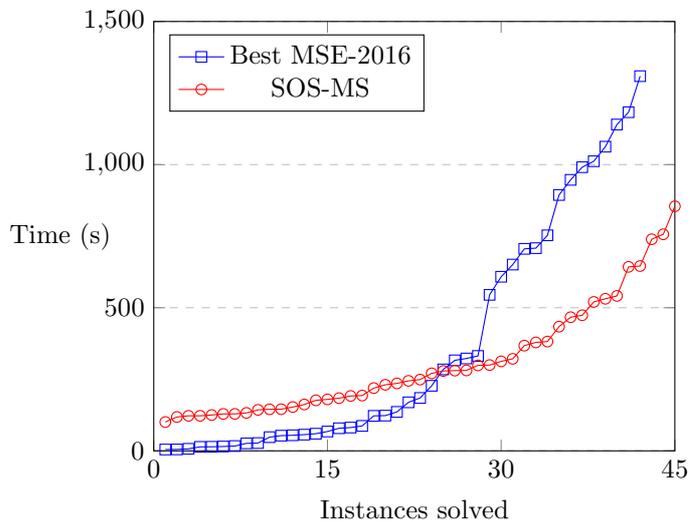

\captionsetup{justification=centering}
    \captionof{figure}{\SOSms{} on 80 variable MAX-3-SAT (basis $SOS_p$)}
    \label{picture_cactusPlotS3V80}
\begin{tikzpicture}[trim axis left, trim axis right]
\begin{axis}[
    ylabel style={rotate=-90},
    xlabel={Instances solved},
    ylabel={Time (s)},
    xmin=0, xmax=45,
    ymin=0, ymax=1800,
    xtick={0,15,30,45},
    ytick={0,500,1000,1500},
    legend pos=north west,
    ymajorgrids=true,
    grid style=dashed,
]

\addplot[
    color=blue,
    mark=square,
    ]
    coordinates{
    (1.00,1.52)(2.00,4.95)(3.00,5.16)(4.00,6.98)(5.00,13.29)(6.00,18.21)(7.00,25.93)(8.00,35.73)(9.00,35.94)(10.00,39.26)(11.00,43.04)(12.00,44.68)(13.00,67.33)(14.00,67.35)(15.00,81.99)(16.00,93.09)(17.00,99.57)(18.00,151.73)(19.00,168.22)(20.00,181.21)(21.00,197.40)(22.00,264.34)(23.00,273.86)(24.00,333.18)(25.00,344.01)(26.00,355.26)(27.00,425.80)(28.00,483.51)(29.00,486.97)(30.00,522.42)(31.00,668.63)(32.00,790.99)(33.00,806.55)(34.00,868.28)(35.00,937.81)(36.00,992.42)(37.00,1044.26)(38.00,1131.22)(39.00,1314.20)(40.00,1393.08)
    };

\addplot[
    color=red,
    mark=o,
    ]
    coordinates {
    (1.00,144.89)(2.00,201.26)(3.00,279.54)(4.00,298.67)(5.00,305.24)(6.00,327.52)(7.00,330.12)(8.00,366.26)(9.00,388.10)(10.00,402.94)(11.00,405.84)(12.00,423.89)(13.00,428.29)(14.00,441.66)(15.00,443.53)(16.00,489.35)(17.00,542.94)(18.00,557.18)(19.00,585.08)(20.00,587.51)(21.00,592.83)(22.00,614.47)(23.00,638.47)(24.00,641.53)(25.00,646.95)(26.00,695.00)(27.00,730.48)(28.00,799.49)(29.00,872.60)(30.00,923.21)(31.00,987.23)(32.00,1030.56)(33.00,1123.17)(34.00,1174.15)(35.00,1307.67)(36.00,1316.13)(37.00,1433.79)(38.00,1559.63)(39.00,1624.88)(40.00,1660.14)(41.00,1665.65)(42.00,1716.55)(43.00,1729.55)
    };
    
    \addlegendentry{CCLS2akms}
    \addlegendentry{\SOSms{}}
\end{axis}
\end{tikzpicture}
\end{center}

\begin{table}[h]
\centering
\begin{tabular}{clllllllllll}
\hline
\multicolumn{1}{l}{} &  & \multicolumn{10}{c}{\textbf{Instance}} \\
\multicolumn{1}{l}{} &  & 0 & 1 & 2 & 3 & 4 & 5 & 6 & 7 & 8 & 9 \\ \hline
\multirow{6}{*}{$m$} & \multicolumn{1}{l|}{2500} & 59 & 177 & 237 & 536 & 75 & 45 & 400 & 361 & 320 & 61 \\
 & \multicolumn{1}{l|}{3000} & 327 & 329 & 96 & 244 & 103 & 554 & 675 & 86 & 339 & 221 \\
 & \multicolumn{1}{l|}{3500} & 322 & 223 & 112 & 134 & 188 & 620 & 23 & 469 & 144 & 252 \\
 & \multicolumn{1}{l|}{4000} & 85 & 5 & 107 & 127 & 347 & 662 & 762 & 320 & - & 289 \\
 & \multicolumn{1}{l|}{4500} & 679 & 334 & 592 & 251 & 732 & 470 & 145 & 223 & 318 & 135 \\
 & \multicolumn{1}{l|}{5000} & 159 & 116 & 663 & 1258 & 975 & 226 & 473 & 598 & 200 & 105 \\ \hline
\end{tabular}
\caption{Unweighted partial MAX-2-SAT running times (seconds)}
\label{table_pM2S_runTimes}
\end{table}

\begin{table}[]
\centering
\begin{tabular}{clllllllllll}
\hline
\multicolumn{1}{l}{} &                         & \multicolumn{10}{c}{\textbf{Instance}}                              \\
  \multicolumn{1}{l}{} &                         & 0    & 1     & 2    & 3    & 4    & 5    & 6    & 7    & 8    & 9   \\ \hline
\multirow{9}{*}{$m$} & \multicolumn{1}{l|}{1000} & 116  & 164   & 61   & 40   & 54   & 85   & 55   & 196  & 344  & 68  \\
                     & \multicolumn{1}{l|}{1500} & 608  & 144   & 447  & 164  & 529  & 190  & 105  & 269  & 349  & 370 \\
                     & \multicolumn{1}{l|}{2000} & 325  & 495   & 326  & 222  & 134  & 233  & 124  & 156  & 178  & 631 \\
                     & \multicolumn{1}{l|}{2500} & 103  & 544   & 282  & 1029 & 318  & 315  & 575  & 926  & 619  & 118 \\
                     & \multicolumn{1}{l|}{3000} & -    & 1341  & 446  & -    & 249  & -    & 1220 & 422  & 1624 & 618 \\
                     & \multicolumn{1}{l|}{3500} & 1667 & 1195  & 1022 & 450  & 1327 & 1351 & 130  & 771  & 196  & 229 \\
                     & \multicolumn{1}{l|}{4000} & 91   & 5     & 208  & 1108 & 930  & -    & -    & 1048 & -    & 338 \\
                     & \multicolumn{1}{l|}{4500} & -    & -     & 1601 & 1220 & -    & 732  & 518  & 1347 & 799  & 965 \\
                     & \multicolumn{1}{l|}{5000} & 338  & 980   & -    & 686  & -    & -    & -    & -    & 222  & 294 \\ \hline
\end{tabular}
\caption{Weighted partial MAX-2-SAT running times (seconds)}
\label{table_wpM2S_runTimes}
\end{table}

\begin{table}[]
\centering
\begin{tabular}{clll}
\hline
\multicolumn{1}{l}{} &  & \multicolumn{2}{c}{\textbf{Running time   (s)}} \\ \cline{3-4} 
\multicolumn{1}{l}{$m$} & \textbf{Inst.} & \textbf{\SOSms{}} & \textbf{CCLS2akms} \\ \hline
\multirow{5}{*}{1400} & 1 & 438.71 & 322.62 \\
 & 2 & 795.71 & 278.78 \\
 & 3 & 696.16 & 752.95 \\
 & 4 & 852.51 & 592.82 \\
 & 5 & 1250.02 & 1054.16 \\ \hline
\multirow{5}{*}{1500} & 1 & 499.26 & 885.66 \\
 & 2 & 653.51 & 1011.90 \\
 & 3 & 399.28 & 292.60 \\
 & 4 & 791.08 & 748.03 \\
 & 5 & $>1800$ & 1323.78 \\ \hline
\end{tabular}
\caption{Weighted 70 variable MAX-3-SAT}
\label{table_wM3S}
\end{table}

\begin{table}[]
\centering
\begin{tabular}{cclllllllllc}
\hline
\multicolumn{1}{l}{} & \multicolumn{1}{l}{} &  &  &  & \multicolumn{6}{c}{\texttt{GAP} at   $\mathbf{X}$ minutes} &  \\ \cline{6-11}
\multicolumn{1}{l}{\textbf{Inst.}} & \multicolumn{1}{l}{\texttt{LB}} & $\mathbf{Q}$ & $|\mathbf{x}|$ & $|C^{\textbf{H}}|$ & 5 & 10 & 15 & 20 & 25 & 30 & $\lambda_\text{min} \cdot | \mathbf{x}| (\times 10^{-2})$ \\ \hline
\multirow{5}{*}{0} & \multirow{5}{*}{492} & 70 & 3288 & 5188 & 4.65 & 3.55 & 3.16 & 2.95 & 2.82 & 2.72 & 4.01 \\
 &  & 75 & 3507 & 6124 & 4.81 & 3.51 & 3.07 & 2.82 & 2.67 & 2.56 & 5.82 \\
 &  & 80 & 3754 & 7140 & 5.35 & 3.61 & 3.03 & 2.72 & 2.53 & 2.40 & 9.64 \\
 &  & 85 & 4035 & 7908 & 6.31 & 4.08 & 3.20 & 2.78 & 2.52 & 2.34 & 18.54 \\
 &  & 90 & 4340 & 8629 & 7.54 & 4.76 & 3.70 & 3.14 & 2.78 & 2.53 & 41.27 \\ \hline
\multirow{5}{*}{1} & \multirow{5}{*}{492} & 70 & 3334 & 4412 & 4.37 & 3.48 & 3.16 & 2.97 & 2.84 & 2.75 & 2.80 \\
 &  & 75 & 3548 & 5509 & 5.03 & 3.58 & 3.13 & 2.88 & 2.72 & 2.60 & 4.49 \\
 &  & 80 & 3780 & 6331 & 5.56 & 3.66 & 3.11 & 2.81 & 2.63 & 2.50 & 6.95 \\
 &  & 85 & 4054 & 7377 & 5.84 & 3.60 & 2.97 & 2.65 & 2.44 & 2.30 & 9.41 \\
 &  & 90 & 4352 & 8329 & 7.77 & 4.44 & 3.46 & 2.97 & 2.67 & 2.46 & 25.62 \\ \hline
\multirow{5}{*}{2} & \multirow{5}{*}{492} & 70 & 3324 & 5067 & 4.09 & 2.99 & 2.58 & 2.35 & 2.20 & 2.10 & 5.29 \\
 &  & 75 & 3536 & 5993 & 4.47 & 3.05 & 2.56 & 2.29 & 2.11 & 1.98 & 9.78 \\
 &  & 80 & 3776 & 7076 & 4.92 & 3.18 & 2.60 & 2.27 & 2.06 & 1.91 & 15.80 \\
 &  & 85 & 4046 & 7839 & 5.05 & 3.25 & 2.51 & 2.16 & 1.93 & 1.77 & 23.72 \\
 &  & 90 & 4346 & 8729 & 7.02 & 4.12 & 3.18 & 2.67 & 2.34 & 2.10 & 52.53 \\ \hline
\multirow{5}{*}{3} & \multirow{5}{*}{491} & 70 & 3296 & 4964 & 5.33 & 4.42 & 4.04 & 3.82 & 3.68 & 3.58 & 1.58 \\
 &  & 75 & 3528 & 5741 & 5.33 & 4.26 & 3.84 & 3.60 & 3.44 & 3.33 & 2.24 \\
 &  & 80 & 3773 & 6501 & 5.79 & 4.26 & 3.76 & 3.49 & 3.31 & 3.18 & 3.47 \\
 &  & 85 & 4038 & 7475 & 5.94 & 4.15 & 3.62 & 3.33 & 3.14 & 3.00 & 4.56 \\
 &  & 90 & 4343 & 8345 & 8.18 & 4.92 & 4.02 & 3.58 & 3.31 & 3.12 & 10.29 \\ \hline
\multirow{5}{*}{4} & \multirow{5}{*}{493} & 70 & 3304 & 4832 & 4.10 & 3.14 & 2.74 & 2.52 & 2.38 & 2.28 & 3.24 \\
 &  & 75 & 3529 & 5815 & 4.54 & 3.17 & 2.70 & 2.44 & 2.26 & 2.13 & 5.41 \\
 &  & 80 & 3775 & 6811 & 4.62 & 3.06 & 2.55 & 2.27 & 2.08 & 1.95 & 7.06 \\
 &  & 85 & 4052 & 7641 & 5.23 & 3.17 & 2.55 & 2.21 & 1.99 & 1.84 & 11.19 \\
 &  & 90 & 4352 & 8695 & 7.53 & 4.16 & 3.11 & 2.60 & 2.28 & 2.05 & 25.95 \\ \hline
\multirow{5}{*}{5} & \multirow{5}{*}{492} & 70 & 3287 & 5207 & 3.96 & 3.15 & 2.82 & 2.64 & 2.52 & 2.44 & 1.80 \\
 &  & 75 & 3506 & 6071 & 4.06 & 3.15 & 2.74 & 2.52 & 2.37 & 2.27 & 2.82 \\
 &  & 80 & 3760 & 6851 & 4.40 & 3.11 & 2.66 & 2.40 & 2.24 & 2.12 & 4.29 \\
 &  & 85 & 4037 & 7927 & 4.99 & 3.29 & 2.76 & 2.46 & 2.26 & 2.12 & 7.77 \\
 &  & 90 & 4347 & 8648 & 6.70 & 4.05 & 3.25 & 2.83 & 2.54 & 2.34 & 22.01 \\ \hline
\multirow{5}{*}{6} & \multirow{5}{*}{493} & 70 & 3298 & 4633 & 3.70 & 2.84 & 2.49 & 2.29 & 2.16 & 2.07 & 3.60 \\
 &  & 75 & 3508 & 5745 & 3.94 & 2.82 & 2.42 & 2.19 & 2.03 & 1.93 & 5.45 \\
 &  & 80 & 3744 & 6659 & 4.27 & 2.82 & 2.38 & 2.13 & 1.97 & 1.84 & 9.95 \\
 &  & 85 & 4018 & 7643 & 4.87 & 3.11 & 2.38 & 2.06 & 1.86 & 1.71 & 15.91 \\
 &  & 90 & 4332 & 8523 & 7.29 & 3.97 & 3.03 & 2.55 & 2.25 & 2.02 & 40.32 \\ \hline
\multirow{5}{*}{7} & \multirow{5}{*}{492} & 70 & 3345 & 4980 & 5.38 & 4.44 & 4.05 & 3.83 & 3.68 & 3.58 & 1.26 \\
 &  & 75 & 3562 & 6117 & 5.48 & 4.42 & 3.97 & 3.72 & 3.55 & 3.43 & 1.75 \\
 &  & 80 & 3803 & 7056 & 5.89 & 4.42 & 3.92 & 3.62 & 3.43 & 3.30 & 2.50 \\
 &  & 85 & 4062 & 7901 & 5.88 & 4.25 & 3.73 & 3.43 & 3.23 & 3.09 & 3.02 \\
 &  & 90 & 4359 & 8709 & 7.50 & 4.68 & 3.93 & 3.56 & 3.31 & 3.12 & 8.27 \\ \hline
\multirow{5}{*}{8} & \multirow{5}{*}{495} & 70 & 3258 & 5155 & 3.83 & 2.65 & 2.23 & 1.99 & 1.84 & 1.72 & 9.72 \\
 &  & 75 & 3484 & 5803 & 4.09 & 2.65 & 2.17 & 1.90 & 1.72 & 1.59 & 15.76 \\
 &  & 80 & 3735 & 6722 & 4.85 & 3.02 & 2.29 & 1.96 & 1.75 & 1.59 & 27.59 \\
 &  & 85 & 4008 & 7547 & 5.11 & 3.08 & 2.37 & 1.90 & 1.66 & 1.50 & 39.49 \\
 &  & 90 & 4324 & 8770 & 7.55 & 4.18 & 3.09 & 2.52 & 2.17 & 1.92 & 76.63 \\ \hline
\multirow{5}{*}{9} & \multirow{5}{*}{491} & 70 & 3341 & 5206 & 5.33 & 4.39 & 4.00 & 3.78 & 3.63 & 3.53 & 1.93 \\
 &  & 75 & 3569 & 5926 & 5.43 & 4.31 & 3.85 & 3.59 & 3.42 & 3.31 & 2.59 \\
 &  & 80 & 3811 & 7158 & 5.78 & 4.27 & 3.75 & 3.46 & 3.27 & 3.14 & 3.60 \\
 &  & 85 & 4078 & 8254 & 5.92 & 4.17 & 3.63 & 3.32 & 3.11 & 2.97 & 4.55 \\
 &  & 90 & 4367 & 9169 & 7.57 & 4.80 & 4.00 & 3.56 & 3.28 & 3.08 & 11.35 \\ \hline
\end{tabular}
\caption{Bounds for partial MAX-3-SAT instances}
\label{table_pM3S}
\end{table}

\bibliographystyle{abbrvnat}
\bibliography{myReferences}

\appendix
\section{PRSM implementation detail}
\label{subsection_compactPRSM}
For the PRSM scheme \eqref{eqn_PRSMiterates}, one only requires access to the matrix $\frac{1}{\beta}S^k$, instead of $S^k$. We therefore propose the following simple adaption to the PRSM scheme, which maintains the loop invariant $Q^k = \frac{1}{\beta}S^k$.
\begin{equation}
\label{eqn_compactPRSM}
\left\lbrace\,
\begin{array}{@{}r@{\quad}l@{}l@{}}
&Z^{k+1} &=  \mathcal{P}_{\mathcal{S}_+} \left (M^k + Q^k \right ),\\
    &Q^{k+1/2} \, &= Q^k + \gamma_1 (M^k - Z^{k+1}),\\
    &M^{k+1} &= \mathcal{P}_{\mathcal{M}_\Logical} \left ( Z^{k+1} - \frac{1}{\beta} I -  Q^{k+1/2}  \right ), \\
    &Q^{k+1} &= Q^{k+1/2} + \gamma_2 (M^{k+1} - Z^{k+1}).
\end{array}
\right.
\end{equation}
Compared to \eqref{eqn_PRSMiterates}, the above scheme does not require the computation of $\frac{1}{\beta}S^k$ twice per iteration.

\section{Runtimes per MAX-3-SAT instance}
\label{appendix_runtimesPerInstance}
We provide the running times of \SOSms{} and the best performing MAX-SAT algorithms from MSE-2016 per tested  instance. \Cref{table_70varM3SAT} reports running times per tested 70 variable MAX-3-SAT instance, corresponding to \Cref{picture_cactusPlotS3V70}. Column $m$ indicates the number of clauses. Column \textbf{Inst.}~reports the name of the instance (e.g., $m = 900$ and \textbf{Inst.}$=3$ corresponds to \texttt{s3v70c900-3.cnf}). The original runtimes of \SOSms{} are here multiplied by a factor 1.4, to account for the different machine compared to those in MSE-2016. A `-'\ value in \Cref{table_70varM3SAT} indicates a time-out (of 30 minutes for MSE-2016 and 30/1.4 minutes for \SOSms{}).
The instances \texttt{s3v70c1500-1.cnf}, \texttt{s3v70c1500-4.cnf} and \texttt{s3v70c1500-5.cnf} remained unsolved in the MSE-2016. Using \SOSms{}, we compute that their optimal values are 1410, 1409 and 1406, respectively.

\Cref{table_80varM3SAT_instances} reports running times per tested 80 variable MAX-3-SAT instance, corresponding to \Cref{picture_cactusPlotS3V80}. Here, column $m$ indicates the number of clauses per instance, \textbf{Inst.}~column reports the name of the instance  (e.g., $m = 1000$ and \textbf{Inst.}$=4$ corresponds to \texttt{s3v80c1000-4.cnf}). Columns \textbf{\SOSms{}} and \textbf{CCLS2akms} provide the running times of each solver, in seconds. Here we provide original computational times, thus not scaled ones.

\begin{table}
\parbox{.5\linewidth}{\begin{tabular}{llll}
\hline
\multicolumn{1}{l}{} &  & \multicolumn{2}{c}{\textbf{Running time   (s)}} \\ \cline{3-4} 
$m$                 & \textbf{Inst.} & \textbf{\SOSms{}} & \textbf{MSE } \\ \hline
\multirow{5}{*}{700}  & 1                 & 179.74                   & 5.00                       \\
                    & 2                 & 145.49                   & 7.06                       \\
                    & 3                 & 298.63                   & 14.14                      \\
                    & 4                 & 381.82                   & 13.83                      \\
                    & 5                 & 145.19                   & 4.61                       \\ \hline
\multirow{5}{*}{800}  & 1                 & 281.03                   & 27.20                      \\
                    & 2                 & 378.84                   & 81.33                      \\
                    & 3                 & 125.14                   & 16.76                      \\
                    & 4                 & 129.03                   & 15.14                      \\
                    & 5                 & 118.23                   & 26.16                      \\ \hline
\multirow{5}{*}{900}  & 1                 & 280.31                   & 67.52                      \\
                    & 2                 & 235.61                   & 59.97                      \\
                    & 3                 & 132.43                   & 53.09                      \\
                    & 4                 & 219.21                   & 54.46                      \\
                    & 5                 & 321.44                   & 87.16                      \\ \hline
\multirow{5}{*}{1000} & 1                 & 312.57                   & 123.08                     \\
                    & 2                 & 193.14                   & 56.66                      \\
                    & 3                 & 122.33                   & 78.94                      \\
                    & 4                 & 100.70                   & 122.22                     \\
                    & 5                 & 143.14                   & 47.84                      \\ \hline
\multirow{5}{*}{1100} & 1                 & 531.20                   & 331.47                     \\
                    & 2                 & 244.50                   & 227.34                     \\
                    & 3                 & 122.36                   & 136.06                     \\
                    & 4                 & 230.81                   & 169.47                     \\
                    & 5                 & 176.46                   & 184.75                     \\ \hline
\multirow{5}{*}{1200} & 1                 & 756.47                   & 752.93                     \\
                    & 2                 & 161.94                   & 322.68                     \\
                    & 3                 & 645.56                   & 608.15                     \\
                    & 4                 & 739.14                   & 1011.74                    \\
                    & 5                 & 473.66                   & 544.98                     \\ \hline
\multirow{5}{*}{1300} & 1                 & 520.31                   & 1139.98                    \\
                    & 2                 & 541.27                   & 946.71                     \\
                    & 3                 & 367.58                   & 708.07                     \\
                    & 4                 & 153.01                   & 316.19                     \\
                    & 5                 & 278.93                   & 650.70                     \\ \hline
\multirow{5}{*}{1400} & 1                 & 129.06                   & 705.64                     \\
                    & 2                 & 854.36                   & 1309.01                    \\
                    & 3                 & 249.01                   & 990.89                     \\
                    & 4                 & 184.05                   & 284.52                     \\
                    & 5                 & 299.69                   & 1062.63                    \\ \hline
\multirow{5}{*}{1500} & 1                 & 466.85                   & -                          \\
                    & 2                 & 270.72                   & 1182.99                    \\
                    & 3                 & 191.68                   & 893.82                     \\
                    & 4                 & 642.74                   & -                          \\
                    & 5                 & 433.93                   & -                          \\ \hline
\end{tabular}
\caption{70 variable MAX-3-SAT instances}
\label{table_70varM3SAT}
}
\quad
\parbox{.45\linewidth}{
\begin{tabular}{llll}
\hline
\multicolumn{1}{l}{} &  & \multicolumn{2}{c}{\textbf{Running time   (s)}} \\ \cline{3-4} 
$m$                         & \textbf{Inst.} & \textbf{\SOSms{}} & \textbf{CCLS2akms} \\ \hline
\multirow{6}{*}{700} & 1              & 441.66              & 6.98                   \\
                            & 2              & 638.47              & 13.29                  \\
                            & 3              & 489.35              & 4.95                   \\
                            & 4              & 144.89              & 1.52                   \\
                            & 5              & 799.49              & 18.21                  \\
                            & 6              & 402.94              & 5.16                   \\ \hline
\multirow{6}{*}{800}          & 1              & 923.21              & 44.68                  \\
                            & 2              & -                   & 197.40                 \\
                            & 3              & 1174.15             & 43.04                  \\
                            & 4              & 557.18              & 25.93                  \\
                            & 5              & 388.10              & 35.73                  \\
                            & 6              & 1123.17             & 39.26                  \\ \hline
\multirow{6}{*}{900}          & 1              & 585.08              & 81.99                  \\
                            & 2              & -                   & 181.21                 \\
                            & 3              & 730.48              & 99.57                  \\
                            & 4              & 305.24              & 67.33                  \\
                            & 5              & 423.89              & 67.35                  \\
                            & 6              & 298.67              & 35.94                  \\ \hline
\multirow{6}{*}{1000}         & 1              & 1316.13             & 355.26                 \\
                            & 2              & 1307.67             & 273.86                 \\
                            & 3              & 327.52              & 93.09                  \\
                            & 4              & 1559.63             & 483.51                 \\
                            & 5              & 405.84              & 151.73                 \\
                            & 6              & 366.26              & 168.22                 \\ \hline
\multirow{6}{*}{1100}         & 1              & 1624.88             & 937.81                 \\
                            & 2              & -                   & 1393.08                \\
                            & 3              & 279.54              & 344.01                 \\
                            & 4              & 1716.55             & 806.55                 \\
                            & 5              & 587.51              & 264.34                 \\
                            & 6              & 1660.14             & 425.80                 \\ \hline
\multirow{6}{*}{1200}         & 1              & 542.94              & 790.99                 \\
                            & 2              & 201.26              & 333.18                 \\
                            & 3              & 592.83              & 522.42                 \\
                            & 4              & 641.53              & 486.97                 \\
                            & 5              & 1030.56             & 868.28                 \\
                            & 6              & 1729.55             & 1131.22                \\ \hline
\multirow{6}{*}{1300}         & 1              & 330.12              & 668.63                 \\
                            & 2              & 614.47              & 1314.20                \\
                            & 3              & -                   & -                      \\
                            & 4              & 428.29              & 992.42                 \\
                            & 5              & -                   & -                      \\
                            & 6              & 1665.65             & -                      \\ \hline
\multirow{6}{*}{1400}         & 1              & 443.53              & -                      \\
                            & 2              & 872.60              & -                      \\
                            & 3              & 695.00              & -                      \\
                            & 4              & 646.95              & 1044.26                \\
                            & 5              & 987.23              & -                      \\
                            & 6              & 1433.79             & -           \\
                            \hline
\end{tabular} 
\caption{80 variable MAX-3-SAT instances}
\label{table_80varM3SAT_instances}
}
\end{table}

\section{Search tree for the partial MAX-2-SAT}
\label{subsection_searchTreeWpM2S}
We provide the search trees of
our SOS-SDP based algorithm for solving various partial MAX-2-SAT instances, as described in \Cref{subsection_wpM2S}. These instances are also reported in \Cref{table_pM2S_runTimes,table_wpM2S_runTimes}. 

For the search trees in \Cref{fig_searchTree1,fig_searchTree2,fig_searchTree3},  each node is given a numeric value between zero and one, or the value \textbf{B}. Numeric values indicate the largest value of $\theta$ for which basis $SOS_s^{\theta}$ was used to compute an upper bound in that node. The value \textbf{B} (\textbf{B} for branch) indicates that no upper bound was computed in this node, but instead immediately a variable was chosen to branch.

The figures show the strength of the SDP bounds, implying that many nodes can be pruned immediately. This also demonstrates the effectiveness of the branching rule, which is able to find many nodes that can be pruned immediately.

\section{Branching process for the partial MAX-2-SAT}
\label{appendix_branchingRulesWPM2S}
During the B\&B search for the optimal solution to the partial MAX-2-SAT, we do not compute SOS-SDP based upper bounds at each node. We describe here the process which decides in which nodes the algorithm computes an upper bound.

Recall that our branching rule, described \Cref{subsection_wpM2S}, selects the variable $i \in [n]$ for which either $\texttt{unitRes}(\Logical, i)$ or $\texttt{unitRes}(\Logical, -i)$ contains many hard unit clauses, and therefore many truth assignments. Each branching step creates two child nodes. We refer to the node which corresponds to the proposition with most truth assignments in the two child nodes as a \textit{bad} node. The forced assignments resulting from the hard unit clauses in that proposition are often sub optimal, which explains the name.

The algorithm for the B\&B search operates in two phases, named phase I and phase II. In phase I, we only compute upper bounds for bad nodes. We exit phase I when the algorithm fails to prune a bad node, or when the number of remaining variables is smaller than some fixed value $n_\text{min}$. This process is given in pseudocode in \Cref{alg_wM2S_branchingPhase1}.

After exiting phase I, the algorithm enter phase II, see \Cref{alg_wM2S_branchingPhase2}. In phase II, before we compute an upper bound in a node, we first attempt to remove variables from the proposition,
by pruning bad nodes. This is described in \Crefrange{alg_wM2Sphase2_multi1}{alg_wM2Sphase2_multi5}. The main difference with phase I is that, when we fail to prune a bad node in these lines, we do not recompute a stronger upper bound with a larger monomial basis. The extra effort in phase I is justified since pruning a node in phase I equates to removing one unassigned variable from the rest of the search tree.

After \Crefrange{alg_wM2Sphase2_multi1}{alg_wM2Sphase2_multi5}, we consider the remaining proposition $\Logical$, and compute the basis $SOS_s^1$. If this basis is too large, we branch immediately. Otherwise, we compute an upper bound. We set the parameters as $\theta_{\text{start}} = 0$ in phase I, $\theta_{\text{start}} = 0.1$ in phase II, and $(b_{\text{max}}, \, n_{\text{min}}, \, s_{\text{max}}) = (15,70,1750)$.

\begin{algorithm}[ht!]
\caption{B\&B search for the (weighted) partial MAX-2-SAT, phase I}\label{alg_wM2S_branchingPhase1}
\textbf{Input: }Lower bound \texttt{LB}, proposition $\Logical$, parameters $(\theta_{\text{start}}, n_{\text{min}}) \in [0,0.5) \times \mathbb{N}$.
\\
Set $\theta = \theta_{\text{start}}$. \\
\While{ $n_\Logical \geq n_{\text{min}}$}{
Determine the truth assignment $\sigma$ according to the branching rule from \Cref{subsection_wpM2S}. \\
Compute $\Logical^\prime = \texttt{unitRes}(\Logical, \sigma)$. \\
{\color{blue} \tcc{Proposition $\Logical^\prime$ corresponds to a bad node.}}
Solve $\texttt{P}_{\Logical^\prime}$, using basis $SOS^{\theta}_s$, to obtain $\texttt{UB}$. \\
\uIf{$\lfloor \texttt{UB} \rfloor \leq \texttt{LB}$}{
Update $\Logical := \texttt{unitRes}(\Logical, -\sigma)$, and reset $\theta$ by $\theta := \theta_{\text{start}}$. \\
{\color{blue} \tcc{Note that we did not compute an upper bound for the old  $\Logical$.} }
}
\Else{
    \uIf{$\theta = \theta_{\text{start}}$}{Set $\theta = 0.5$ if $\texttt{UB} - \texttt{LB} < 10$, set $\theta = 1$ otherwise. \\
    {\color{blue} \tcc{Recompute a stronger upper bound with a larger $\theta$.}}
    }
    \Else{
    {\color{blue} \tcc{Stronger upper bound was unable to prune the node.}}
    Compute $\Logical^{\prime \prime} = \texttt{unitRes}(\Logical, -\sigma)$. \\
    Add two nodes corresponding to $\Logical^\prime$ and $\Logical^{\prime \prime}$ to the search tree. \\
    \textbf{break} \\
    {\color{blue} \tcc{Move to phase II of the algorithm (see \Cref{alg_wM2S_branchingPhase2}).}}
    }
}
}
\end{algorithm}

\begin{algorithm}[ht!]
\caption{B\&B search for the (weighted) partial MAX-2-SAT, phase II}\label{alg_wM2S_branchingPhase2}
\textbf{Input: }Lower bound \texttt{LB}, parameters $(\theta_{\text{start}},b_{\text{max}}, \, n_{\text{min}}, \, s_{\text{max}}) \in [0,1) \times \mathbb{N} \times \mathbb{N} \times \mathbb{N}$.
\\
\While{The search tree contains a node which is neither branched nor pruned}{
Consider an unbranched and unpruned node in the search tree, with proposition $\Logical$. \\
Set $b = 0$. \\
\While{$b < b_{\text{max}}$ \& $n_{\Logical} > n_{\text{min}}$ \label{line_multipleBranch}}{
Determine the truth assignment $\sigma$ according to the branching rule from \Cref{subsection_wpM2S}. \label{alg_wM2Sphase2_multi1}\\
Compute $\Logical^\prime = \texttt{unitRes}(\Logical, \sigma)$. \\
{\color{blue} \tcc{Proposition $\Logical^\prime$ corresponds to a bad node.}}
Solve $\texttt{P}_{\Logical^\prime}$, using basis $SOS^{\theta_{\text{start}}}_s$, to obtain $\texttt{UB}$. \\
\uIf{$\lfloor \texttt{UB} \rfloor \leq \texttt{LB}$}{
Update $\Logical := \texttt{unitRes}(\Logical, -\sigma)$, and $b := b+1$. \\
{\color{blue} \tcc{Note that we did not compute an upper bound for the old  $\Logical$.} }
}
\Else{ $\textbf{break}$ \label{alg_wM2Sphase2_multi5}} 
}
Compute monomial basis $SOS_s^1$ for $\Logical$. \\
\uIf{$|SOS_s^1| > s_{\text{max}}$}{
Branch the node corresponding to $\Logical$, add its two children to the search tree and continue with B\&B search. \label{algLine_branchNode_wpM2S} \\
{\color{blue} \tcc{For efficiency reasons, compute upper bounds only when the basis is small enough.} }
}
\Else{Solve \ref{eqn_SumSquaresExplicit}, using basis $SOS^{1}_s$, to obtain \texttt{UB}. \\
\uIf{$\lfloor \texttt{UB} \rfloor \leq \texttt{LB}$}{Prune the node corresponding to $\Logical$, and continue with the B\&B search.}
\Else{Perform \Cref{algLine_branchNode_wpM2S}.}
}
}
\end{algorithm}

\clearpage
\newgeometry{left=0.4cm,bottom=0.1cm,top=0.1cm,right=0.4cm}
\begin{sidewaysfigure}
    \centering
    \includegraphics[width=840pt] {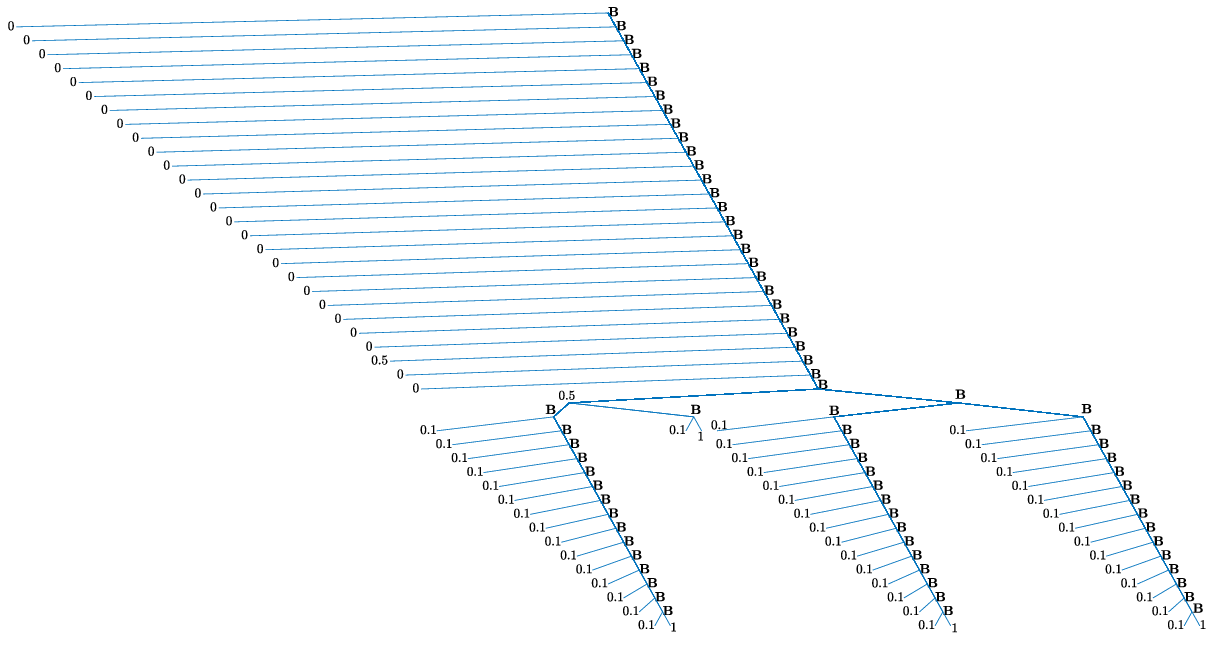}
    \caption{Search tree for \texttt{file\_rpms\_wcnf\_L2\_V150\_C2500\_H150\_4.wcnf}}
    \label{fig_searchTree1}
\end{sidewaysfigure}   

\begin{sidewaysfigure}
    \centering
    \includegraphics[width=830pt]{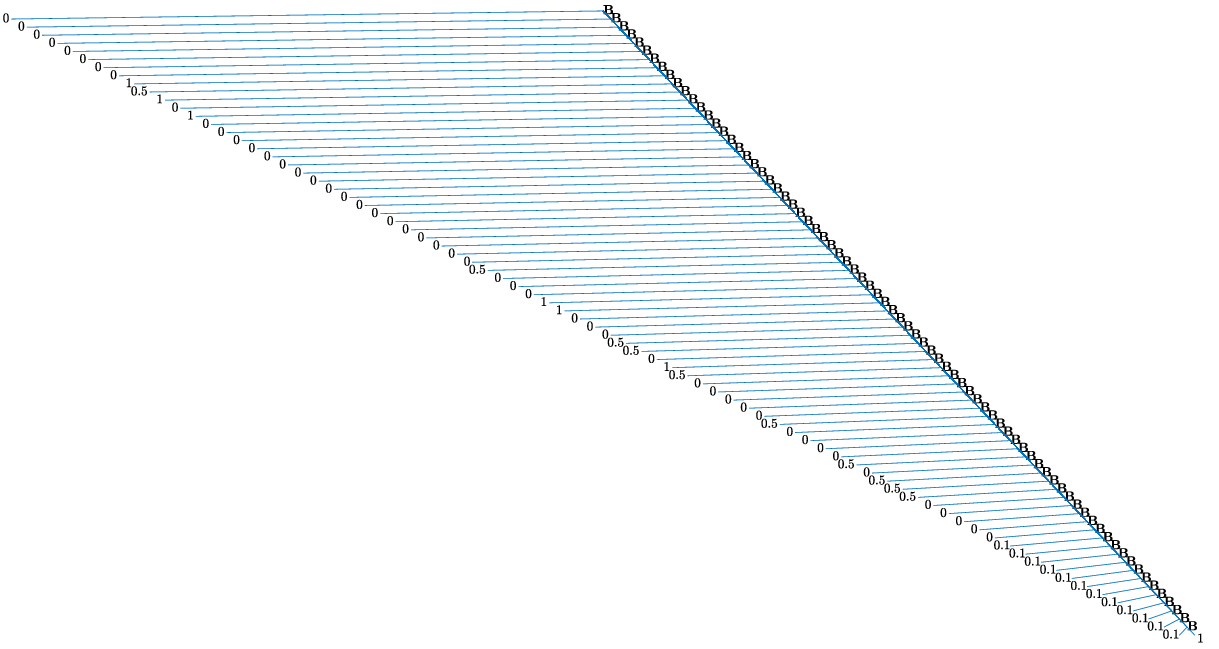}
    \caption{Search tree for \texttt{file\_rwpms\_wcnf\_L2\_V150\_C1000\_H150\_1.wcnf}}
    \label{fig_searchTree2}
\end{sidewaysfigure}

\begin{sidewaysfigure}
    \centering
    \includegraphics[width=830pt]{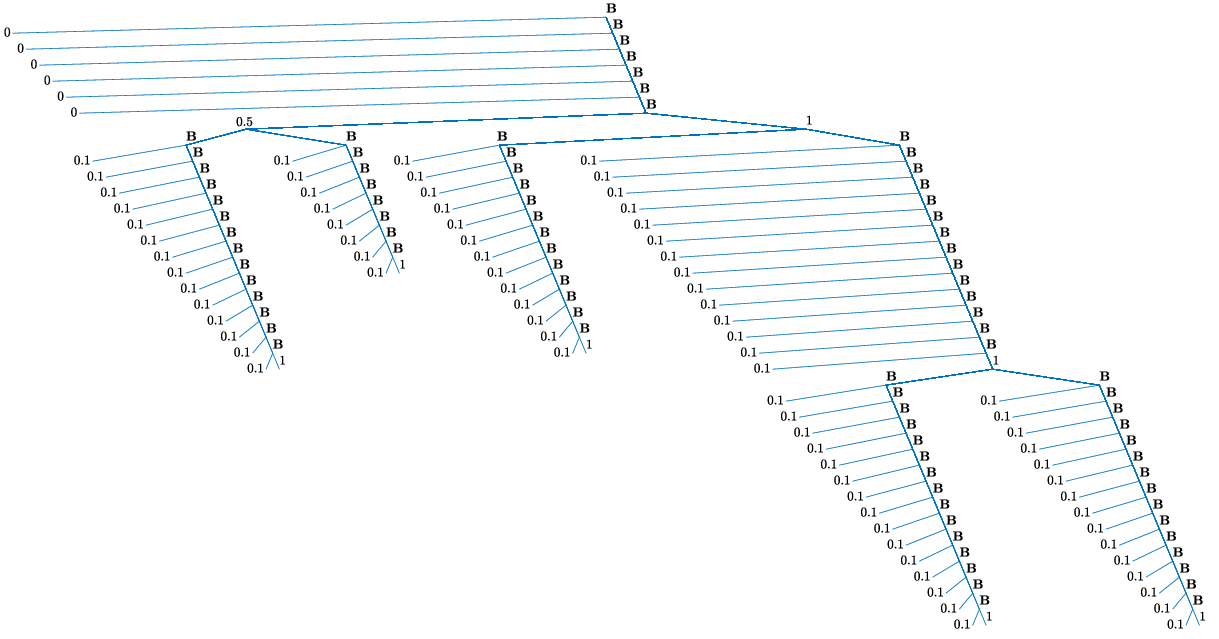}
    \caption{Search tree for \texttt{file\_rpms\_wcnf\_L2\_V150\_C4000\_H150\_3.wcnf}}
    \label{fig_searchTree3}
\end{sidewaysfigure}

\restoregeometry

\end{document}